\documentclass[11pt]{article}

\usepackage[margin =1in]{geometry}
\usepackage{amssymb,amsmath,amsthm}
\usepackage{enumerate}
\usepackage{hyperref}
\usepackage{cite}
\usepackage[capitalize]{cleveref}
\usepackage{enumitem}
\usepackage{color}
\usepackage{cancel}
\usepackage{tikz}
\usetikzlibrary{shapes}
\usepackage{svg}
\usepackage{import}
\usepackage{optidef}

\usepackage[utf8]{inputenc} \usepackage[T1]{fontenc}    \usepackage{hyperref}       \usepackage{url}            \usepackage{booktabs}       \usepackage{amsfonts}       \usepackage{nicefrac}       \usepackage{microtype}
\usepackage{multirow}
\usepackage{xfrac}
\usepackage{todonotes}
\usepackage{titlesec}
\usepackage{caption}
\usepackage{subcaption}
\usepackage{amsmath}

\titleformat{\subsubsection}
{\normalfont\normalsize\bfseries}{\thesubsubsection}{1em}{}

\usepackage{graphicx}

\usepackage{appendix}
\numberwithin{equation}{section}

\newcommand{\tr}{\mathrm{Tr}}

\newcommand{\norm}[1]{\left \| #1 \right \|}

\newcommand{\inclu}[0] {\ar@{^{(}->}}

\newcommand{\R}{\mathbb{R}}

\newcommand{\N}{\mathbb{N}}

\newcommand{\cS}{\mathcal{S}}

\newcommand{\abs}[1]{\left| #1 \right|}

\newcommand{\Diag}{{\rm Diag}}

\newtheorem{thm}{Theorem}[section]
\newtheorem{theorem}{Theorem}[section]
\newtheorem{conjecture}{Conjecture}[section]

\newtheorem{proposition}[thm]{Proposition}
\newtheorem{lemma}[thm]{Lemma}

\crefname{claim}{claim}{claims}
\Crefname{claim}{Claim}{Claims}
\crefname{lem}{lemma}{lemmas}
\Crefname{lem}{Lemma}{Lemmas}
\crefname{algorithm}{algorithm}{algorithms}
\Crefname{algorithm}{Algorithm}{Algorithms}

\newtheorem{remark}{Remark}

\theoremstyle{remark}

\newcommand{\set}[1]{\left\{ #1 \right\}}

\usepackage{mathtools}

\usepackage[algo2e, boxruled]{algorithm2e}

\usepackage[T1]{fontenc}
\usepackage{lmodern}

\newcommand{\lk}{\lambda^{(k)}}
\newcommand{\hk}{\mathfrak{h}^{(k)}}
\newcommand{\by}{\times}
\newcommand{\mb}{\mathbf}

\DeclareMathOperator{\rev}{rev}

\newcounter{mathematicaref}[section]
\newcommand{\nextmathematica}{\stepcounter{mathematicaref}{\color{red}\texttt{Mathematica proof \arabic{section}.\arabic{mathematicaref}}}}

\newcommand{\mfhk}{\mathfrak{h}^{(k)}}
\begin{document}

\title{Accelerated Gradient Descent via Long Steps}

\author{
    Benjamin Grimmer\footnote{Johns Hopkins University, Department of Applied Mathematics and Statistics, \texttt{grimmer@jhu.edu}}
    \and
    Kevin Shu\footnote{Georgia Institute of Technology, School of Mathematics, \texttt{kshu8@gatech.edu}}
    \and
    Alex L.\ Wang\footnote{Purdue University, Daniels School of Business, \texttt{wang5984@purdue.edu}}
}
	\date{\today}
	\maketitle

    \begin{abstract}
        Recently Grimmer~\cite{Grimmer2023-long} showed that for smooth convex optimization, by utilizing longer steps periodically, gradient descent's textbook $LD^2/2T$ convergence guarantees can be improved by constant factors, conjecturing an accelerated rate strictly faster than $O(1/T)$ could be possible. Here we prove such a big-O gain, establishing gradient descent's first accelerated convergence rate in this setting. Namely, we prove a $O(1/T^{1.0564})$ rate for smooth convex minimization by utilizing a nonconstant nonperiodic sequence of increasingly large stepsizes. It remains open if one can achieve the $O(1/T^{1.178})$ rate conjectured by~\cite{gupta2023branch} or the optimal gradient method rate of $O(1/T^2)$. Big-O convergence rate accelerations from long steps follow from our theory for strongly convex optimization, similar to but somewhat weaker than those concurrently developed by Altschuler and Parrilo~\cite{altschuler2023acceleration}.
    \end{abstract} 

    \section{Introduction}
We consider minimizing a continuously differentiable, convex objective function $f\colon\mathbb{R}^n\rightarrow \mathbb{R}$ via gradient descent. We assume that $f$ has an $L$-Lipschitz continuous gradient (i.e., $f$ is $L$-smooth).
Gradient descent proceeds by iterating
\begin{equation}
	x_{i+1} = x_i - \frac{h_i}{L} \nabla f(x_i) 
\end{equation}
with (normalized) stepsizes $h_i>0$ starting from some $x_0\in\R^n$.
We assume a minimizer $x_\star$ of $f$ exists and that $f$'s sublevel sets are bounded with $D = \sup\set{\norm{x-x_\star}_2\mid f(x)\leq f(x_0)}<\infty$.

When utilizing {\it constant stepsizes}, until recently, the best known guarantee was the textbook result~\cite{Bertsekas2015ConvexOA} that fixing $h_i=1$ ensures $f(x_T)-f(x_\star) \leq LD^2/2T$.
This was improved by the tight convergence theory of Teboulle and Vaisbourd~\cite{Teboulle2022}, showing a rate of
$$ f(x_T)-f(x_\star) \leq \frac{LD^2}{4 T}$$
when the stepsizes $h_i = 1$. Utilizing nonconstant stepsizes monotonically converging up to $2$, they further showed a rate approaching $LD^2/8T$.
These coefficient improvements were first conjectured by~\cite{taylor2017interpolation}.

By utilizing {\it nonconstant periodically long stepsizes}, Grimmer~\cite{Grimmer2023-long} showed improved convergence rates are possible outside the classic range of stepsizes $(0,2)$. We refer to steps with $h_i>2$ as long steps since they go beyond the classic regime $h_i\in(0,2)$ where descent on the objective value is guaranteed. Their strongest result, resulting from a computer-aided semidefinite programming proof technique, showed repeating a cycle of $127$ stepsizes $h_0,\dots,h_{126}$ ranging from $1.4$ to $370.0$ gives a rate of
$$ \min_{i\leq T} f(x_i)-f(x_\star) \leq \frac{LD^2}{5.83463 T} +O(1/T^2).$$
Note, bounding $\min_{i\leq T} f(x_i)-f(x_\star)$ (or a similar quantity) is natural for such long step methods as monotone decrease of the objective is no longer ensured.
By considering longer and more complex patterns, increasing gains in the coefficient appear to follow. However, the reliance on numerically solving semidefinite programs with size depending on the pattern length limited this prior work's ability to explore and prove continued improvements in convergence rates. Grimmer conjectured at least a $O(1/T\log(T))$ rate would follow if one could design and analyze (algebraically) cyclic patterns of generic length.

We show greater gains are possible. By using {\it nonconstant, nonperiodic stepsizes} $h_i$, we prove
\begin{equation}\label{eq:main-rate}
	\min_{i\leq T} f(x_i)-f(x_\star)  = \frac{11.7816\times LD^2}{T^{1.0564}}.
\end{equation}
Proving this relies on semidefinite programming-based analysis techniques and considers the overall effect of many iterations at once (rather than the one-step inductions typical to most first-order method analysis).

In related work, Das Gupta et.\ al.~\cite{gupta2023branch} produced numerically globally optimal stepsize selections via a branch-and-bound procedure
for gradient descent with a fixed number of steps $T\in [0,25]$. By fitting to asymptotics of their numerical guarantees~\cite[Figure 2]{gupta2023branch}, they conjectured a $O(1/T^{1.178})$ rate may be possible and may be best possible. Our work leaves open the gap between our $O(1/T^{1.0564})$ rate and their conjecture, as well as the gap between their conjecture and the known lower bound for general gradient methods of $O(1/T^2)$~\cite{nesterov-textbook}.

Generally, studying accelerated convergence rates stemming from long steps can yield several advantages/insights beyond what classic momentum methods can provide. Understanding the acceleration stemming from long steps may yield insights into the fundamental mechanism enabling acceleration; we have shown that changing the update directions based on an auxiliary momentum sequence is not needed to beat $O(1/T)$. Hence, an acceleration can be attained by a method storing only one vector in memory at each step rather than two. Further, using long steps may partially mitigate the effects of inexact or stochastic gradients, known to hamper momentum methods~\cite{Olivier2014}, as no momentum term exists to propagate past errors into future steps.
Lastly, we note that continued work in this direction may yield theoretical support for such cyclic long stepsize patterns used in neural network training~\cite{Smith2015CyclicalLR,Smith2017SuperconvergenceVF}.

Stronger guarantees for gradient descent with variable stepsizes are known in specialized settings, like $\mu$-strongly convex minimization. Classically, gradient descent with constant stepsizes $h_i = 1/L$ produces an $\epsilon$-minimizer in $O(\kappa\log(1/\epsilon))$ iterations, where $\kappa=L/\mu$. 
Concurrent to this work, Altschuler and Parrilo~\cite{altschuler2023acceleration} recently showed an accelerated rate through the inclusion of long steps of $O(\kappa^{0.786434}\log(1/\epsilon))$ (extending their prior preliminary results in~\cite{altschuler2018greed}). Our convergence theory, using a different pattern of long steps, also improves on the classic $O(\kappa\log(1/\epsilon))$, although at a weaker rate. Our Theorem~\ref{thm:strongly-convex} ensures that under our long stepsize selection, gradient descent has a $O(\kappa^{0.94662}\log(1/\epsilon))$ rate. The silver ratio, $1+\sqrt{2}$, occurs prominently in both our analysis and theirs, indicating potential deeper connections.

Altschuler and Parrilo~\cite{altschuler2023acceleration}'s faster accelerated rate for smooth strongly convex problems can be extended to give a $O(1/T^{1.271553})$ guarantee for a {\it modified} gradient descent method for general smooth convex problems, the main focus of this work. Doing so requires running gradient descent on a modified objective function (whose choice depends on specifying a target accuracy and an initial distance to optimal). In contrast, our results show acceleration beyond $O(1/T)$ is possible for gradient descent via long steps alone, i.e., without needing this modification and additional problem knowledge.

In the further specialized case of minimizing strongly convex quadratics, the optimal stepsizes were given by~\cite{young1953}, which attain the optimal $O(\kappa^{1/2}\log(1/\epsilon))$ rate. 
For nonconvex optimization, exact worst-case guarantees for gradient descent with short steps $h_k\in (0,1]$  were given by Abbaszadehpeivasti et al.~\cite{Abbaszadehpeivasti2021TheEW}.
The potential use of longer steps (greater than $2/L$) in nonconvex settings is an interesting future direction but beyond the scope of this work.

In the remainder of this introduction, we define our stepsize selection which accelerates due to the inclusion of long steps. To prove our accelerated rates, Section~\ref{sec:review} first reviews the semidefinite programming analysis technique of~\cite{Grimmer2023-long} based on the performance estimation problem techniques of~\cite{drori2012PerformanceOF,taylor2017smooth}.
Specifically, the proof of our accelerated convergence rates utilize the ``straightforward'' property of stepsize patterns.
Section~\ref{sec:proof} proves the claimed convergence rate assuming that certain finite-length blocks within our nonperiodic stepsize pattern are straightforward.
Section~\ref{sec:certificate_part_1} shows that straightforwardness of a stepsize pattern can be certified by producing a feasible solution to an associated spectral set.
Finally, we close the loop and show that appropriate blocks within our stepsize pattern are straightforward by constructing such certificates in Section~\ref{sec:certificate}. Appendices~\ref{sec:lambdadRowColSum} and~\ref{sec:lambdaM} verify the necessary conditions on our certificates. Several symbolically intense calculations or simplifications are deferred to the associated Mathematica\cite{mathematica} notebook available at the Github repository \url{https://github.com/ootks/GDLongSteps}. This same Github repository also contains Julia code that computes our step size sequences and our associated certificates.

\subsection{The Proposed Nonconstant Nonperiodic Stepsizes}

We build our overall sequence of $h = (h_0,h_1,h_2,\dots )$ from several finite-length patterns of lengths $2^{k+1}-1$ for $k\in\mathbb{N}$.\footnote{Here and throughout, we adopt the convention $\mathbb{N} = \set{0,1,2,\dots}$ contains zero.} We denote the building block pattern associated with each $k$ by some
$ \mathfrak{h}^{(k)} = (\mathfrak{h}^{(k)}_0,\mathfrak{h}^{(k)}_1,\dots \mathfrak{h}^{(k)}_{2^{k+1}-2}) . $
Below, we first define these building blocks and then explain how they are concatenated and rescaled to produce an accelerated stepsize sequence for gradient descent satisfying~\eqref{eq:main-rate}.

\subsubsection{Defining the $\mathfrak{h}^{(k)}$ Building Blocks}

We will define $\mathfrak{h}^{(k)} \in \R^{2^{k+1}-1}$ in terms of three sequences: $\mu_i$, $\alpha_i$ and $\beta_i$.
For all $i\geq 0$, define
\begin{align}\label{eq:beta}
	\beta_i &\coloneqq 1 + (1+\sqrt{2})^{i-1}.
\end{align}
We define $\alpha_i$ and $\mu_i$ inductively.
For $i \geq 0$, define
\begin{align*}
	\mu_i \coloneqq 2\sum_{\ell=0}^{i-1} \alpha_{\ell} + \sum_{\ell=0}^{i-2} 2(2^{i-\ell-1}-1) \beta_{\ell} + 2.
\end{align*}
Here, we say that the empty sum is 0, so that $\mu_0 = 2$, and $\mu_1 = 2\alpha_0 + 2$.
In general, $\mu_i$ only depends on $\alpha_\ell$ for $\ell \leq i - 1$.
We may then inductively define
\begin{align}
	\alpha_i &\coloneqq \text{the unique root larger than $1$ of}\nonumber\\
	&\qquad q_i(x) \coloneqq  2(x-1)^2 + (\mu_i - 1) (x-1) - (\beta_{i+1}-1)(\mu_i-1).\label{eq:alpha}
\end{align}
Note that $q_i(1) = - (\beta_{i+1}-1)(\mu_i-1) < 0$, so that $q_i$ has a unique root larger than $1$ and $\alpha_i$ is well-defined.

For $i \in \mathbb{N}$, we let $\nu(i)$ be the largest $k$ so that $2^k$ divides $i$ with the convention that $\nu(0)=\infty$. This is sometimes also known as the $2$-adic valuation.
For $\ell\geq 0$, let $\pi^{(\ell)} \in \R^{2^{\ell}-1}$ be the vector where $\pi^{(\ell)}_i = \beta_{\nu(i)}$.
We list the first few vectors $\pi^{(\ell)}$ for concreteness:
\begin{align*}
	\pi^{(0)} &= \emptyset \text{ is the empty vector},\\
	\pi^{(1)} &= [\beta_0],\\
	\pi^{(2)} &= [\beta_0, \beta_1, \beta_0],\text{ and}\\
	\pi^{(3)} &= [\beta_0, \beta_1, \beta_0,\beta_2,\beta_0,\beta_1,\beta_0].
\end{align*}

The building block $\mathfrak{h}^{(k)}$ will be a pattern of length $t^{(k)} = 2^{k+1} -1$. Although this pattern has exponentially many stepsizes, it will contain only $2k$ distinct values.
This $k$th building block stepsize pattern takes the form
\begin{align}
	\label{eq:BuildingBlockPattern}
	\begin{aligned}
		\mathfrak{h}^{(k)} = [&\alpha_0, \pi^{(0)}, \alpha_1, \pi^{(1)}, \dots,
		\alpha_{k-1}, \pi^{(k-1)},
		\ \mu_k,\ 
		\pi^{(k-1)}, \alpha_{k-1}, \dots, \pi^{(1)}, \alpha_1, \pi^{(0)}, \alpha_0].
	\end{aligned}
\end{align}
Here, if $v_1, \dots, v_k$ are vectors or numbers, then the notation $[v_1, \dots, v_k]$ denotes the concatenation of these vectors.
Note that the middle entry in this pattern equals the total of all the other entries in $\mathfrak{h}^{(k)}$ plus two. 
Also note that $\pi^{(\ell)}$ is symmetric about its middle entry, and so is $\mathfrak{h}^{(k)}$.

This construction was arrived at through substantial computer-search over patterns with the necessary properties (see our Theorem~\ref{thm:straightforward_from_rank_one}). Although the above pattern may seem somewhat cryptic, given the values in $\mathfrak{h}^{(k-1)}$, to produce the next pattern $\mathfrak{h}^{(k)}$ of the form~\eqref{eq:BuildingBlockPattern} just requires specifying three new numbers $\beta_{k-2}$, $\alpha_{k-1}$ and $\mu_k$. The values of the $\beta$ sequence follow a nice exponential pattern~\eqref{eq:beta}. Once this is set, the values for $\alpha_{k-1}$ and $\mu_k$ are then determined entirely by a system of two equations in two variables that simply imposes two \emph{necessary} conditions for straightforwardness of $\mathfrak{h}^{(k)}$ (see Section~\ref{subsec:tightness}).

Following this construction, the first four building block patterns are, for example, given by
\begin{align*}
	\mathfrak{h}^{(0)} &= \left[ 2 \right] ,\\
	\mathfrak{h}^{(1)} &= \left[ \frac{3}{2},  5,  \frac{3}{2}\right] ,\\
	\mathfrak{h}^{(2)} &= \left[ \frac{3}{2},  1+\sqrt{2}, \sqrt{2}, \ 7+4\sqrt{2},  \sqrt{2},\  1+\sqrt{2}, \frac{3}{2}\right] ,\\
	\mathfrak{h}^{(3)} &= \bigg[ \frac{3}{2}, 1+\sqrt{2}, \sqrt{2},\ -\frac{1}{2}-\sqrt{2}+\frac{3\sqrt{5}}{2}+\sqrt{10}, \ \sqrt{2},2,\sqrt{2},\\&\qquad	
	1 + (3 + 2 \sqrt{2}) (3 + \sqrt{5}),\\
	&  \qquad \sqrt{2},2,\sqrt{2},\ -\frac{1}{2}-\sqrt{2}+\frac{3\sqrt{5}}{2}+\sqrt{10},\ \sqrt{2}, 1+\sqrt{2}, \frac{3}{2}\bigg].
\end{align*}
The values of $\mathfrak{h}^{(4)}$, $\mathfrak{h}^{(5)}$, and $\mathfrak{h}^{(6)}$ are plotted in Figure~\ref{fig:example-stepsizes}, showcasing their symmetries and fractal nature.
Below we provide bounds on how the quantities $\alpha_k$, $\mu_k$, and $H_k := \sum_{i=0}^{t^{(k)-1}} \mathfrak{h}^{(k)}_i$ grow asymptotically (proof deferred in Appendix~\ref{app:proof-of-asymptotics}).

\begin{figure}
    \centering\begin{minipage}{0.5\linewidth}\includegraphics[width=\linewidth]{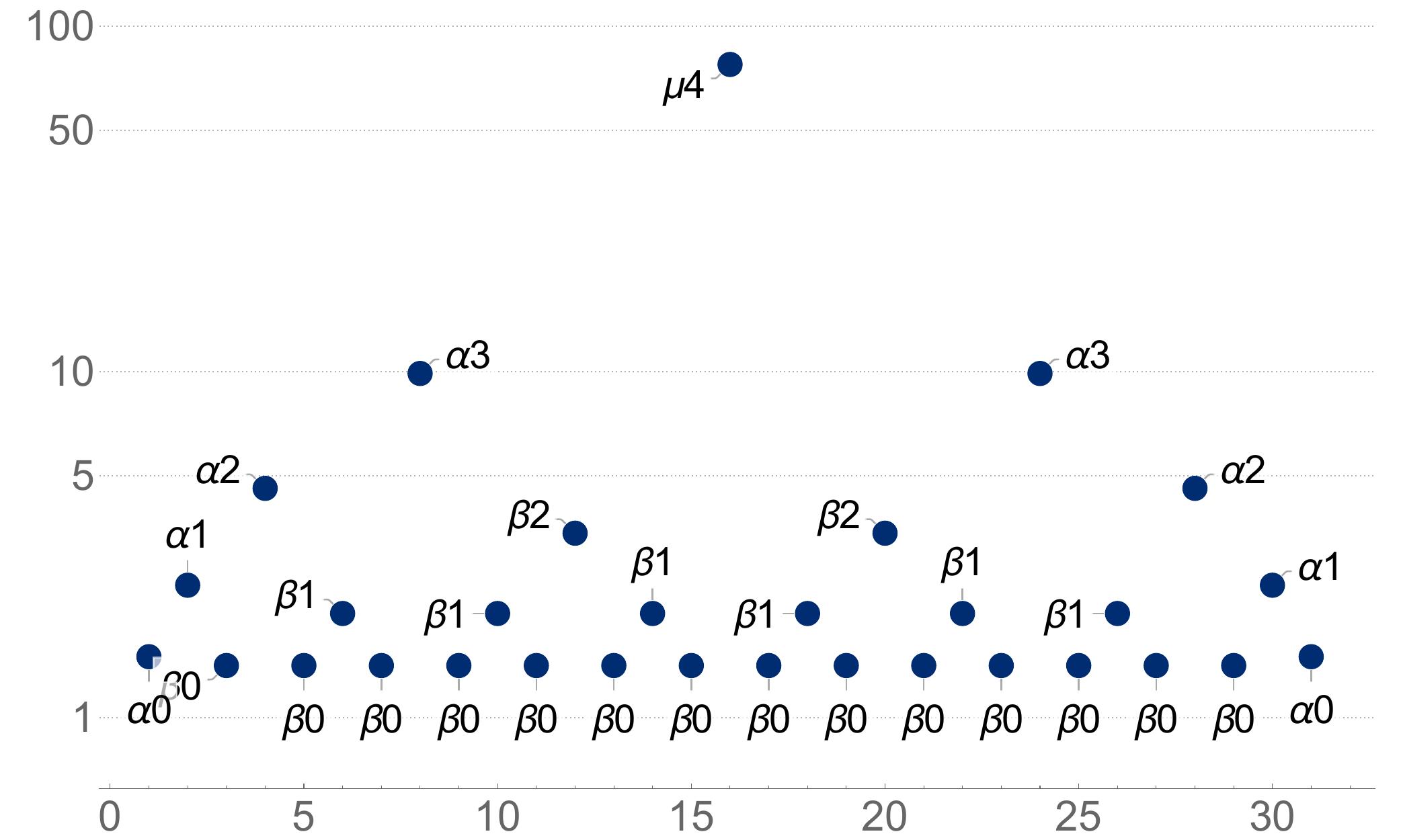}\end{minipage}\begin{minipage}{0.25\linewidth}\includegraphics[width=\linewidth]{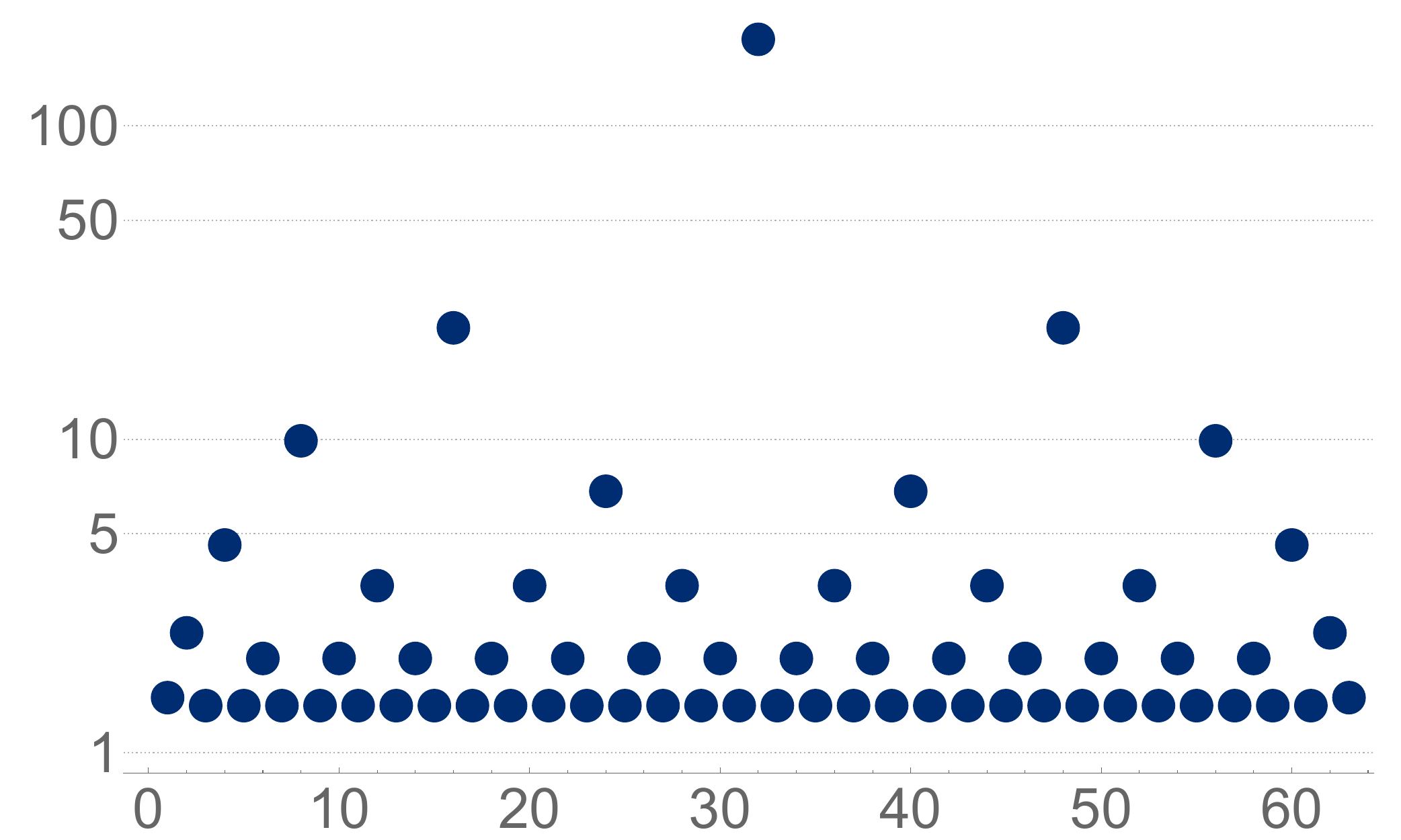}\\ \includegraphics[width=\linewidth]{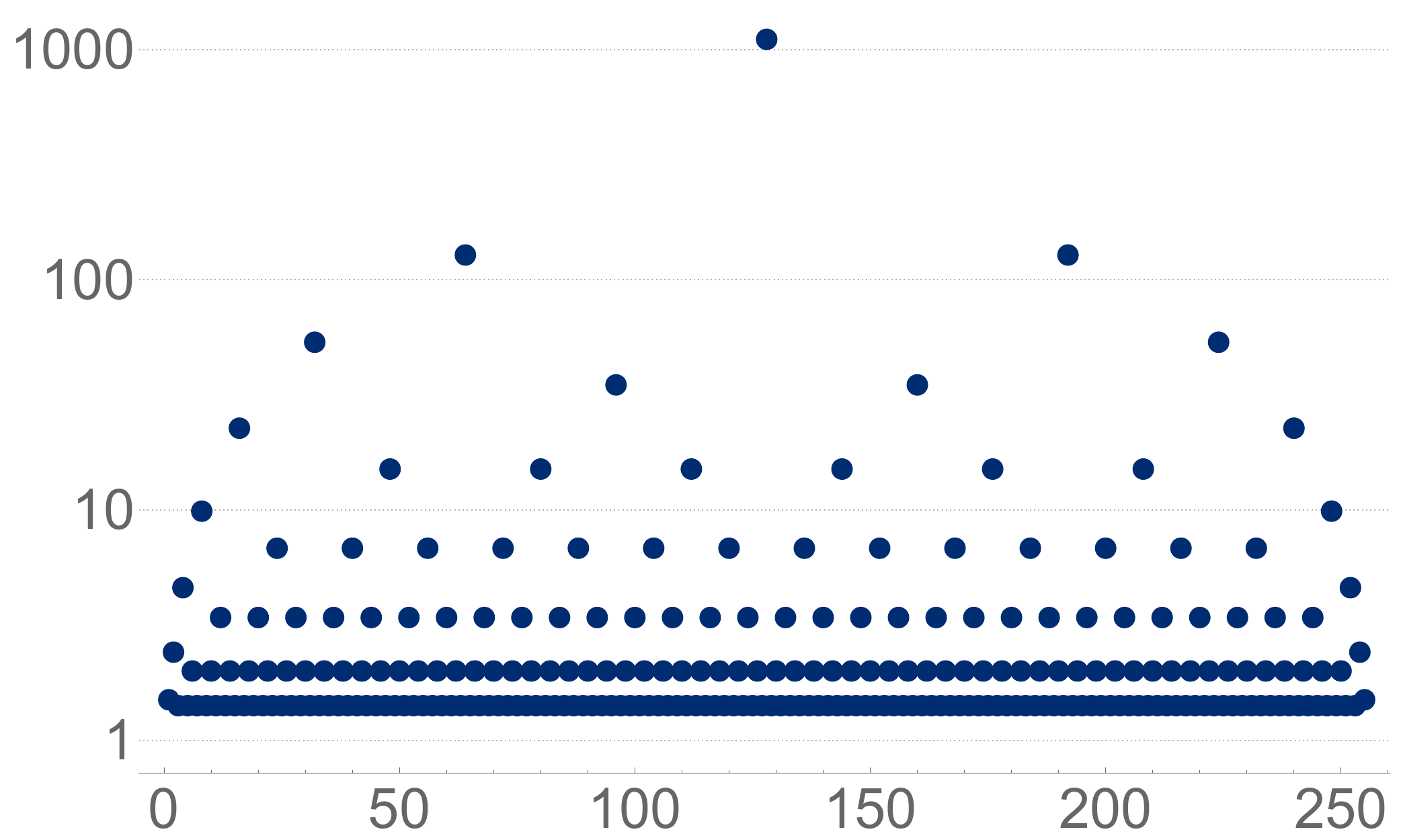}\end{minipage}\begin{minipage}{0.25\linewidth}\includegraphics[width=\linewidth]{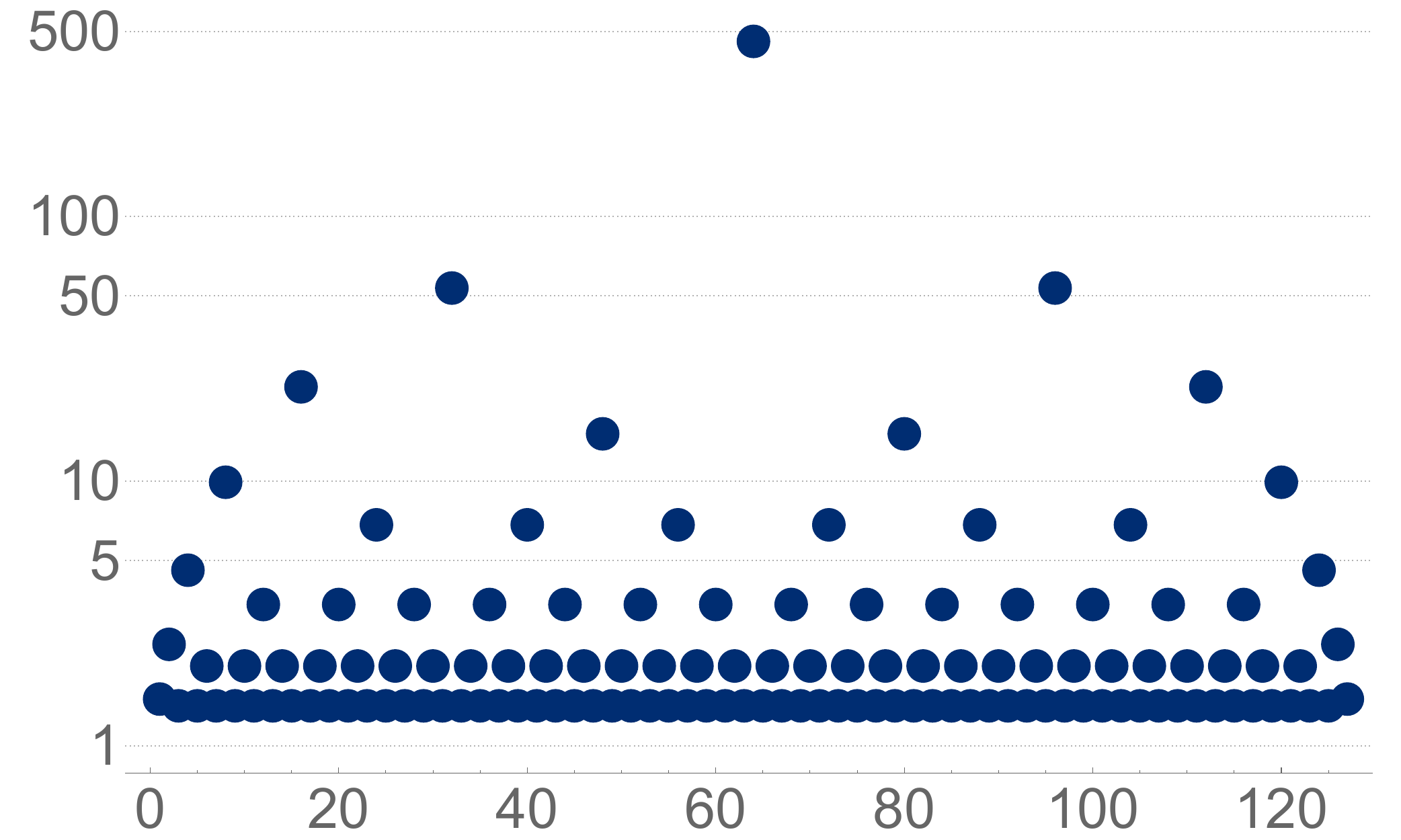}\\ \includegraphics[width=\linewidth]{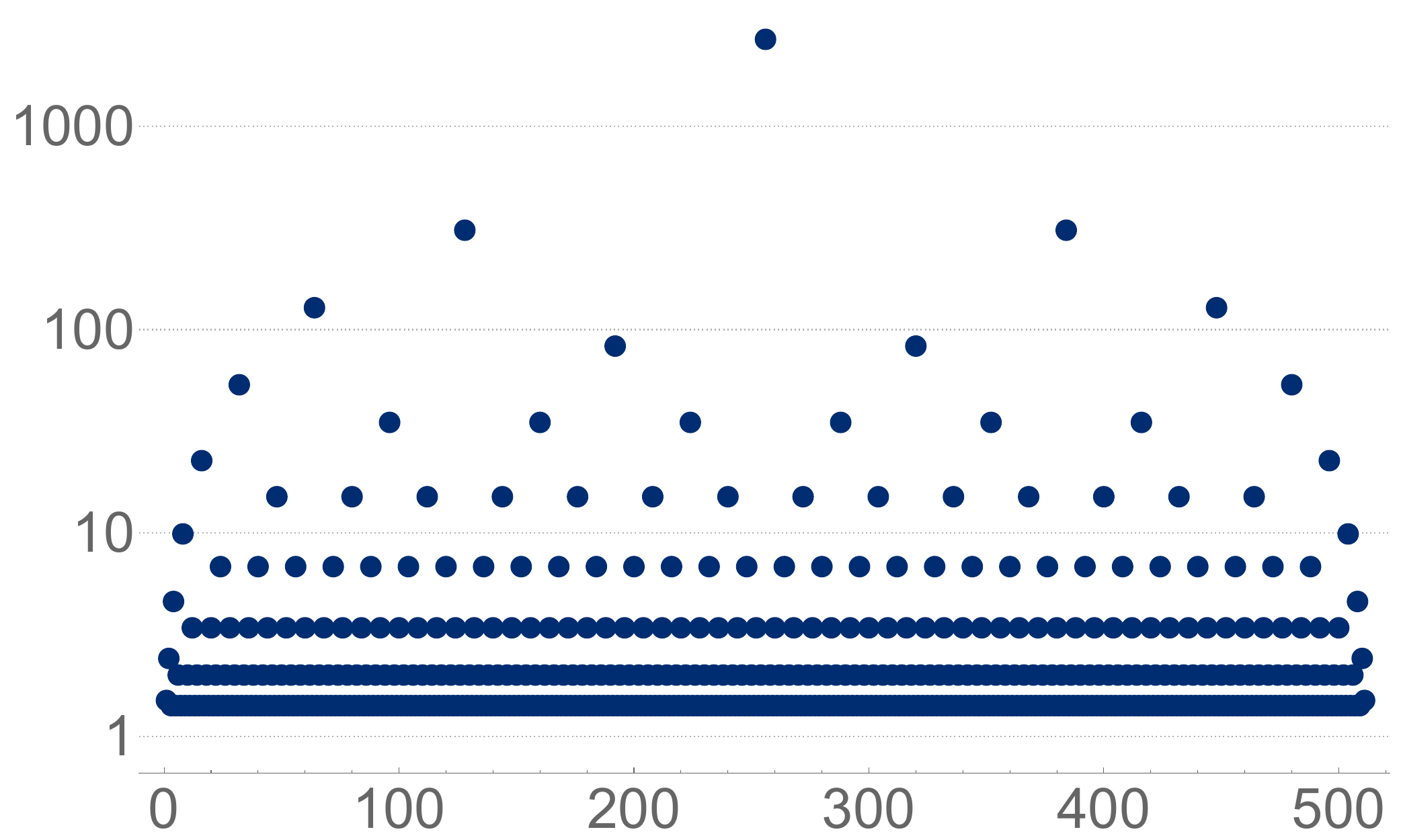}\end{minipage}
    \caption{Log-scale values of building block patterns $\mathfrak{h}^{(4)}$ labeled and unlabeled $\mathfrak{h}^{(5)},\mathfrak{h}^{(6)},\mathfrak{h}^{(7)},\mathfrak{h}^{(8)}$.}\label{fig:example-stepsizes}
\end{figure}
\begin{lemma} \label{lem:helpful-bounds-Hk-muk}
    For all $k\geq 1$, it holds that
    \begin{align*}
        \beta_k &\leq \alpha_k \leq \beta_{k+1},\\
        2\sqrt{2}(1+\sqrt{2})^k &\leq H_k \leq 4\sqrt{2} (1+\sqrt{2})^k,\\
        \sqrt{2}(1+\sqrt{2})^k &\leq \mu_k - 1 \leq 2\sqrt{2}(1+\sqrt{2})^k,\\
        (2-\sqrt{2})(1+\sqrt{2})^{k-1}&\leq \mu_k - \mu_{k-1}.
    \end{align*}
\end{lemma}
From this, we see these building block patterns not only have exponentially large steps occur periodically, but also have exponentially large average stepsizes. In particular, they have
$ \mathrm{avg}(\mathfrak{h}^{(k)}) \geq \sqrt{2}\left(\frac{1+\sqrt{2}}{2}\right)^k$.
Despite containing exponentially large steps on average, in Section~\ref{sec:certificate}, we show that applying these patterns always yields an objective value decrease (by the end of the pattern).

\subsubsection{Building the Proposed Nonconstant, Nonperiodic Stepsize Sequence}
We construct our accelerated sequence of long stepsizes from rescaled versions of the stepsize building block patterns $\mathfrak{h}^{(k)}$ by some fixed scalar $\eta\in(0,1)$. 
We first apply $(1-\eta)\mathfrak{h}^{(0)}$ a fixed number of times, then apply the pattern $(1-\eta)\mathfrak{h}^{(1)}$ a fixed number of times, and so on. Each rescaled pattern $(1-\eta)\mathfrak{h}^{(k)}$ will be applied
\begin{equation}\label{eq:definition-H_k}
	R_k = \left\lceil \frac{1}{(1-\eta)\mathrm{avg}(\mathfrak{h}^{(k)})(2^{k+1}-1)\Delta^{(k+1)}} \right\rceil
\end{equation}
times where the associated parameter $\Delta^{(k)}$ is defined as in Lemma~\ref{lem:straightforward}.
\begin{remark}
    We do not believe the choice of $\Delta^{(k)}$ therein is as large as possible. Improvements on that parameter directly would improve our guarantees as fewer applications of each pattern would be needed.
    In fact, we propose a conjecture following our \cref{lem:straightforward} on how $\Delta^{(k)}$ should scale with $k$ (see \cref{conj:delta_silver_squared}). This conjecture is supported numerically and a proof of it would directly lead to an $O(1/T^{1.119})$ convergence rate guarantee.
\end{remark}

As a result, the proposed nonconstant, nonperiodic stepsize sequence is
\begin{equation}\label{eq:whole-sequence}
	h = (1-\eta) \left(\underbrace{\mathfrak{h}^{(0)}, \mathinner{\cdotp\mkern-2mu\cdotp\mkern-2mu\cdotp} ,\mathfrak{h}^{(0)}}_{\text{\small $R_0$ times}}, \ \underbrace{\mathfrak{h}^{(1)}, \mathinner{\cdotp\mkern-2mu\cdotp\mkern-2mu\cdotp} ,\mathfrak{h}^{(1)}}_{\text{\small $R_1$ \  times}}, \ \underbrace{\mathfrak{h}^{(2)}, \mathinner{\cdotp\mkern-2mu\cdotp\mkern-2mu\cdotp} ,\mathfrak{h}^{(2)}}_{\text{\small $R_2$ \  times}}, \mathinner{\cdotp\mkern-2mu\cdotp\mkern-2mu\cdotp} ,\ \underbrace{\mathfrak{h}^{(k)}, \mathinner{\cdotp\mkern-2mu\cdotp\mkern-2mu\cdotp} ,\mathfrak{h}^{(k)}}_{\text{\small $R_k$ \  times}}, \mathinner{\cdotp\mkern-2mu\cdotp\mkern-2mu\cdotp}\right) .   
\end{equation}
We denote the first iteration where stepsizes are drawn from the pattern $\mathfrak{h}^{(k)}$ by
\begin{equation}\label{eq:definition-I_k}
	I_k = \sum_{l=0}^{k-1} (2^{l+1}-1)R_l .
\end{equation}
Note the value of $\Delta^{(k)}$ shrinks geometrically. As a result, the iteration counts where we switch to the next building block stepsize pattern $I_k$ grows geometrically.  For example, setting $\eta=1/2$, the proposed sequences of stepsizes would be of the form
\begin{align*}
    h &= \Bigg[\underbrace{ 1,1,1, \ \dots}_{\text{\small $R_0$ iterations}}\ ,\underbrace{\frac{3}{4}, \frac{5}{2},\frac{3}{4},\frac{3}{4}, \frac{5}{2},\frac{3}{4},\frac{3}{4}, \frac{5}{2},\frac{3}{4},\ \dots}_{\text{\small $3R_1$ iterations}}\ , \\
    &\quad\qquad\qquad\underbrace{\frac{3}{4},  \frac{1+\sqrt{2}}{2}, \frac{\sqrt{2}}{2}, \frac{7+4\sqrt{2}}{2},  \frac{\sqrt{2}}{2}, \frac{1+\sqrt{2}}{2}, \frac{3}{4}, \ \dots}_{\text{\small $7R_2$ iterations}}  \Bigg] .    
\end{align*}
     \section{Performance Estimation and Straightforwardness} \label{sec:review}
Our proof machinery is built upon the performance estimation problem (PEP) ideas of~\cite{drori2012PerformanceOF,taylor2017interpolation,Grimmer2023-long}. We first introduce this PEP line of work~\cite{drori2012PerformanceOF,taylor2017interpolation} and associated semidefinite programs applied to our particular setting. Then, we introduce the improved semidefinite programming technique of~\cite{Grimmer2023-long}, identifying a class of stepsize patterns, dubbed straightforward, for which the effects of long steps can be analyzed.

\subsection{Performance Estimation Problems}
Give a fixed smoothness constant $L$, distance bound $D$, stepsize pattern $\mathfrak{h}\in\mathbb{R}^t$ and problem dimension $n$, one can consider the worst possible final objective gap $f(x_t)-f(x_\star)$ as a function of the initial objective gap $f(x_0)-f(x_\star)\leq \delta$. We denote this by the infinite-dimensional optimization problem
\begin{equation} \label{eq:value-function}
    p_{L,D}(\delta) := \begin{cases}
    \max_{x_0,x_\star,f} & f(x_t) - f(x_\star) \\
    \mathrm{s.t.} & f \text{ is convex, } L\text{-smooth}\\
    & \|x_0-x_\star\|_2 \leq D\\
    & f(x_0) - f(x_\star) \leq \delta \\
    & \nabla f(x_\star)=0 \\
    & x_{i+1} = x_i - \frac{\mathfrak{h}_i}{L} \nabla f(x_k) \quad \forall i=0,\dots, t-1 \ .
    \end{cases}
\end{equation}

The performance estimation problem (PEP) results of~\cite{drori2012PerformanceOF,taylor2017interpolation} establish that this problem can be relaxed (often tightly instead as a reformulation) to a finite-dimensional semidefinite minimization problem. Their PEP process of reformulations is carried out below, following the notation used in Grimmer~\cite{Grimmer2023-long}, to introduce the needed notations here and for completeness. 

\noindent {\bf Step 1: A QCQP reformulation.} First, as proposed by Drori and Teboulle~\cite{drori2012PerformanceOF}, one can discretize the infinite-dimensional problem defining $p_{L,D}(\delta)$ over all possible objective values $f_k$ and gradients $g_k$ at the points $x_k$ with $k\in I_t^\star := \{\star,0,1,\dots t\}$ as done below. Using the interpolation theorem of Taylor et al.~\cite{taylor2017interpolation}, this gives an exact reformulation rather than a relaxation, giving
\begin{equation}
p_{L,D}(\delta) = \begin{cases}
    \max_{x_0,f,g} & f_t - f_\star\\
    \mathrm{s.t.} & f_i \geq f_j + g_j^{\intercal}(x_i-x_j) + \frac{1}{2L}\|g_i-g_j\|^2_2 \qquad \forall i,j\in I_t^\star \\
    & \|x_0-x_\star\|^2_2 \leq D^2\\
    & f_0 - f_\star \leq \delta \\
    & x_\star=0,f_\star=0, g_\star=0 \\
    & x_{i+1} = x_i - \frac{\mathfrak{h}_i}{L} g_i \qquad \forall i=0,\dots, t-1
\end{cases}\label{eq:QCQP-form}
\end{equation}
where, without loss of generality, we have fixed $x_\star=0,f_\star=0, g_\star=0$. \\[-8pt]

\noindent {\bf Step 2: An SDP relaxation.} Second, one can relax the nonconvex problem~\eqref{eq:QCQP-form} to the following SDP as done in~\cite{taylor2017interpolation,taylor2017smooth,Drori2018EfficientFM}. Define
\begin{align*}
H & :=[x_{0}\mid g_{0}\mid g_{1}\mid\ldots\mid g_{t}]\in\mathbb{R}^{d\times(t+2)}\ ,\\
G & :=H^{\intercal}H\in\mathbb{S}_{+}^{t+2}\ ,\\
F & :=[f_{0}\mid f_{1}\mid\ldots\mid f_{t}]\in\mathbb{R}^{1\times(t+1)} \ , 
\end{align*}
with the following notation for selecting columns and elements of $H$ and $F$: 
\begin{align*}
\mathbf{g}_{\star} & :=0\in\mathbb{R}^{t+2},\;\mathbf{g}_{i}:=e_{i+2}\in\mathbb{R}^{t+2},\quad i\in[0:t]\\
\mathbf{x}_{0} & :=e_{1}\in\mathbb{R}^{t+2},\;\mathbf{x}_{\star}:=0\in\mathbb{R}^{t+2}\ ,\\
\mathbf{x}_{i} & :=\mathbf{x}_{0}-\frac{1}{L}\sum_{j=0}^{i-1}\mathfrak{h}_j\mathbf{g}_{j}\in\mathbb{R}^{t+2},\quad i\in[1:t]\\
\mathbf{f}_{\star} & :=0\in\mathbb{R}^{t+1},\;\mathbf{f}_{i}:=e_{i+1}\in\mathbb{R}^{t+1},\quad i\in[0:t] \ .
\end{align*}
This notation ensures
$
x_{i}=H\mathbf{x}_{i}$, $g_{i}=H\mathbf{g}_{i}$, and $f_{i}=F\mathbf{f}_{i}.
$
Furthermore,
for $i,j\in I_{t}^{\star}$, define 
\begin{align*}
A_{i,j}(\mathfrak{h})& :=\mathbf{g}_{j}\odot(\mathbf{x}_{i}-\mathbf{x}_{j})\in\mathbb{S}^{t+2}\ ,\\
B_{i,j}(\mathfrak{h})& :=(\mathbf{x}_{i}-\mathbf{x}_{j})\odot(\mathbf{x}_{i}-\mathbf{x}_{j})\in\mathbb{S}_{+}^{t+2}\ ,\\
C_{i,j}& :=(\mathbf{g}_{i}-\mathbf{g}_{j})\odot(\mathbf{g}_{i}-\mathbf{g}_{j})\in\mathbb{S}_{+}^{t+2}\ ,\\
a_{i,j}& :=\mathbf{f}_{j}-\mathbf{f}_{i}\in\mathbb{R}^{t+1}
\end{align*}
where $x\odot y= \frac{1}{2}(xy^{\intercal} + yx^{\intercal})$ denotes the symmetric outer product. This
notation is defined so that $ g_{j}^{\intercal}(x_{i}-x_{j}) =\tr GA_{i,j}(\mathfrak{h})$, $\|x_{i}-x_{j}\|^{2}_2=\tr GB_{i,j}(\mathfrak{h})$, and $\|g_{i}-g_{j}\|^{2}_2=\tr GC_{i,j}$ for any $i,j\in I_{t}^{\star}$.
Then the QCQP formulation~\eqref{eq:QCQP-form} can be relaxed to
\begin{align}
p_{L,D}(\delta)
\leq & \begin{cases}
\max_{F,G} & F\mathbf{f}_{t}\\
\textrm{s.t.} & Fa_{i,j}+\tr GA_{i,j}(\mathfrak{h})+\frac{1}{2L}\tr GC_{i,j}\leq0,\quad i,j\in I_{t}^{\star}:i\neq j \\
& -G\preceq 0\\
& \tr GB_{0,\star}\leq D^{2}\\
& F\mathbf{f}_{0} \leq \delta \ .
\end{cases}\label{eq:SDP-form}
\end{align}
Under an additional rank condition (that the problem dimension $n$ exceeds $t+2$), the QCQP~\eqref{eq:QCQP-form} and SDP~\eqref{eq:SDP-form} are actually equivalent. However, this is not needed for our analysis, so we make no such assumption. 

\noindent {\bf Step 3: The upper bounding dual SDP.} Third, note the maximization SDP~\eqref{eq:SDP-form} is bounded above by its dual minimization SDP by weak duality, giving
\begin{equation}
p_{L,D}(\delta) \leq
\begin{cases}
    \min_{\lambda,v,w,Z} & D^{2} v + \delta w \\
\textrm{s.t.} &
\sum_{i,j\in I_{t}^{\star}:i\neq j}\lambda_{i,j}a_{i,j}= a_{\star, t} - w a_{\star, 0} \\
&v B_{0,\star} +\sum_{i,j\in I_{t}^{\star}:i\neq j}\lambda_{i,j}\left(A_{i,j}(\mathfrak{h})+\frac{1}{2L}C_{i,j}\right)=Z \\
&Z\succeq0 \\
&v,w\geq0,\;\lambda_{i,j}\geq0,\quad i,j\in I_{t}^{\star}:i\neq j \ .
\end{cases}\label{eq:dual-form}
\end{equation}
Although it is not needed for our analysis, equality holds here as well (i.e., strong duality holds) due to~\cite[Theorem 6]{taylor2017interpolation}.

The matrix $Z$ above is entirely determined by its equality constraint, so one can consider it as being a function of $\lambda\in\mathbb{R}^{(t+2)\times(t+2)}$, $v\in\mathbb{R}$, and $\mathfrak{h}\in\mathbb{R}^{t}$,
\begin{align*}
    Z(\lambda,v,\mathfrak{h}) = v B_{0,\star} +\sum_{i,j\in I_{t}^{\star}:i\neq j}\lambda_{i,j}\left(A_{i,j}(\mathfrak{h})+\frac{1}{2L}C_{i,j}\right).
\end{align*}
The dependence of this function on each parameter is made clear by considering $Z$ broken into the following blocks
$$ Z(\lambda, v, \mathfrak{h}) = \begin{bmatrix} v & m(\lambda)^{\intercal} \\ m(\lambda) & M(\lambda, \mathfrak{h}) \end{bmatrix} . $$
This first entry only depends on $v$ and $v$ only occurs in this first entry. The remainder of the first column and row are linear functions of only $\lambda$, denoted by $m(\lambda)$. The remaining $(t+1)\times (t+1)$ block of $Z$ is linear in $\lambda$ and affine in $\mathfrak{h}$, denoted by $M(\lambda, \mathfrak{h})$.

\subsection{Straightforward Stepsize Patterns}
The performance estimation technique defined above does not provide a mechanism to give convergence rates for gradient descent as it only considers a fixed number of iterations $t$. To enable these PEP semidefinite programs to provide convergence rate theorems, Grimmer~\cite{Grimmer2023-long} proposed considering stepsize patterns where this worst-case function is bounded above by
$$ p_{L,D}(\delta) \leq \delta - \frac{\sum_{i=0}^{t-1} \mathfrak{h}_i}{LD^2}\delta^2  $$
for all $\delta \in [0, LD^2\Delta]$ for some $\Delta>0$. Any stepsize pattern $\mathfrak{h}\in\mathbb{R}^t$ where this holds is called {\it straightforward} with parameter $\Delta$. Let $\delta_{i} = f(x_i)-f(x_\star)$. Repeated application of such patterns can be analyzed inductively by the descent recurrence relation $\delta_{(s+1)t} \leq \delta_{st} - \frac{\sum_{i=0}^{t-1} \mathfrak{h}_i}{LD^2}\delta^2_{st} $.

Theorem~3.2 of~\cite{Grimmer2023-long} showed that a stepsize pattern is straightforward with parameter $\Delta$ if the following spectral set is nonempty
$$ \mathcal{S}_{\mathfrak{h},\Delta} = \left\{ (\lambda,\gamma)\in\mathbb{R}^{(t+2)\times (t+2)}\times\mathbb{R}^{(t+2)\times (t+2)} \mid \begin{array}{l}
    \sum_{i,j\in I_{t}^{\star}:i\neq j}\lambda_{i,j}a_{i,j}= a_{\star, t} - a_{\star, 0}\\
    \sum_{i,j\in I_{t}^{\star}:i\neq j}\gamma_{i,j}a_{i,j}= 2\sum_{i=0}^{t-1} \mathfrak{h}_i a_{\star, 0}\\
    m(\lambda)=0 \\
    \lambda \geq 0, \lambda + \Delta\gamma \geq 0\\
    \begin{bmatrix} \sum_{i=0}^{t-1} \mathfrak{h}_i & m(\gamma)^{\intercal} \\ m(\gamma) & M(\lambda, \mathfrak{h}) \end{bmatrix}\succeq 0\\
    \begin{bmatrix} \sum_{i=0}^{t-1} \mathfrak{h}_i & m(\gamma)^{\intercal} \\ m(\gamma) & M(\lambda+\Delta\gamma, \mathfrak{h}) \end{bmatrix}\succeq 0 
    \end{array}\right\} .$$
Hence proving a pattern is straightforward amounts to identifying a feasible solution to a semidefinite program. Grimmer used this to automate the search for long straightforward patterns, generating their constant factor convergence rate improvements for periodic stepsize sequences. 

This straightforward structure is critical to our proof development as well. We will show that any rescaling $\eta\in(0,1)$ of our building block patterns $(1-\eta)\mathfrak{h}^{(k)}$ is straightforward. However, in contrast to this prior work, our proof of this will be entirely analytic and hence apply for all $k$. The move from computer-generated certificates to exact algebraic formulas was essential to move the resulting performance-bound gains from being constant factor improvements to our accelerated big-O convergence.

     \section{Accelerated Convergence Rate Analysis}\label{sec:proof}

Our convergence rate analysis relies on three lemmas, (i) showing each rescaled building block pattern is straightforward, (ii) guaranteeing progress is made after $R_k$ applications of each pattern, and (iii) bounding the total number of steps in each pattern. Our first lemma requires substantial and nontrivial constructions and verification to prove. This is deferred to Sections~\ref{sec:certificate_part_1} and beyond.
\begin{lemma}\label{lem:straightforward}
	Each (scaled) building block pattern $(1-\eta) \mathfrak{h}^{(k)}$ is straightforward. In particular, $\mathfrak{h}^{(0)}/2$ has parameter $\Delta=1/2$ and for $k>0$, $\mathfrak{h}^{(k)}/2$ has parameter
	$$ \Delta^{(k)} = \frac{1}{768768 \sqrt{2}}\frac{1}{(1+\sqrt{2})^{4k}} > 9\times 10^{-7} (33.98)^{-k} >0 . $$ 
\end{lemma}
Note these parameters $\Delta^{(k)}$ are chosen very conservatively. Any at most exponentially decaying lower bound suffices to give a rate strictly faster than $O(1/T)$. The slack in our above bound can be seen by numerically computing the largest $\Delta$ such $\mathcal{S}_{\mathfrak{h}^{(k)}/2,\Delta}$ is nonempty (implying straightforwardness for $\mathfrak{h}^{(k)}/2$) for $k=1,\dots 5$. The resulting numerical values are given below
\begin{align*}
	k=1 \implies \quad &\Delta^{(1)}_{\mathtt{numerical}} \approx 9.33\times 10^{-2} > \Delta^{(1)}\approx 2.70\times 10^{-8} , \\
	k=2 \implies \quad &\Delta^{(2)}_{\mathtt{numerical}} \approx 1.28\times 10^{-2} > \Delta^{(2)}\approx 7.97\times 10^{-10} , \\
	k=3 \implies \quad &\Delta^{(3)}_{\mathtt{numerical}} \approx 2.03\times 10^{-3} > \Delta^{(3)}\approx 2.34\times 10^{-11} , \\
	k=4 \implies \quad &\Delta^{(4)}_{\mathtt{numerical}} \approx 3.34\times 10^{-4} > \Delta^{(4)}\approx 6.90\times 10^{-13} , \\
	k=5 \implies \quad &\Delta^{(5)}_{\mathtt{numerical}} \approx 5.63\times 10^{-5} > \Delta^{(5)}\approx 2.03\times 10^{-14} .
\end{align*}
The successive ratios in the numerical values of $\Delta^{(k)}$ are decreasing and seem to approach $(1+\sqrt{2})^2\approx 5.83$. For example, $\Delta^{(4)}_{\mathtt{numerical}}/\Delta^{(5)}_{\mathtt{numerical}} \approx 5.93$. This suggests the following conjecture:
\begin{conjecture}
    \label{conj:delta_silver_squared}
There exists $c>0$ such that for all $k\geq 1$, the building block $\mfhk/2$ is straightforward with $\Delta^{(k)} \geq c (1+\sqrt{2})^{-2k}$.
\end{conjecture}
As discussed previously, proving this conjecture would directly lead to an $O(1/T^{1.119})$ convergence rate guarantee. We also mention that if one could show the \emph{even stronger} bound of $\Delta^{(k)} =\Omega((1+\sqrt{2})^{-k})$, then our steplength schedule $h$ would have the nice property that each block $\mathfrak{h}^{(k)}$ is repeated only constant number of times and the resulting convergence rate guarantee, $O(1/T^{1.2715})$, would match the rate achieved by Altschuler and Parrilo~\cite{altschuler2023acceleration} with their modified gradient descent algorithm.

Our second lemma analyzes the objective gap of gradient descent with the stepsize sequence~\eqref{eq:whole-sequence} after completing $R_k$ applications of $(1-\eta)\mathfrak{h}^{(k)}$. This lemma follows directly from \cref{lem:straightforward}. Recall $I_{k+1}$ denotes the iteration of gradient descent just after applications of $\mathfrak{h}^{(k)}$ has completed. 
\begin{lemma}\label{lem:inductive-progress}
	Each $k\geq 0$ has $f(x_{I_k}) - f(x_\star) \leq LD^2 \Delta^{(k)}$.
\end{lemma}
\begin{proof}
	Note this trivially holds for $k=0$ since $I_0=0$ and $f(x_0)-f(x_\star)\leq \frac{1}{2}LD^2$. We prove this inductively by showing that if $\delta_{I_{k}} \leq LD^2 \Delta^{(k)}$, then $\delta_{I_{k+1}} \leq LD^2\Delta^{(k+1)}$.
	Suppose $\delta_{I_{k}}\leq LD^2\Delta^{(k)}$. Then the straightforwardness of $(1-\eta)\mathfrak{h}^{(k)}$ ensures that the objective gap decreases with each application of the pattern $(1-\eta)\mathfrak{h}^{(k)}$. Namely, for any $s\geq 0$, we have
	$$ \delta_{I_{k} + (2^{k+1}-1)(s+1)} \leq \delta_{I_{k} + (2^{k+1}-1)s} - \frac{(1-\eta)\mathrm{avg}(\mathfrak{h}^{(k)})(2^{k+1}-1)}{LD^2} \delta_{I_{k} + (2^{k+1}-1)s}^2. $$
	Solving this recurrence relation (of the standard form $\delta_{s+1}\leq\delta_{s}-C\delta_{s}^2$) ensures that for all $s\geq 0$
	\begin{equation}\label{eq:pattern-recurrence-solve}
		\delta_{I_{k} + (2^{k+1}-1)s} \leq \frac{LD^2}{(1-\eta)\mathrm{avg}(\mathfrak{h}^{(k)})(2^{k+1}-1) s}.
	\end{equation}
	Hence iteration $I_{k+1}$, after $s=R_k$ applications of $(1-\eta)\mathfrak{h}^{(k)}$, has $\delta_{I_{k+1}} \leq LD^2\Delta^{(k+1)}$.
\end{proof}
To convert this bound to a convergence rate guarantee, we need a bound on $I_k$, given below. 
\begin{lemma}\label{lem:I_k-bounds}
	Let $c= \frac{1}{768768 \sqrt{2}}$ and $d=(1+\sqrt{2})^4$. If $\eta=1/2$, then $I_0=0$ and for $k>0$,
	$$I_k \leq
	\frac{(1+10^{-8})(1+\sqrt{2})}{c\left(1 - \frac{1+\sqrt{2}}{2d}\right)} \left(\frac{2 d}{1+\sqrt{2}}\right)^{k} . $$
\end{lemma}
\begin{proof}
	Trivially $I_0=0$. Recall from Lemma~\ref{lem:straightforward} that the straightforwardness parameter is given by
	$ \Delta^{(0)} = 1/2 $ and $\Delta^{(k)}=c/d^{k}$ if $k\geq 1$. Then for $l\geq 0$,
	\begin{align*}
		R_l = \left\lceil \frac{1}{(1-\eta)\mathrm{avg}(\mathfrak{h}^{(l)})(2^{l+1}-1)\Delta^{(l+1)}} \right\rceil
		\leq \left\lceil \frac{d^{l+1}}{c(1+\sqrt{2})^l} \right\rceil
		& \leq (1+10^{-7})\left( \frac{d^{l+1}}{c(1+\sqrt{2})^l} \right),
	\end{align*}
	using that $(1-\eta)\mathrm{avg}(\mathfrak{h}^{(l))})(2^{l+1}-1) \geq (1+\sqrt{2})^l$ and that $\frac{d^{l+1}}{c(1+\sqrt{2})^l} \geq 10^{7}$ for all $l\geq 0$.
	Then one can bound $I_{k}$ for $k>0$ as
	\begin{align*}
		I_{k} = \sum_{l=0}^{k-1} (2^{l+1}-1)R_l & <  \sum_{l=0}^{k-1} \frac{2(1+10^{-7})d}{c}\left(\frac{2d}{1+\sqrt{2}}\right)^l\\
		&=  \frac{(1+10^{-7})(1+\sqrt{2})}{c\left(1 - \frac{1+\sqrt{2}}{2d}\right)} \left(\frac{2 d}{1+\sqrt{2}}\right)^{k}. \qedhere
	\end{align*}
\end{proof}

From this, our accelerated convergence guarantee for gradient descent is immediate.
\begin{theorem}\label{thm:main-result}
	For any target accuracy $0<\epsilon<\frac{1}{2}LD^2$ and rescaling $\eta=1/2$, gradient descent with stepsize sequence~\eqref{eq:whole-sequence} has $f(x_T)-f(x_\star)\leq \epsilon$ by some iteration 
	$$ T \leq  7.00809\left(\frac{LD^2}{\epsilon}\right)^{0.976091} .$$
\end{theorem}
\begin{proof}
	As in Lemma~\ref{lem:I_k-bounds}, let $c= \frac{1}{768768 \sqrt{2}}$ and $d=(1+\sqrt{2})^4$ for which $\Delta^{(k)}=c/d^{k}$ if $k\geq 1$. Let $k(\epsilon)$ be the last pattern $k$ with $\Delta^{(k)} > \epsilon/LD^2$. Since $\Delta^{(k)}$ monotonically decreases to zero, $k(\epsilon)$ is finite.
	Lemma~\ref{lem:I_k-bounds} bounds the number of steps before building block pattern $k(\epsilon)$ is used by
	$$I_{k(\epsilon)} = \begin{cases}
		0 & \text{ if } k(\epsilon)=0\\
		\frac{(1+10^{-7})(1+\sqrt{2})}{c\left(1 - \frac{1+\sqrt{2}}{2d}\right)} \left(\frac{2 d}{1+\sqrt{2}}\right)^{k(\epsilon)} & \text{ if } k(\epsilon)>0.
	\end{cases}$$
	A bound on the number of steps using pattern $\mathfrak{h}^{(k(\epsilon))}$ before an $\epsilon$-minimizer follows from the convergence guarantee~\eqref{eq:pattern-recurrence-solve}. Namely, the number of applications of the pattern $(1-\eta)\mathfrak{h}^{k(\epsilon)}$ needed is at most 
	$$ \frac{LD^2}{(1+\sqrt{2})^{k(\epsilon)}\epsilon}.$$
	Hence, if $k(\epsilon)>0$ and so $k(\epsilon) = \lfloor\log_d(c LD^2/\epsilon)\rfloor$, an $\epsilon$-minimizer is found by iteration
	\begin{align*}
		& \frac{(1+10^{-7})(1+\sqrt{2})}{c\left(1 - \frac{1+\sqrt{2}}{2d}\right)} \left(\frac{2 d}{1+\sqrt{2}}\right)^{k(\epsilon)} + (2^{k(\epsilon)+1}-1)\frac{LD^2}{(1+\sqrt{2})^{k(\epsilon)}\epsilon} \\
		&\leq \left(\frac{(1+10^{-7})(1+\sqrt{2})}{c^{\log_d((1+\sqrt{2})/2)}\left(1 - \frac{1+\sqrt{2}}{2d}\right)} + \frac{1+\sqrt{2}}{c^{\log_d((1+\sqrt{2})/2)}}\right)\left(\frac{LD^2}{\epsilon}\right)^{1-\log_d((1+\sqrt{2})/2)} \\
		& \leq 10.3280 \left(\frac{LD^2}{\epsilon}\right)^{0.94662}.
	\end{align*}
	Otherwise if $k(\epsilon)=0$, then an $\epsilon$-minimizer is found within $LD^2/\epsilon$ iterations. Since $k(\epsilon)=0$ requires $\epsilon > LD^2 c/d$, one can verify $LD^2/\epsilon < 10.3280 \left(\frac{LD^2}{\epsilon}\right)^{0.94662}$.
\end{proof}
This convergence rate improves if stronger lower bounds on $\Delta^{(k)} = c/d^{k}$ are provided. Convergence theory matching the conjectured optimal rate of~\cite{altschuler2023acceleration} would follow immediately if one could show a bound with $d=1+\sqrt{2}$. This would give a rate of $O(1/\epsilon^{\log_{1+\sqrt{2}}(2)})$ and has the nice property that $R_k$ would become a constant. Such an idealized, potentially optimal, accelerated stepsize pattern would just apply each $\mathfrak{h}^{(k)}$ a constant number of times.

\subsection{Accelerated Linear Convergence for Strongly Convex Optimization}
Lemma~\ref{lem:straightforward} further enables direct analysis of gradient descent for strongly convex minimization. Let $D_k = \sup\{\|x -x_\star\| \mid f(x)\leq f(x_k)\}$ and $t^{(k)} = 2^{k+1}-1$. Note that $\mu$-strong convexity of $f$ (defined as $f-\frac{\mu}{2}\|\cdot\|^2$ being convex) ensures $\frac{\mu}{2}D_k^2 \leq \delta_k$.\footnote{In fact, this is the only property of strong convexity used in our analysis here. So our convergence guarantee presented in Theorem~\ref{thm:strongly-convex} holds more generally for any problem satisfying only a quadratic growth bound.}
Observe that if $\delta_{st^{(k)}} \leq LD_k^2\Delta^{(k)}$, the objective gap contracts after applying the pattern $(1-\eta)\mathfrak{h}^{(k)}$ with
$$\delta_{(s+1)t^{(k)}} \leq \left(1 -  \frac{\mu\sum_{i=0}^{t^{(k)}-1} (1-\eta)\mathfrak{h}^{(k)}_i}{2L}\right)\delta_{st^{(k)}}. $$
Conversely, if $\delta_{st^{(k)}} > LD_k^2\Delta^{(k)}$, straightforwardness ensures that
$$\delta_{(s+1)t^{(k)}} \leq \left(1 -  \Delta\sum_{i=0}^{t^{(k)}-1} (1-\eta)\mathfrak{h}^{(k)}_i \right)\delta_{st^{(k)}}. $$
Hence, every application of this pattern yields a contraction of at least
\begin{equation} \label{eq:contraction}
	\delta_{(s+1)t^{(k)}} \leq \left(1 -  \min\left\{\Delta, \frac{\mu}{2L}\right\}\sum_{i=0}^{t^{(k)}-1} (1-\eta)\mathfrak{h}^{(k)}_i \right)\delta_{st^{(k)}} .
\end{equation}
Given the condition number $\kappa = L/\mu$ and fixing $\eta=1/2$, one can select the stepsize pattern giving the best contraction. Consider $k(\kappa) = \sup\{k \mid \Delta^{(k)} \geq \frac{1}{2\kappa}\}$. Supposing $k(\kappa)>0$, we have $\Delta^{(k(\kappa))}=c/d^{k(\kappa)}$, giving bounds of $\frac{1}{2\kappa} \leq \Delta^{(k(\kappa))} \leq \frac{d}{2\kappa}$. Noting $\sum_{i=0}^{t^{(k)}-1} \mathfrak{h}^{(k(\kappa))}_i \geq 2(1+\sqrt{2})^k$, one has $\sum_{i=0}^{t^{(k)}} \mathfrak{h}^{(k)}_i \geq 2(2c\kappa/d)^{\log_{d}(1+\sqrt{2})}$. Then the guarantee~\eqref{eq:contraction} gives a contraction factor after applying the pattern of stepsizes of
$1 - (2c\kappa/d)^{\log_{d}(1+\sqrt{2})}\kappa^{-1}.$
Since this contraction is attained after $t^{(k(\kappa))} = 2^{k(\kappa)+1}-1 \leq 2(2c\kappa/d)^{\log_{d}(2)}$ iterations, when amortized, the per-iteration contraction factor is
\begin{align*}
	(1 - (2c\kappa/d)^{\log_{d}(1+\sqrt{2})}\kappa^{-1})^{1/t^{(k(\kappa))}} &\leq  1 - 2^{-1}(2c\kappa/d)^{\log_{d}(1+\sqrt{2})}\kappa^{-1}(2c\kappa/d)^{-\log_{d}(2)}\\
	& \leq 1- 0.204652\times \kappa^{-0.94662}.
\end{align*}
This gives the following convergence theorem. Note this accelerated rate is slower than that concurrently developed by~\cite{altschuler2023acceleration}. If one could improve the value of $d$ above to $1+\sqrt{2}$, our rate would improve to match theirs. 
\begin{theorem}\label{thm:strongly-convex}
	For any target accuracy $\epsilon>0$ and $L$-smooth, $\mu$-strongly convex $f$, gradient descent repeating stepsizes $\mathfrak{h}^{(k(L/\mu))}$ has $x_T$ as an $\epsilon$-minimizer for all $T=st^{(k(L/\mu))}$ with $s\in\mathbb{N}$ and
	$$ T \geq 4.88633(L/\mu)^{0.94662}\log\left(\frac{f(x_0)-f(x_\star)}{\epsilon}\right). $$
\end{theorem}

\subsection{Tight Bounds on Straightforward Patterns} \label{subsec:tightness}

Some insight into the structure of straightforward stepsize sequences follows from considering particular ``bad'' problem instances. Since straightforward patterns always yield a descent, showing a failure to achieve a descent on any instance suffices to prove a pattern is not straightforward. Three elementary bounds of this form are given below.
\begin{proposition}\label{prop:product-rule}
	If $\mathfrak{h}\in\mathbb{R}^{t}$ is straightforward, then 
	$$-1 < \prod_{i=0}^{t-1} (\mathfrak{h}_i - 1) < 1 . $$
\end{proposition}
\begin{proof}
	Consider the one-dimensional objective function $f(x)=\frac{1}{2}x^2$ with $x_0=1$. Then each gradient descent step is $x_{k+1} = (1-h_k)x_k$. Hence $x_t = \prod_{i=0}^{t-1} (1-\mathfrak{h}_i)$. To achieve any descent by the end of this pattern, this must be between one and minus one.
\end{proof}
\begin{proposition}\label{prop:sum-rule}
	If $\mathfrak{h}\in\mathbb{R}^{t}$ is straightforward, then for all $i\in \{0,\dots,t-1\}$
	$$\mathfrak{h}_i < \sum_{j\neq i} \mathfrak{h}_j+2 .$$
\end{proposition}
\begin{proof}
	Consider the one-dimensional Huber objective function $f(x)=\frac{1}{2}x^2$ if $|x|\leq 1$ and $f(x)=|x|-\frac{1}{2}$ otherwise. Letting $x_0 = -\sum_{j<i}\mathfrak{h}_j -1$, one has $x_i = -1$. Supposing $\mathfrak{h}_i \geq  \sum_{j\neq i} \mathfrak{h}_j+2$, one has $x_{i+1} \geq \sum_{j\neq i}\mathfrak{h}_j +1$. Consequently, gradient descent failed to descend as $x_t \geq \sum_{j < i}\mathfrak{h}_j +1 $.
\end{proof}
\begin{proposition}\label{prop:mixed-rule}
	If $\mathfrak{h}\in\mathbb{R}^{t}$ is straightforward, then for all $i\in \{0,\dots,t-1\}$ and $i_+ = (i+1) \mathrm{\ mod\ } t$,
	$$\sum_{j\not\in \{i,i_+\}} \mathfrak{h}_j +1 > (1-\mathfrak{h}_i)(1-\mathfrak{h}_{i_+}).$$
\end{proposition}
\begin{proof}
    Without loss of generality, $i=0$ and $i_+=1$. Consider the one-dimensional objective function $\frac{1}{2}x^2+\frac{1}{2}$ for $x\leq 1$ and $x$ otherwise. Suppose $(1-\mathfrak{h}_i)(1-\mathfrak{h}_{i_+}) \geq \sum_{j\not\in \{i,i_+\}} \mathfrak{h}_j +1 \geq 0$. Letting $x_0=1$, we then have $x_1=(1-\mathfrak{h}_0)\leq 1$ and $x_2=(1-\mathfrak{h}_1)(1-\mathfrak{h}_0)$. By our assumption, we then have $x_2,\dots,x_t\geq 1$. So no descent is achieved by applying the pattern and hence it is not straightforward.
\end{proof}

The first two of these bounds on how large straightforward patterns are actually tight in all of our proposed building block patterns $\mathfrak{h}^{(k)}$. The selection of the middle stepsize $\mu_k$ is exactly the sum of all the other stepsizes plus two, matching Proposition~\ref{prop:sum-rule}. The selection of the one-quarter and three-quarters stepsize $\alpha_{k-2}$ is exactly the choice making the product in Proposition~\ref{prop:product-rule} equal one (Appendix~\ref{app:product-tightness} verifies this).

Arguments based on these bounds can show that our first two building block patterns have the longest total length possible among all straightforward patterns. We believe this is true for all of our proposed patterns but only provide proof for the global maximality for $k=0,1$.
\begin{theorem}
	For $t=1$, no straightforward pattern has $\sum_{i=0}^{t-1}\mathfrak{h_i}\geq 2$.\\ (Hence the rescaled patterns $(1-\eta)\mathfrak{h}^{(0)}$ approach having maximum length.)
\end{theorem}
\begin{proof}
	This is an immediate consequence of Proposition~\ref{prop:sum-rule} with $i=0$.
\end{proof}
\begin{theorem}
	For $t=3$, no symmetric straightforward pattern has $\sum_{i=0}^{t-1}\mathfrak{h_i}\geq 8$.\\ (Hence the rescaled patterns $(1-\eta)\mathfrak{h}^{(1)}$ approach having maximum length.)
\end{theorem}
\begin{proof}
Maximizing $\sum_{i=0}^2 h_i$ over the region constrained by the inequalities of Propositions~\ref{prop:product-rule}, \ref{prop:sum-rule}, and~\ref{prop:mixed-rule} ensures no pattern $h=(a,b,a)$ has $2a+b \geq 8$ (See \nextmathematica).
\end{proof}

\section{A Spectral Certificate for Straightforwardness}\label{sec:certificate_part_1}

All that remains to complete our analysis is the proof of Lemma~\ref{lem:straightforward}. The remainder of this paper completes the substantial technical work needed to verify this. As an overview, the main result in this section (\cref{thm:straightforward_from_rank_one}) gives a sufficient condition for straightforwardness that is more practical to verify than the nonemptiness of $\mathcal{S}_{\mathfrak{h},\Delta}$. Then Section~\ref{sec:certificate} constructs certificates $\lambda^{(k)}$ showing that the sufficient conditions of \cref{thm:straightforward_from_rank_one} are met for each $\mathfrak{h}^{(k)}$. Finally, Sections~\ref{sec:lambdadRowColSum} and~\ref{sec:lambdaM} do the heavy algebraic work explicitly verifying these sufficient conditions are met.

This section's main result, \cref{thm:straightforward_from_rank_one}, considers patterns $\mathfrak{h}$ that may themselves not be straightforward, but for which $(1-\eta)\mathfrak{h}$ is straightforward for any $\eta\in(0,1)$.
Below, we show the set of straightforward stepsize patterns is star-convex with respect to the all zero stepsize pattern, justifying the search for such a rescaling theorem. \begin{proposition}
Suppose $\cS_{\mathfrak{h}, \Delta}$ is nonempty and $\Delta\geq 0$. Then, $\cS_{\theta \mathfrak{h}, \omega \Delta}$ is nonempty for all $\theta\in(0,1]$ and $\omega\in[0,1]$.
\end{proposition}
\begin{proof}
    It suffices to prove the proposition with $\omega=1$ as the constraints defining $\cS_{\theta \mathfrak h, \omega\Delta}$ only relax as $\omega$ decreases.
    The case $\theta=1$ follows by definition.
    Let $\theta\in(0,1)$ and let $(\lambda,\gamma)\in \cS_{\mathfrak{h},\Delta}$. We claim that $(\lambda, \theta \gamma) \in \cS_{\theta  \mathfrak{h},\Delta/\theta }$.
    The first four constraints defining $\cS_{\theta \mathfrak{h}, \Delta/\theta }$ hold.
    We will need the following fact to verify the remaining constraints: for any fixed $\xi\geq 0$, $\theta \mapsto M(\xi,\theta \mathfrak{h})$ is an affine function in $\theta $ with a PSD constant term, $M(\xi, 0)$.
    Then, 
    \begin{align*}
        \begin{pmatrix}
        \theta \sum_i \mathfrak{h}_i
         & \theta m(\gamma)^\intercal\\
        \theta m(\gamma) & M(\lambda, \theta  \mathfrak{h})
        \end{pmatrix} = \begin{pmatrix}
            \theta \sum \mathfrak{h}_i & \theta  m(\gamma)^\intercal\\
            \theta m(\gamma) & \theta  M(\lambda, \mathfrak{h})+ (1-\theta )M(\lambda,0)
            \end{pmatrix}\succeq 0.
    \end{align*}
    The other constraint holds similarly. We deduce that $(\lambda,\theta\gamma)\in \cS_{\theta  \mathfrak{h}, \Delta/\theta }\subseteq \cS_{\theta \mathfrak{h},\Delta}$. 
\end{proof}
This rescaling result motivates the following theorem capable of certifying that a pattern $\mathfrak{h}$ has $(1-\eta)\mathfrak{h}$ straightforward for all $\eta\in (0,1)$. This result is especially useful when compared to Theorem 3.2 of~\cite{Grimmer2023-long} as only $\lambda\in\mathbb{R}^{(t+2)\times(t+2)}$ needs to be provided; the existence of $\gamma$ and $\Delta>0$ for which $(\lambda,(1-\eta)\gamma)\in\mathcal{S}_{(1-\eta)\mathfrak{h},\Delta}$ can be deduced from $\lambda$. We also provide quantitative bounds on $\Delta$ in terms of a matrix $W_2(\lambda)$ that will be defined in the following subsection.

\begin{theorem}
    \label{thm:straightforward_from_rank_one}
    Suppose $\mathfrak{h}\in\R^{t}$, $\lambda\in\R^{(t+2)\by(t+2)}$ satisfy the following properties:
    \begin{align*}
        \begin{cases}
            H\coloneqq \sum_{i=0}^{t-1} \mathfrak{h}_i >0\\
            \lambda \text{ is zero on its } \star\text{th row and column}\\
            \lambda\geq 0\\
            \lambda_{i,i+1}>0,\,\forall i = 0,\dots,t-1\\
            \sum_{\substack{i,j=0\\ i\neq j}}^t\lambda_{i,j}a_{i,j} = a_{\star,t} - a_{\star,0}\\
            M(\lambda,\mathfrak{h}) \text{ is nonnegative and rank-one}.
        \end{cases}
    \end{align*}
    Then, there exists $\gamma\in\R^{(t+2)\by(t+2)}$ so that for all $\eta\in (0,1)$ there exists $\Delta>0$ so that $(\lambda,(1-\eta)\gamma)\in\cS_{(1-\eta)\mathfrak{h},\Delta}$.
    Here, $\Delta>0$ may depend on $\eta$.

    If we additionally fix $\eta=1/2$ and assume that $t\geq 3$, $H\geq 8$, and $\mathfrak{L}\in(0,1]$ is a lower-bound on the second-smallest eigenvalue of $W_2(\lambda)$, then we may take any $\Delta>0$ satisfying
    \begin{align*}
        \Delta \leq \min\left(\frac{\min_{i\in[0,t-1]}\lambda_{i,i+1}}{H},\, \frac{\mathfrak{L}}{21H^3}\right).
    \end{align*}
    \end{theorem}
    
    \subsection{Proof of \cref{thm:straightforward_from_rank_one}}
    We divide our proof into three parts. First in Section~\ref{subsubsec:gamma-construction}, we construct $\gamma$ from $\lambda$ satisfying the needed linear equality constraint. Then Section~\ref{subsubsec:Delta-positive} shows a positive $\Delta>0$ exists with $(\lambda,(1-\eta)\gamma)\in\cS_{(1-\eta)\mathfrak{h},\Delta}$. Finally, Section~\ref{subsubsec:Delta-bound} improves this, providing a quantitative lower bound on $\Delta$.
    
    To prove \cref{thm:straightforward_from_rank_one}, we require some additional notation. Let
    \begin{align*}
        W_1(\lambda, \mathfrak{h}) &\coloneqq \sum_{i\neq j} \lambda_{i,j} A_{i,j}(\mathfrak{h})_{0:t,0:t},\\
        W_2(\lambda) &\coloneqq \sum_{i\neq j} \frac{\lambda_{i,j}}{2} (C_{i,j})_{0:t,0:t}.
    \end{align*}
    With this notation, we may decompose $M(\lambda,\mathfrak{h}) = W_1(\lambda, \mathfrak{h}) + W_2(\lambda)$.
    Note that $W_1$ is bilinear in its arguments and $W_2$ is linear in $\lambda$.

    \subsubsection{Construction of Associated $\gamma$}\label{subsubsec:gamma-construction}
As $M(\lambda,\mathfrak{h})$ is both rank-one and nonnegative, we may write $M(\lambda,\mathfrak{h})=\frac{\phi\phi^\intercal}{2}$ where $\phi\in\R^{t+1}_+$ is indexed by $[0,t]$. 
    We will construct $\gamma\in\R^{(t+2)\by(t+2)}$ from $\phi$ as follows.
    First, set $\gamma_{\star,i} = \sqrt{2H}\phi_i$ for all $i = 0,\dots,t$.
    Then, for $i = 0,\dots,t-1$, set $\gamma_{i,i+1} = \sum_{j=0}^i \gamma_{\star,j} - 2H$.
    All other entries in $\gamma$ are set as zero.

    Below, we verify that the second constraint defining $\mathcal{S}_{\mathfrak{h},\Delta}$, i.e., the linear constraint on $\gamma$, is satisfied for our construction.
    The lemma below explains what the first two linear constraints in the definition of $\mathcal{S}_{\mathfrak{h},\Delta}$ require of $\lambda$ and $\gamma$. Its proof is immediate from expanding the definition of $a_{i,j}$.
    \begin{lemma}
    \label{lem:lin_constraints}
    The equation $\sum_{i,j=0}^t\lambda_{i,j}a_{i,j} = a_{\star,t} - a_{\star,0}$ holds if and only if
    \begin{itemize}
        \item The sum of the zeroth row of $\lambda$ is one larger than the sum of the zeroth column of $\lambda$.
        \item For $i = 1,\dots,t-1$, the sum of the $i^{\mathrm{th}}$ row of $\lambda$ is equal to the sum of the $i^{\mathrm{th}}$ column of $\lambda$.
        \item The sum of the $t^{\mathrm{th}}$ row of $\lambda$ is one less than the sum of the $t^{\mathrm{th}}$ column of $\lambda$.
    \end{itemize}
    Suppose the support of $\gamma$ is contained in $\set{(\star,i):\, i\in[0,t]}\cup\set{(i,i+1):\, i\in[0,t-1]}$, then the equation $\sum_{i\neq j} \gamma_{i,j}a_{i,j} = 2 \sum_{i=0}^{t-1}\mathfrak{h}_ia_{\star,0}$ holds if and only if
    \begin{align}
        \label{eq:gamma_lin_constraints}
        \begin{cases}
            \gamma_{\star,0} - \gamma_{0,1} - 2H = 0\\
            \gamma_{\star,i} - \gamma_{i,i+1} + \gamma_{i-1,i} = 0,\,\forall i = 1,\dots,t-1\\
            \gamma_{\star,t} + \gamma_{t-1,t} = 0.
        \end{cases}
    \end{align}
    \end{lemma}

    Comparing \cref{lem:lin_constraints} with our construction of $\gamma$, we see that the first $t$ constraints in \eqref{eq:gamma_lin_constraints} are satisfied. To show that the last constraint is satisfied as well, it is enough to show that
    the sum of all left-hand side expressions in \eqref{eq:gamma_lin_constraints} is zero, i.e., $2H = \sum_{i=0}^t \gamma_{\star,i} = \sqrt{2H} \mb 1^\intercal \phi$. This is done in the following lemma.
    \begin{lemma}
    \label{lem:z_l1}
    It holds that $\mb 1^\intercal \phi = \sqrt{2H}$.
    \end{lemma}
    \begin{proof}
    We compute $\tfrac{1}{2}(\mb 1^\intercal \phi)^2 = \mb 1^\intercal M(\lambda,\mathfrak{h}) \mb 1
        = \mb 1^\intercal(W_1(\lambda, \mathfrak{h}) + W_2(\lambda))\mb 1
        =\mb 1^\intercal W_1(\lambda, \mathfrak{h})\mb 1$
    where the last equality follows as $\lambda$ is zero on its $\star^{\mathrm{th}}$ row and column so that $\mb 1 \in\ker\left((C_{i,j})_{0:t,0:t}\right)$ for any $(i,j)$ satisfying $\lambda_{i,j}>0$. We continue by writing out the definition of $W_1(\lambda,\mathfrak{h})$ and rearranging indices of the resulting summation:
    \begin{align*}
        \frac{1}{2}(\mb 1^\intercal \phi)^2 &= \sum_{\substack{i,j=0\\i\neq j}}^{t}\lambda_{i,j} \mb 1^\intercal (A_{i,j}(\mathfrak{h}))_{2:t+2,2:t+2} \mb 1\\
        &= \sum_{\substack{i,j=0\\i<j}}^{t}\sum_{\ell=i}^{j-1}\lambda_{i,j}  \mathfrak{h}_\ell - \sum_{\substack{i,j=0\\i>j}}^{t}\sum_{\ell=j}^{i-1}\lambda_{i,j}  \mathfrak{h}_\ell\\
        &= \sum_{\ell=0}^{t-1}\left(\sum_{i=0}^\ell\sum_{j=\ell+1}^{t}\lambda_{i,j} - \sum_{i=\ell+1}^{t}\sum_{j=0}^\ell\lambda_{i,j}\right) \mathfrak{h}_\ell.
    \end{align*}
    Now, we claim that for each $\ell \in[0,t-1]$, that the term in parentheses above is exactly equal to one. Indeed, the difference of the two sums is equal to the difference between then sum of all entries in rows $0$ through $\ell$ with the sum of all entries in columns $0$ through $\ell$. Thus, \cref{lem:lin_constraints} states that the term in parentheses is exactly equal to one so that $1^\intercal \phi = \sqrt{2H}$.
    \end{proof}
    
\subsubsection{Existence of Associated $\Delta>0$}\label{subsubsec:Delta-positive}
The first three defining constraints of $\cS_{(1-\eta)\mathfrak{h},\Delta}$ are satisfied regardless of $\Delta$ and $\eta$. Similarly, $\lambda\geq 0$ does not depend on $\Delta$ or $\eta$.
    Next, $\lambda + \Delta(1-\eta)\gamma$ is nonnegative for all $\Delta>0$ small enough by the assumption that $\lambda_{i,i+1}>0$ for all $i\in[0,t-1]$ and the observation that those are the only negative entries of $\gamma$.
    It remains to check that the two PSD constraints defining $\cS_{(1-\eta)\mathfrak{h},\Delta}$ hold for all $\Delta>0$ small enough.
    
    For the first one, we compute
    \begin{align*}
        \begin{pmatrix}
        (1-\eta) H &  (1-\eta) m(\gamma)^\intercal\\
        (1-\eta)m(\gamma) & W_1(\lambda,(1-\eta)\mathfrak{h}) + W_2(\lambda)
        \end{pmatrix}&=
        (1-\eta)\mathfrak{A} + \eta \mathfrak{B},
    \end{align*}
    where
    \[
        \mathfrak{A} = 
        \begin{pmatrix}
            H &  m(\gamma)^\intercal\\
            m(\gamma) & W_1(\lambda,\mathfrak{h}) + W_2(\lambda)
            \end{pmatrix} = \begin{pmatrix}
            H &  -\sqrt{H/2}\,\phi^\intercal\\
            -\sqrt{H/2}\,\phi & \frac{\phi\phi^\intercal}{2}
            \end{pmatrix}, 
    \]
    and 
    \[
        \mathfrak{B} = \begin{pmatrix} 0 &\\ & W_2(\lambda) \end{pmatrix}.
    \]
    $\mathfrak{A}$ is PSD by construction (in fact, rank-one).
    $\mathfrak{B}$ is PSD as $W_2$ maps nonnegative matrices to PSD matrices. We deduce that the first PSD constraint in the definition of $\cS_{(1-\eta)\mathfrak{h},\Delta}$ holds regardless of $\Delta$ and $\eta$.
    
    For the second constraint, we compute
    \begin{align*}
        &\begin{pmatrix}
        (1-\eta) H & (1-\eta)m(\gamma)^\intercal\\
        (1-\eta)m(\gamma)&  W_1(\lambda + \Delta(1-\eta)\gamma,(1-\eta)\mathfrak{h}) + W_2(\lambda + \Delta(1-\eta)\gamma)
        \end{pmatrix}\\
        &\qquad= (1-\eta)\begin{pmatrix}
        H & m(\gamma)^\intercal\\
        m(\gamma) & W_1(\lambda,\mathfrak{h}) + W_2(\lambda)
        \end{pmatrix}
        + \eta \begin{pmatrix}
        0&\\&W_2(\lambda)
        \end{pmatrix}\\
        &\qquad\qquad +
        \Delta\begin{pmatrix}
        0 & \\ & (1-\eta)^2 W_1(\gamma,\mathfrak{h}) + (1-\eta) W_2(\gamma)
        \end{pmatrix}.
    \end{align*}
    For convenience, let $(1-\eta)Q_1 + \eta Q_2 + \Delta Q_3$ denote the three parts of the right-hand-side expression.
    First, we claim that for any $\eta\in(0,1)$, the expression $(1-\eta)Q_1 + \eta Q_2$ has rank $t+1$.
    To see this, recall that $\lambda_{i,i+1}>0$ for all $i\in[0,t-1]$ so that $W_2(\lambda) \succeq L$, where $L$ is a positive multiple of the Laplacian of a path graph on $t+1$ vertices and has rank $t$.
    Then, the claim follows as $Q_1$ has rank one with nontrivial support on its first column and $Q_2$ has rank $t$ on the subspace orthogonal to the first coordinate.
    We additionally deduce that $\ker((1-\eta)Q_1 + \eta Q_2) = \R\mb 1$.
    
    We next evaluate $\mb 1^\intercal W_1(\gamma,\mathfrak{h})\mb 1$ and $\mb 1^\intercal W_2(\gamma)\mb 1$. For the first expression,
    \begin{align*}
        \mb 1^\intercal W_1(\gamma,\mathfrak{h}) \mb 1 &= \sum_{i=0}^t \gamma_{\star,i}\left(\sum_{j=0}^{i-1} \mathfrak{h}_j\right) + \sum_{i=0}^{t-1} \gamma_{i,i+1} \mathfrak{h}_i\\
        &= \sum_{i=0}^t\sum_{j=0}^{i-1} \gamma_{\star,i} \mathfrak{h}_j
        + \sum_{j=0}^{t-1}\sum_{i=0}^j \gamma_{\star,i} \mathfrak{h}_j - 2H^2\\
        &= \left(\sum_{i=0}^t \gamma_{\star,i}\right)\left(\sum_{j=0}^{t-1}\mathfrak{h}_j\right) - 2H^2
        = 0.
    \end{align*}
    For the second expression, we have
    $\mb 1^\intercal W_2(\gamma) \mb 1 = \frac{1}{2}\sum_{i=0}^t \gamma_{\star,i} = H$.

    We deduce that $\mb 1^\intercal Q_3\mb 1 = (1-\eta)H >0$, or equivalently, the matrix $Q_3$ has a positive component in the kernel of $(1-\eta)Q_1 + \eta Q_2$. 
    Thus, the Schur complement lemma shows the existence of a positive $\Delta$ satisfying the theorem statement.

    \subsubsection{A Quantitative Lower Bound on $\Delta$}\label{subsubsec:Delta-bound}
    For the second part of the theorem statement, we will assume $t\geq 3$, $H\geq 8$, $\eta = 1/2$.
    Additionally, we will assume that $\mathfrak{L}\in(0,1]$ lower bounds the second-smallest eigenvalue of $W_2(\lambda)$.

    We will now repeat portions of the proof of the first claim more quantitatively to derive explicit bounds on $\Delta$.
    By the above arguments, it suffices to pick $\Delta$ so that $\lambda+\frac{\Delta}{2}\gamma\geq 0$ and
    $\frac{1}{2}Q_1+\frac{1}{2}Q_2+\Delta Q_3\succeq0$.
    
    First, as the superdiagonal entries of $\gamma$ are defined as $\gamma_{i,i+1}=\sum_{j=0}^i\gamma_{\star,j}-2H = \sqrt{2H}\sum_{j=0}^i\phi_i - 2H$, we have that
    each of these entries is bounded in magnitude by $2H$ (see \cref{lem:z_l1}).
    In particular, the requirement that $\lambda+\frac{\Delta}{2}\gamma\geq 0$ is satisfied as long as
    \begin{align*}
        \Delta \leq \frac{\min_{i\in[0,t-1]}\lambda_{i,i+1}}{H}.
    \end{align*}
    This is the first term in our bound on $\Delta$.

    We now turn to the constraint $\frac{1}{2}Q_1 + \frac{1}{2}Q_2 + \Delta Q_3\succeq 0$.

    Fix an ordered orthonormal basis of $\R^{t+2}$ of the form
    \[
        \set{\tfrac{(t+1, -\mb1_{t+1})}{\sqrt{(t+1)^2 + (t+1)}}, \frac{\xi}{\|\xi\|_2}, \dots, \tfrac{\mb1_{t+2}}{\sqrt{t+2}}},
    \]
    where $\xi$ is the projection of $(\sqrt{H},-\phi/\sqrt{2})$ onto the orthogonal complement of $\tfrac{(t+1, -\mb1_{t+1})}{\sqrt{(t+1)^2 + (t+1)}}$.
    To see that this is possible, note that $\tfrac{(t+1, -\mb1_{t+1})}{\sqrt{(t+1)^2 + (t+1)}},$ $\frac{\xi}{\|\xi\|_2}$ and $\tfrac{\mb1_{t+2}}{\sqrt{t+2}}$ are orthogonal and unit norm.
    Note that $(\sqrt{H},-\phi/\sqrt{2})$ is in the span of the first two basis vectors.
    Also note that the first and last vectors in this basis span the kernel of $Q_2$.
    
    We define and bound $\zeta$ as follows:
    \begin{align*}
        \zeta \coloneqq \frac{1}{\sqrt{(t+1)^2 + (t+1)}}\begin{pmatrix}
            t+1\\
            -\mb1_{t+1}
        \end{pmatrix} ^\intercal \begin{pmatrix}
        \sqrt{H}\\-\phi/\sqrt{2}
        \end{pmatrix} &= \frac{\sqrt{t+2}\sqrt{H}}{\sqrt{t+1}}
        \geq \sqrt{H},
    \end{align*}
    where the last inequality follows from our assumption $t\geq 3$.

    Abusing notation, we will also write $Q_i$ to denote the matrix $Q_i$ written in this new basis.
    Then,
$\frac{1}{2}Q_1+\frac{1}{2}Q_2$ can be bounded below by 
    \begin{align*}
        \frac{1}{2}Q_1 + \frac{1}{2}Q_2\succeq \frac{1}{2}\begin{pmatrix}
        \zeta^2  & \zeta\norm{\xi}_2\\
        \zeta\norm{\xi}_2 & \norm{\xi}^2_2 \\
        && 0_{(t-1)\by(t-1)}\\
        & & & 0 
        \end{pmatrix} + \frac{\mathfrak{L}}{2}\begin{pmatrix}
        0 &\\
        & 1 \\ && I_{(t-1)\by(t-1)}\\
        &&& 0
        \end{pmatrix}.
    \end{align*}

    We bound the minimum eigenvalue of the top-left two-by-two submatrix here as the determinant divided by the trace of the submatrix:
    \begin{align*}
        \frac{\zeta^2\mathfrak{L}/2}{\zeta^2 + \norm{\xi}_2^2 + \mathfrak{L}} &= \frac{\zeta^2\mathfrak{L}/2}{(H + \norm{\phi}^2_2/2)  + \mathfrak{L}}\\
        &\geq \frac{H\mathfrak{L}}{2(2H + \mathfrak{L})}\\
        &\geq \frac{4\mathfrak{L}}{17}.
    \end{align*}
    Here, the first line follows as $\zeta^2+\norm{\xi}_2^2$ and $H+\norm{\phi}_2^2/2$ both measure the $\ell_2$-norm of $(\sqrt{H},\phi/\sqrt{2})$. The second line follows from our lower bound on $\zeta$ and \cref{lem:z_l1} (implying $\norm{\phi}_2^2\leq \norm{\phi}_1^2 = 2H$). The final line uses our bounds on $\mathfrak{L}$ and $H$ which we assumed at the start of this argument.
    It is also clear that the middle $(t-1)\by(t-1)$ block has minimum eigenvalue $\mathfrak{L}/2$.

    Plugging back into our lower bound on $\frac{1}{2}Q_1 + \frac{1}{2}Q_2$ gives
    \begin{align*}
        \frac{1}{2}Q_1 + \frac{1}{2}Q_2 &\succeq \frac{4\mathfrak{L}}{17}\begin{pmatrix}
            I_{(t+1)\times(t+1)}\\
            & 0
            \end{pmatrix}.
    \end{align*}
    We may also write $Q_3$ in this basis in block form by first letting $q_3$ be the orthogonal projection of $Q_3 \tfrac{\mb1_{t+2}}{\sqrt{t+2}}$ onto the subspace orthogonal to $\mb 1_{t+1}$ and writing
    \[
        Q_3 = \begin{pmatrix}
            Q_3' & q_3\\
            q_3^\intercal &  \frac{1}{t+2}\mb1_{t+2}^{\intercal}Q_3\mb1_{t+2}
        \end{pmatrix}.
    \]
    Note that $\mb1_{t+2}^{\intercal}Q_3\mb1_{t+2} = \frac{1}{2}H$ from our previous section.

    We may bound
    \[
        \frac{1}{2}Q_1 + \frac{1}{2}Q_2 + \Delta Q_3 \succeq \begin{pmatrix}
            \frac{4\mathfrak{L}}{17}I_{(t+1)\times(t+1)} + \Delta Q_3' & \Delta q_3\\
            \Delta q_3^\intercal & \frac{\Delta H}{2(t+2)}
        \end{pmatrix}.
    \]
    We now apply the Schur complement lemma to the bottom right entry of this second matrix, which yields that this matrix is PSD if and only if
    \[
        \frac{4\mathfrak{L}}{17}I_{(t+1)\times(t+1)} + \Delta \left(Q_3' - \frac{2(t+2)}{H}  q_3q_3^{\intercal}\right) \succeq 0.
    \]

    This is PSD if
    \[
        \Delta \le  \frac{4\mathfrak{L}}{17\norm{Q_3' - \frac{2(t+2)}{H}  q_3q_3^{\intercal}}_\textup{op} }.
    \]

    It remains to give an upper bound on
    \[
        \norm{Q_3' - \frac{2(t+2)}{H}  q_3q_3^{\intercal}}_\textup{op} \le 
        \norm{Q_3}_\textup{op} + \frac{2(t+2)}{H}\norm{q_3q_3^{\intercal}}_\textup{op} \leq 
        \norm{Q_3}_\textup{op} + \frac{2}{H}\norm{Q_3 \mb1_{t+2}}_2^2.
    \]
    The second inequality follows because $q_3$ is the orthogonal projection of $Q_3 \tfrac{\mb1_{t+2}}{\sqrt{t+2}}$ onto a subspace.
    
    We now bound $\norm{Q_3}_\textup{op}$ and $\norm{Q_3\mb1_{t+2}}_2$ separately.

    For this, we apply the triangle inequality to break the summation over all entries in $\gamma$ into summations over just the $\star$th row and the first superdiagonal:
    \begin{align*}
        \norm{Q_3}_\textup{op} &= \norm{W_1(\gamma,\mathfrak{h})/4 + W_2(\gamma)/2}_\textup{op}\\
        &\leq \frac{1}{4}\norm{\sum_{i=0}^{t}\gamma_{\star,i}\left((A_{\star,i}(\mathfrak{h}))_{0:t,0:t} + (C_{\star,i})_{0:t,0:t}\right)}_\textup{op}\\
        &\qquad + \frac{1}{4}\norm{\sum_{i=0}^{t-1} \gamma_{i,i+1}\left((A_{i,i+1}(\mathfrak{h}))_{0:t,0:t}+(C_{i,i+1})_{0:t,0:t}\right)}_\textup{op}
    \end{align*}
    and
    \begin{align*}
        \norm{Q_3 \mb 1_{t+2}}_2 &\leq \norm{Q_3 \mb 1_{t+2}}_1\\
        &\leq \frac{1}{4}\norm{\sum_{i=0}^{t}\gamma_{\star,i}\left((A_{\star,i}(\mathfrak{h}))_{0:t,0:t} + (C_{\star,i})_{0:t,0:t}\right)\mb 1_{t+1}}_1\\
        &\qquad + \frac{1}{4}\norm{\sum_{i=0}^{t-1} \gamma_{i,i+1}\left((A_{i,i+1}(\mathfrak{h}))_{0:t,0:t}+(C_{i,i+1})_{0:t,0:t}\right)\mb 1_{t+1}}_1.
    \end{align*}
    In both cases, we will deal with the first term using an $\ell_1$-$\ell_\infty$ bound: note that $\sum_{i=0}^t \abs{\gamma_{\star,i}} = 2H$. On the other hand, for any $i\in[0,t]$, we have $\norm{(A_{\star,i}(\mathfrak{h}))_{0:t,0:t} + (C_{\star,i})_{0:t,0:t}}_\textup{op} \leq 1+\norm{\mathfrak{h}}_2 \leq 2H$ and $\norm{\left((A_{\star,i}(\mathfrak{h}))_{0:t,0:t} + (C_{\star,i})_{0:t,0:t}\right)\mb 1_{t+1}}_\textup{1} \leq 1 + H \leq 2H$.
    
    For the second term in the bound on $\norm{Q_3}_\textup{op}$, note that for all $i \in[0,t-1]$, we have
    \begin{gather*}
        (A_{i,i+1}(\mathfrak{h}))_{0:t,0:t} = \mathfrak{h}_i (e_{i+1} \odot e_i)\\
        (C_{i,i+1})_{0:t,0:t} = (e_i - e_{i+1}) \odot (e_i - e_{i+1}).
    \end{gather*}
    We deduce that
    \begin{align*}
        \sum_{i=0}^{t-1} \gamma_{i,i+1}\left((A_{i,i+1}(\mathfrak{h}))_{0:t,0:t}+(C_{i,i+1})_{0:t,0:t}\right)
    \end{align*}
    is a tridiagonal matrix. Using the bounds $\abs{\gamma_{i,i+1}}\leq 2H$ and $\mathfrak{h}_i \leq H$, we have that all entries in this tridiagonal matrix have magnitude bounded by $2(2H)(1+H/2) \leq 3H^2$. We may bound the operator norm of this matrix as the sum of the operator norms of each diagonal. Thus,
    \begin{align*}
        \norm{\sum_{i=0}^{t-1} \gamma_{i,i+1}\left((A_{i,i+1}(\mathfrak{h}))_{0:t,0:t}+(C_{i,i+1})_{0:t,0:t}\right)}_\textup{op}\leq 9H^2
    \end{align*}

    For the second term in our bound on $\norm{Q_3 \mb 1_{t+2}}_1$, a direct calculation shows
    \begin{align*}
        \norm{\sum_{i=0}^{t-1} \gamma_{i,i+1}\left((A_{i,i+1}(\mathfrak{h}))_{0:t,0:t}+(C_{i,i+1})_{0:t,0:t}\right)\mb 1_{t+1}}_1 &=
        \frac{1}{2}\norm{\sum_{i=0}^{t-1} \gamma_{i,i+1}h_i(e_i + e_{i+1})}_1.
    \end{align*}
    Thus, we may bound this quantity above by $2H^2$.

    Putting all the pieces together gives
    \begin{align*}
        \norm{Q_3}_\textup{op} \leq \frac{13}{4}H^2\\
        \norm{Q_3 \mb 1_{t+2}}_2 \leq \frac{3}{2}H^2
    \end{align*}

    This implies that our final expression is PSD as long as
    \[
        \Delta \le  \frac{4\mathfrak{L}}{17\left(\frac{13}{4}H^2 + \frac{9}{2}H^3\right).
}.
    \]
    Invoking our bound on $H$, we have that $(\lambda, \gamma/2)\in \mathcal{S}_{\mathfrak{h}/2,\Delta}$ for any
    \begin{align*}
        \Delta\leq \frac{\mathfrak{L}}{21H^3}.
    \end{align*}

     \section{Proof of Lemma~\ref{lem:straightforward} and Construction of Certificates}\label{sec:certificate}
First for $k=0$, we address the straightforwardness of $h^{(0)}/2 = [1]$ with parameter $\Delta^{(0)}=1/2$ individually. One can verify this by noting the below values of $(\lambda,\gamma)$ are a member of $\mathcal{S}_{[1], 1/2}$ (see \nextmathematica):
$$ \lambda = 
    \begin{pmatrix}
            0 & 0 & 0 \\
            0 & 0 & 1 \\
            0 & 0 & 0 \\
    \end{pmatrix}, \qquad \gamma = 
    \begin{pmatrix}
            0 & 1 & 1 \\
            0 & 0 & -1 \\
            0 & 0 & 0 \\
    \end{pmatrix}.$$
For each $k\geq 1$, in light of \Cref{thm:straightforward_from_rank_one}, Lemma~\ref{lem:straightforward} follows if we can demonstrate certificates $\lambda^{(k)}$ satisfying the sufficient conditions on straightforwardness therein for each $\mathfrak{h}^{(k)}$. We do this in four parts. First, we provide a construction for our claimed certificates $\lambda^{(k)}$, which satisfy the first three conditions in \Cref{thm:straightforward_from_rank_one}. Next,
we will derive lower bounds on $\lambda_{i,i+1}$
and the second-smallest eigenvalue of $W_2(\lambda)$:
\begin{gather}
    \min_{i\in[0,2^{k+1}-2]}\lambda^{(k)}_{i,i+1} \geq 
    \frac{2-\sqrt{2}}{8\sqrt{2}}(1+\sqrt{2})^{-2k+1}\label{eq:bounds_on_lk_superdiagonal}\\
    \text{the second-smallest eigenvalue of }W_2(\lambda^{(k)}) \geq \frac{1}{286}(1+\sqrt{2})^{-k}.\label{eq:bounds_on_lk_lambda2}
\end{gather}
The first bound simply requires checking the relevant entries in our construction of $\lambda^{(k)}$ and is deferred to Appendix~\ref{app:Superdiagonal-bound}.
The second bound is proved in \cref{subsec:lambda2_bound}.
Finally, we verify the last two conditions in \cref{thm:straightforward_from_rank_one}, i.e., that $\sum_{i\neq j}\lk_{i,j}a_{i,j} = a_{\star,2^{k+1}-1} - a_{\star,0}$ and that $M(\lambda,\mathfrak{h})$ is a nonnegative rank-one matrix. For the second claim, we will explicitly show that $M(\lambda^{(k)},\mathfrak{h}^{(k)}) = \frac{\phi^{(k)}(\phi^{(k)})^{\intercal}}{2}$ where $\phi^{(k)}\in\mathbb{R}^{2^{k+1}}$ is defined as $\phi^{(0)} = \left[\sqrt{\mu_0 - 1}, w_0\right]$ and
\begin{align}
    \phi^{(k)} = \left[0_{2^k - 1}, \sqrt{\mu_k - 1}, w_k\right]\qquad\forall k\geq 1.\label{eq:phiDef}
\end{align}
Here, $w_k$ is a vector that will be constructed shortly.
The proofs of these two facts are relatively tedious but ultimately straightforward. We defer the proofs of these two statements to \cref{sec:lambdadRowColSum,sec:lambdaM}.

The proof of \cref{lem:straightforward} will then be a direct application of \cref{thm:straightforward_from_rank_one} and the bounds stated in \eqref{eq:bounds_on_lk_superdiagonal} and \eqref{eq:bounds_on_lk_lambda2}.

The remainder of this section provides our construction for the certificates $\lambda^{(k)}$ and verifies the lower bound on the second smallest eigenvalue of $W_2(\lambda)$.

\subsection{Example Certificates $\lambda^{(k)}$}
Although our construction of $\lambda^{(k)}$ only uses elementary arithmetic operations, it is still quite complicated.
We first present as examples our certificates $\lambda^{(k)}$ for $k = 1,2$. Together with Theorem~\ref{thm:straightforward_from_rank_one}, this proves the straightforwardness of each $(1-\eta)\mathfrak{h}^{(k)}$ for $k=1,2$.

    \[
        \lambda^{(1)} = 
        \begin{pmatrix}
            0 & 0 & 0 & 0 & 0 \\
            0 & 0 & 2 & 0 & 0 \\
            0 & 1 & 0 & \frac{1}{2} & \frac{1}{2} \\
            0 & 0 & 0 & 0 & \frac{1}{2} \\
            0 & 0 & 0 & 0 & 0 \\
        \end{pmatrix},
    \]
    
    \[
        \lambda^{(2)} = 
        \begin{pmatrix}
            0 & 0 & 0 & 0 & 0 & 0 & 0 & 0 & 0 \\
            0 & 0 & 2 & 0 & 0 & 0 & 0 & 0 & 0 \\
            0 & 1 & 0 & \sqrt{2} & 1 & 0 & 0 & 0 & 0 \\
            0 & 0 & 0 & 0 & 2 \left(\sqrt{2}+1\right) & 0 & 0 & 0 & 0 \\
            0 & 0 & \sqrt{2} & \sqrt{2}+2 & 0 & \frac{3}{\sqrt{2}}-2 & 1-\frac{1}{\sqrt{2}} & 1-\frac{1}{\sqrt{2}} & 1-\frac{1}{\sqrt{2}} \\
            0 & 0 & 0 & 0 & 0 & 0 & \frac{1}{\sqrt{2}}-\frac{1}{2} & 0 & 0 \\
            0 & 0 & 0 & 0 & 0 & \frac{3}{2}-\sqrt{2} & 0 & \frac{1}{\sqrt{2}}-\frac{1}{2} & \frac{1}{\sqrt{2}}-\frac{1}{2} \\
            0 & 0 & 0 & 0 & 0 & 0 & 0 & 0 & \frac{1}{2} \\
            0 & 0 & 0 & 0 & 0 & 0 & 0 & 0 & 0 \\
        \end{pmatrix}.
    \]
\subsection{Preliminary Definitions}
We will begin with some auxilliary definitions that will be used in our definition of $\lambda$.
For $i\geq 0$, we let $p(i)$ denote the number of one's in the binary expansion of $i+1$ and we let $z(i) = \lfloor \log_2(i+1) \rfloor$.

At times, it will be convenient to index entries of $\lk$ backwards from the bottom-right instead of top-left. We define $\rev_k(t) \coloneqq 2^{k+1} - 2 - t$. The value of $k$ will always be clear from context and we will simply write $\rev(t)$. \Cref{lem:rev_identities} lists some useful relationships between $\nu(r+1)$, $z(r)$, $p(r)$ and $\nu(\rev(r+1))$, $z(2^{k+2}-2-\rev(r))$ and $p(\rev(r))$.

For $k\geq 1$, recursively define $\sigma^{(k)}$ to be a vector of length $2^{k-1} -1$ as follows: $\sigma^{(1)} = \emptyset$ is the empty vector, and for $k \geq 2$,
\[
    \sigma^{(k)} = \frac{1}{2}(1+\sqrt{2})^{-2(k-1)}\left[\pi^{(k-2)},\beta_k,2(1+\sqrt{2})^{2(k-1)}\sigma^{(k-1)}\right].
\]
For $k \geq 0$, let $\rho_k$ be a vector of length $\lfloor 2^{k-1}\rfloor+1+2^k$ as follows: $\rho_0 = [0, 1]$ and for $k\geq 1$,
\[
    \rho_{k} = \left[(1+\sqrt{2})^{k-2},(1+\sqrt{2})^{2k-1}\sigma^{(k)}-\frac{\pi^{(k-1)}}{2(1+\sqrt{2})},0,\pi^{(k)},1\right].
\]
For $k \geq 0$, let $w_k$ be the vector of length $2^k$ as follows:
$w_0 = [1]$ and for $k\geq 1$,
\begin{align*}
    w_k = \left[\frac{\pi^{(k-1)}}{\sqrt{\mu_k - 1}},\frac{\beta_k}{\sqrt{\mu_k - 1}},w_{k-1}\right].
\end{align*}

\subsection{Construction of Certificate $\lambda^{(k)}$ for each $\mathfrak{h}^{(k)}$}
\label{subsec:lk_construction}
We now present our construction for $\lambda^{(k)}$.
Throughout this section, fix $k\geq 1$. 
The matrix $\lambda^{(k)}$ is a $(2^{k+1}+1)\times(2^{k+1}+1)$ matrix that we construct below. We will index the rows and columns of $\lambda^{(k)}$ by $\{\star,0,1,\dots,2^{k+1}-1\}$.

We will define the entries of $\lk$ row by row, and denote by $\lk_i$ the $i^{th}$ row of $\lk$.
We first describe the support of $\lk$, i.e. the set of entries $(i,j)$ so that $\lk_{i,j} \neq 0$. We set $\lk_\star=\lk_{2^{k+1}-1}=0$ to be the all-zeros row.
Now consider any $i\in[0,2^{k+1} -2]$ and let $\ell = \nu(i+1)$.
The support of $\lk_i$ is given by $[i-\lfloor2^{\ell-1}\rfloor, i+2^{\ell}]\setminus \{i\}$.
Equivalently, $\lk_{ij} \neq 0$ if and only if $i > j$ and $i - j \le 2^{\nu(i+1)}$ or $j > i$ and $j - i \le 2^{\nu(i+1)-1}$.
Given $v \in \R^{\lfloor2^{\ell-1}\rfloor+1+2^{\ell}}$, we will write $\lk_{i} \sim v$ to denote that $\lk_{i j} = v_{j -i + \lfloor2^{\ell-1} \rfloor}$ for all $j \in [i-\lfloor2^{\ell-1}\rfloor, i+2^{\ell}]\setminus \{i\}$ (and 0 for all other values of $j$).

There are five cases for $\lk_i$ depending on $i$. See Figure~\ref{fig:lambda_rows} for a depiction of the different cases.

\textbf{Case 1}: $i+1 <2^k$ and $i+1$ is not a power of 2.

\[
    \lk_i \sim \frac{\mu_{z(i)+1}-1}{(1+\sqrt{2})^{2(z(i)-p(i))+3}} \rho_{\nu(i+1)}.
\]

\textbf{Case 2}: $i+1 > 2^k$ and $2^{k+1}-1 - i$ is not a power of 2.

\[
    \lk_i \sim \frac{(1+\sqrt{2})^{2(p(i) + z(2^{k+1}-2-i) - k)-1}}{\mu_{z(2^{k+1}-2-i)+1}-1} \rho_{\nu(i+1)}.
\]

\textbf{Case 3}: $i = 2^{\ell}-1$ with $\ell < k$.

In this case, we let $\lk_0 \sim [0,2]$ if $\ell=0$
and for $\ell \geq 1$, let
\[
    \lk_i \sim \left[\frac{1}{2}(\mu_{\ell}-1)(1+\sqrt{2})^{-(\ell-1)}-1,(\mu_{\ell}-1)\sigma^{(\ell)},0,\frac{\mu_{\ell+1}-1}{2(1+\sqrt{2})^{2\ell}}\pi^{(\ell)},\frac{\mu_{\ell+1}-1}{2(1+\sqrt{2})^{2\ell}}\right].
\]

\textbf{Case 4}: $i = 2^k-1$.

In this case, let
\[
    \lk_i \sim \left[\frac{1}{2}(\mu_{k}-1)(1+\sqrt{2})^{-(k-1)}-1,(\mu_{k}-1)\sigma^{(k)},0,\frac{w_k}{\sqrt{\mu_k-1}}\right].
\]

\textbf{Case 5}: $i = 2^{k+1}-2^{\ell}-1$ with $\ell < k$.

In this case, we let $\lk_{2^{k+1} - 2} = \left[0,\frac{1}{2}\right] = \left[0,\left(\frac{1}{\sqrt{\mu_0 - 1}} - \frac{1}{\sqrt{\mu_1 - 1}}\right) w_0\right]$ for $\ell=0$
and for $\ell\geq 1$, let
\begin{align*}
    &\lk_i \sim\\
    &\left[\frac{(1+\sqrt{2})^{\ell-1}}{\mu_{\ell+1}-1},\frac{1}{\mu_{\ell+1}-1}\left(\!(1+\sqrt{2})^{2\ell}\sigma^{(\ell)}-\frac{\pi^{(\ell-1)}}{2}\!\right)\!,0,\left(\frac{1}{\sqrt{\mu_{\ell}-1}} - \frac{1}{\sqrt{\mu_{\ell+1}-1}} \right)w_{\ell}\right]\!.
\end{align*}

\begin{figure}
    \centering
    \begin{tikzpicture}[scale=0.35]
        \draw [step=1.0,black,thin] (-1,-16) grid (16,1);

        \foreach \a/\b in {0/1,
            1/0,1/2,1/3,
            3/1,3/2,3/4,3/5,3/6,3/7
            } 
            \draw (\b+0.5,-\a-0.5) node[regular polygon,regular polygon sides=4,fill=blue!50,scale=1] {};
        \foreach \a/\b in {
            7/3,7/4,7/5,7/6,7/8,7/9,7/10,7/11,7/12,7/13,7/14,7/15
            } 
            \draw (\b+0.5,-\a-0.5) node[regular polygon,regular polygon sides=4,fill=green!50,scale=1.0] {};
        \foreach \a/\b in {
            11/9,11/10,11/12,11/13,11/14,11/15,
            13/12,13/14,13/15,
            14/15
            } 
            \draw (\b+0.5,-\a-0.5) node[regular polygon,regular polygon sides=4,fill=red!50,scale=1.0] {};
            \foreach \a/\b in {
            2/3,
            4/5,
            5/4,5/6,5/7,
            6/7
            } 
            \draw (\b+0.5,-\a-0.5) node[regular polygon,regular polygon sides=4,fill=magenta!50,scale=1.0] {};
        \foreach \a/\b in {
            8/9,
            9/8,9/10,9/11,
            10/11,
            12/13
            } 
            \draw (\b+0.5,-\a-0.5) node[regular polygon,regular polygon sides=4,fill=cyan!50,scale=1.0] {};
      \end{tikzpicture}\hspace{0.5em}
      \begin{tikzpicture}[scale=0.18]
        \draw [step=1.0,black,very thin] (-1,-32) grid (32,1);
        \foreach \a/\b in {0/1,
            1/0,1/2,1/3,
            3/1,3/2,3/4,3/5,3/6,3/7,
            7/3,7/4,7/5,7/6,7/8,7/9,7/10,7/11,7/12,7/13,7/14,7/15
            } 
            \draw (\b+0.5,-\a-0.5) node[regular polygon,regular polygon sides=4,fill=blue!50,scale=0.5] {};
        \foreach \a/\b in {
            15/7,15/8,15/9,15/10,15/11,15/12,15/13,15/14,15/16,15/17,15/18,15/19,15/20,15/21,15/22,15/23,15/24,15/25,15/26,15/27,15/28,15/29,15/30,15/31
            } 
            \draw (\b+0.5,-\a-0.5) node[regular polygon,regular polygon sides=4,fill=green!50,scale=0.5] {};
        \foreach \a/\b in {
            23/19,23/20,23/21,23/22,23/24,23/25,23/26,23/27,23/28,23/29,23/30,23/31,
            27/25,27/26,27/28,27/29,27/30,27/31,
            29/28,29/30,29/31,
            30/31
            } 
            \draw (\b+0.5,-\a-0.5) node[regular polygon,regular polygon sides=4,fill=red!50,scale=0.5] {};
            \foreach \a/\b in {
            2/3,
            4/5,
            5/4,5/6,5/7,
            6/7,
            8/9,
            9/8,9/10,9/11,
            10/11,
            11/9,11/10,11/12,11/13,11/14,11/15,
            12/13,
            13/12,13/14,13/15,
            14/15
            } 
            \draw (\b+0.5,-\a-0.5) node[regular polygon,regular polygon sides=4,fill=magenta!50,scale=0.5] {};
        \foreach \a/\b in {
            16/17,
            17/16,17/18,17/19,
            18/19,
            19/17,19/18,19/20,19/21,19/22,19/23,
            20/21,
            21/20,21/22,21/23,
            22/23,
            24/25,
            25/24,25/26,25/27,
            26/27,
            28/29
            } 
            \draw (\b+0.5,-\a-0.5) node[regular polygon,regular polygon sides=4,fill=cyan!50,scale=0.5] {};
      \end{tikzpicture}
      \caption{Supports of $\lambda^{(3)}$ and $\lambda^{(4)}$. The rows of $\lambda^{(k)}$ are defined according to Case 1 (magenta), Case 2 (cyan), Case 3 (blue), Case 4 (green), Case 5 (red).}
      \label{fig:lambda_rows}
\end{figure}
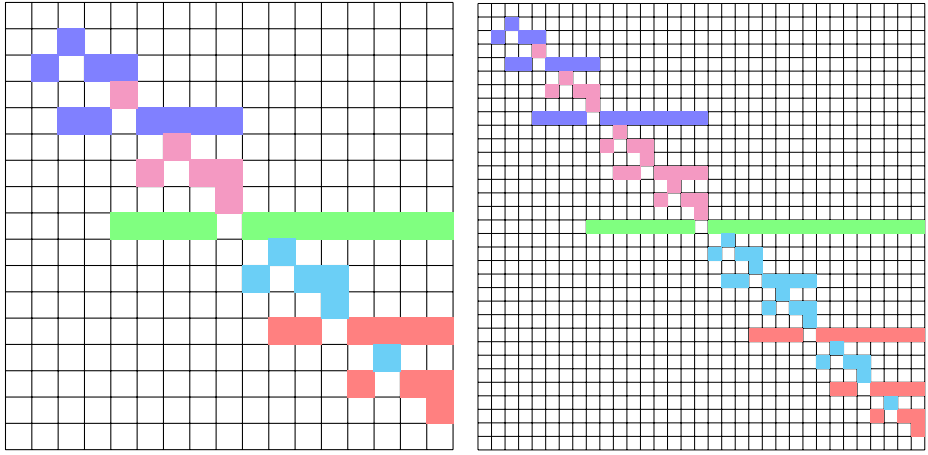

While we find this vector concatenation notation to be more compact and easier to read than specifying each entry of $\lambda$ separately, we will also give $\lambda$ entry by entry.

\textbf{Case 1}: $i+1 <2^k$ and $i+1$ is not a power of 2.

Let $z = z(i)$, $\ell = \nu(i+1)$, and $p = p(i)$, then
\[
    \lk_{ij} = \begin{cases}
                    \frac{\mu_{z+1}-1}{(1+\sqrt{2})^{2(z-p)-\ell+5}}& \text{ if } j = i - 2^{\ell - 1},\\
                    \frac{\mu_{z+1}-1}{(1+\sqrt{2})^{2(z-p)+3}}(\frac{1}{2}(1+\sqrt{2})^{2(\ell-a)-3}\beta_{a+2} - \frac{\beta_{a}}{2(1+\sqrt{2})}) & \text{ if } j = i - 2^{a}
                    \vspace{-0.5em}\\
                    &\qquad\quad\text{for }0 \le a < \ell-1,\\
                    \frac{\mu_{z+1}-1}{2(1+\sqrt{2})^{2(z-p)+4}}(((1+\sqrt{2})^{2(\ell-a)-2} - 1)\beta_{\nu(j+1)}) & \text{ if } i - 2^{a} > j > i - 2^{a+1}
                    \vspace{-0.5em}\\
                    &\qquad\quad\text{for }0 \le a < \ell-1,\\
                    \frac{\mu_{z+1}-1}{(1+\sqrt{2})^{2(z-p)+3}}\beta_{\nu(j + 1)} & \text{ if } i + 2^{\ell} > j > i,\\
                    \frac{\mu_{z+1}-1}{(1+\sqrt{2})^{2(z-p)+3}} & \text{ if } j = i + 2^{\ell},\\
                    0 & \text{ otherwise.}
                \end{cases}
\]

\textbf{Case 2}: $i+1 > 2^k$ and $2^{k+1}-1 - i$ is not a power of 2.

Let $z = z(2^{k+1}-2-i)$, $\ell = \nu(i+1)$, and $p = p(i)$, then
\[
    \lk_{ij} = \begin{cases}
                    \frac{(1+\sqrt{2})^{2(p + z-k)-3+\ell}}{\mu_{z+1}-1} & \text{ if } j = i - 2^{\ell - 1},\\
                    \frac{(1+\sqrt{2})^{2(p + z - k)-1}}{2(\mu_{z+1}-1)} 
                    ((1+\sqrt{2})^{2(\ell-a)-3}\beta_{a+2} - \frac{\beta_{a}}{(1+\sqrt{2})})& \text{ if } j = i - 2^{a}
                    \vspace{-0.5em}\\
                    &\qquad\quad\text{ for }0 \le a < \ell-1,\\
                    \frac{(1+\sqrt{2})^{2(p + z - k-1)}}{2(\mu_{z+1}-1)}((1+\sqrt{2})^{2(\ell-a)-2} - 1)\beta_{\nu(j+1)}
                    & \text{ if } i - 2^{a+1} < j < i - 2^{a}
                    \vspace{-0.5em}\\
                    &\qquad\quad\text{for }0 \le a < \ell-1,\\
                    \frac{(1+\sqrt{2})^{2(p + z - k)-1}}{\mu_{z+1}-1} \beta_{\nu(j + 1)} & \text{ if } i + 2^{\ell} > j > i,\\
                    \frac{(1+\sqrt{2})^{2(p + z - k)-1}}{\mu_{z+1}-1} & \text{ if } j = i + 2^{\ell},\\
                    0 & \text{ otherwise.}
                \end{cases}
\]

\textbf{Case 3}: $i = 2^{\ell}-1$ with $\ell < k$.
\[
    \lk_{ij} = \begin{cases}
                    \frac{1}{2}(\mu_{\ell}-1)(1+\sqrt{2})^{-(\ell-1)}-1& \text{ if } j = 2^{\ell - 1} - 1,\\
                    \frac{1}{2}(\mu_{\ell}-1)(1+\sqrt{2})^{-2(a+1)}\beta_{a+2}& \text{ if } j = i - 2^{a}\text{ for }0 \le a < \ell-1,\\
                    \frac{1}{2}(\mu_{\ell}-1)(1+\sqrt{2})^{-2(a+1)}\beta_{\nu(j+1)} & \text{ if } i - 2^{a} > j > i - 2^{a+1}
                    \vspace{-0.5em}\\
                    &\qquad\quad\text{for }0 \le a < \ell-1,\\
                    \frac{\mu_{\ell+1}-1}{2(1+\sqrt{2})^{2\ell}}\beta_{\nu(j + 1)} & \text{ if } 2^{\ell+1}-1 > j > 2^{\ell}-1,\\
                    \frac{\mu_{\ell+1}-1}{2(1+\sqrt{2})^{2\ell}} & \text{ if } j = 2^{\ell+1}-1,\\
                    0 & \text{ otherwise.}
                \end{cases}
\]

\textbf{Case 4}: $i = 2^k-1$.
\[
    \lk_{ij} = \begin{cases}
                    \frac{1}{2}(\mu_{k}-1)(1+\sqrt{2})^{-(k-1)}-1& \text{ if } j = 2^{k - 1} - 1,\\
                    \frac{1}{2}(\mu_{k}-1)(1+\sqrt{2})^{-2(a+1)}\beta_{a+2}& \text{ if } j = i - 2^{a}\text{ for }0 \le a < k-1,\\
                    \frac{1}{2}(\mu_{k}-1)(1+\sqrt{2})^{-2(a+1)}\beta_{\nu(j+1)} & \text{ if } i - 2^{a} > j > i - 2^{a+1}\vspace{-0.5em}\\
                    &\qquad\quad \text{for }0 \le a < k-1,\\
                    \frac{1}{\sqrt{\mu_{k}-1}\sqrt{\mu_{a+1}-1}}\beta_{a+1} & \text{ if } j = \rev(2^{a} - 1) \text{ for }0  \le a < k,\\
                    \frac{1}{\sqrt{\mu_{k}-1}\sqrt{\mu_{a+1}-1}}\beta_{\nu(j+1)} & \text{ if } \rev(2^{a} -1) > j > \rev(2^{a+1} -1)
                    \vspace{-0.5em}\\
                    &\qquad\quad\text{for }0  \le a < k,\\
                    0 & \text{ otherwise.}
                \end{cases}
\]

\textbf{Case 5}: $i = 2^{k+1}-2^{\ell}-1$ with $\ell < k$.
\[
    \lk_{ij} = \begin{cases}
                    \frac{(1+\sqrt{2})^{\ell-1}}{\mu_{\ell+1}-1}& \text{ if } j = i-2^{\ell-1},\\
                    \frac{1}{2(\mu_{\ell+1}-1)} \left((1+\sqrt{2})^{2(\ell-a-1)}\beta_{a+2}-\beta_a\right)  & \text{ if } j = i - 2^{a}\text{ for }0 \le a < \ell-1,\\
                    \frac{1}{2(\mu_{\ell+1}-1)}\left((1+\sqrt{2})^{2(\ell-a-1)}-1\right)\beta_{\nu(j+1)} & \text{ if } i - 2^{a} > j > i - 2^{a+1}
                    \vspace{-0.5em}\\
                    &\qquad\quad\text{for }0 \le a < k-1,\\
                    \left(\frac{1}{\sqrt{\mu_{\ell}-1}}-\frac{1}{\sqrt{\mu_{\ell+1}-1}}\right)\frac{\beta_{a+1}}{\sqrt{\mu_{a+1}-1}} & \text{ if } j = \rev(2^{a} - 1)\text{ for }0  \le a < \ell,\\
                    \left(\frac{1}{\sqrt{\mu_{\ell}-1}}-\frac{1}{\sqrt{\mu_{\ell+1}-1}}\right)\frac{\beta_{\nu(j+1)}}{\sqrt{\mu_{a+1}-1}} & \text{ if } \rev(2^{a} -1) > j > \rev(2^{a+1} -1)
                    \vspace{-0.5em}\\
                    &\qquad\quad\text{ for }0  \le a < \ell,\\
                    0 & \text{ otherwise.}
                \end{cases}
\]

\subsection{Lower bound on second smallest eigenvalue of $W_2(\lambda)$}
\label{subsec:lambda2_bound}

In the remainder of this section, we fix $k\geq 1$ and let $\lambda = \lk$.
In this subsection only, we will let $\mathfrak{L}(\cdot)$ denote the second-smallest eigenvalue of its argument.

Recall that $W_2(\lambda)$ is
\begin{align*}
    W_2(\lambda) = \sum_{i\neq j}\lambda_{i,j} (e_i - e_j) \odot (e_i-e_j).
\end{align*}
We recognize this as the Laplacian of the weighted graph on the vertices $[0,2^{k+1} - 1]$ where the vertices $i$ and $j$ are connected with an edge with weight $\lambda_{i,j} + \lambda_{j,i}$.
We will lower-bound $\mathfrak{L}(W_2(\lambda))$ by identifying a simpler weighted graph that is dominated by our original graph and bounding its second-smallest eigenvalue instead.

The following lemma computes some lower bounds on the entries of $\lambda$. This will allow us to identify our simpler graph. Its proof simply requires checking the relevant entries of $\lambda$ and is deferred to \cref{app:edge_weight_bound}.

\begin{lemma}
    \label{lem:lower_bound_edge_weights}
Let $k\geq 1$.
If $\ell \in[1,k]$, then
\begin{align*}
    \lambda_{2^\ell - 1, 2^{\ell-1}-1}\geq \frac{1}{\sqrt{2}}.
\end{align*}
If $\ell\in[1,k]$ and $j\in(2^{\ell-1}-1, 2^\ell -1)$, then
\begin{align*}
    \lambda_{2^\ell - 1, j}\geq \frac{(1+\sqrt{2})^2}{\sqrt{2}} (1+\sqrt{2})^{-\ell}.
\end{align*}
If $j\in(2^k-1, 2^{k+1}-1]$, then
\begin{align*}
    \lambda_{2^k-1, j} \geq  \frac{1}{2\sqrt{2}}(1+\sqrt{2})^{-k}.
\end{align*}
\end{lemma}

We are now ready to prove a lower bound on $\mathfrak{L}(W_2(\lambda))$.

\begin{theorem}
Let $k \geq 1$. Then, the second smallest eigenvalue of $W_2(\lambda)$ is at least $\frac{1}{286}(1+\sqrt{2})^{-k}$.
\end{theorem}
\begin{proof}
We describe a weighted graph on the vertices $[0,2^{k+1} - 1]$ that we will refer to throughout this proof as the \emph{caterpillar graph}.
Begin with a path graph on the vertices $\set{2^0-1, 2^1 - 1, 2^2 - 1, \dots, 2^k-1}$ with edge weights $\frac{1}{\sqrt{2}}$. We refer to the vertices on this path as the \emph{path vertices}.
Next, for every $\ell\in[2,k]$, add an edge from $2^{\ell}-1$ to every $j\in(2^{\ell-1} - 1, 2^\ell - 1)$ with edge weight $\frac{1}{2\sqrt{2}}(1+\sqrt{2})^{-k}$. Also add an edge from $2^k - 1$ to every $j\in(2^k - 1, 2^{k+1} - 1]$. We refer to these vertices as \emph{leaf vertices}.
There are $N\coloneqq 2^{k+1} - (k+1)$ leaf vertices.
One may check that $N\geq 2$ for all $k\geq 1$ and that every vertex in $[0,2^{k+1} - 1]$ is either a leaf vertex or a path vertex.

Let $\mathcal{L}$ denote the Laplacian of the caterpillar graph.
\cref{lem:lower_bound_edge_weights} implies that
$\mathfrak{L}(W_2(\lambda))\geq \mathfrak{L}(\mathcal{L})$.

Now, suppose for the sake of contradiction that $\mathfrak{L}(W_2(\lambda))\leq \frac{1}{286}(1+\sqrt{2})^{-k}$. For simplicity, let $\Xi = \frac{1}{286}(1+\sqrt{2})^{-k}$ within this proof.
We will deduce from this assumption that the ``vertex-weighted'' star graph that arises from contracting all path vertices in the caterpillar graph to a single vertex has small algebraic connectivity, from which we will derive a contradiction.

By assumption, there exists a nonzero vector $x\in\R^{2^{k+1}}$ indexed by $[0,2^{k+1}-1]$ for which
\begin{align*}
    x^\intercal \mathcal{L} x \leq \Xi,
    \qquad
    x^\intercal x = 1,
    \qquad
    \mb 1^\intercal x =0.
\end{align*}
From this, we deduce that for each $i\in[1,k]$, that $\frac{1}{\sqrt{2}}(x(2^i - 1) - x(2^{i-1} - 1))^2 \leq \Xi$. Chaining these inequalities, for any pair of path vertices, the difference in $x$ values is bounded above by $k 2^{1/4}\Xi^{1/2}$.

Now, define $\tilde x\in \R^{1+N}$ to be a vector that is indexed by a new vertex $o$ and all leaf vertices. We set $\tilde x$ to agree with $x$ on all leaf vertices and to take the value $\tilde x (o) = \frac{1}{k+1}\sum_{\ell=0}^k x(2^\ell - 1)$ on the new vertex.
Define $\mu(o) = k+1$ and $\mu(i) = 1$ for all leaf vertices $i$.
Our construction ensures that
\begin{align*}
    \mu(o)\tilde x(o) + \sum_{\textup{leaf vtx } i}\mu(i) \tilde x(i) = \mb 1^\intercal x = 0.
\end{align*}

Next, let $\mathcal{L}_\textup{star}$ denote the weighted Laplacian on the leaf vertices and the vertex $o$, that contains an edge of weight $\frac{1}{2\sqrt{2}}(1+\sqrt{2})^{-k}$ from $o$ to each leaf vertex. Given a leaf vertex $i$, let $\textup{parent}(i)$ denote the path vertex that $i$ was attached to in the caterpillar graph. Then,
\begin{align*}
    \tilde x^\intercal \mathcal{L}_\textup{star}\tilde x 
    &=\frac{1}{2\sqrt{2}}(1+\sqrt{2})^{-k} \sum_{\textup{leaf vtx } i} (\tilde x(i) - \tilde x(o))^2\\
    &\leq \frac{1}{2\sqrt{2}}(1+\sqrt{2})^{-k} \sum_{\textup{leaf vtx } i} 2(x(i) - x(\textup{parent}(i)))^2 + 2(x(\textup{parent}(i)) - \tilde x(o))^2\\
    &\leq 2 \Xi + 2k^2\Xi \left(\frac{2}{1+\sqrt{2}}\right)^k.
\end{align*}
Here, on the second line, we have used the inequality $(a-b)^2 \leq 2(a-c)^2 + 2(c-b)^2$.

We also have
\begin{align*}
    \mu(o)\tilde x(o)^2 + \sum_{\textup{leaf vtx } i}\mu(i) \tilde x(i)^2 &=
    \sum_{\text{path vtx } i}x(i)^2 + (\tilde x(o)^2 - x(i)^2) + \sum_{\textup{leaf vtx } i}x(i)^2\\
    &\geq 1 - 4k(k+1)\sqrt{\Xi}
\end{align*}
Here, on the second line, we have used the inequality $\abs{a^2 - b^2} \leq \abs{a-b}\abs{a+b}$. We have also used the fact that $\norm{x}_2 = 1$ so that $\abs{\tilde x(0)}, \abs{x(i)} \leq 1$.

Now, note that $\tilde x$ is a feasible solution to the following problem
\begin{align}
    \label{eq:var_char_lambda2}
    \min_{y\in\R^{N+1}}\set{\frac{y^\intercal \mathcal{L}_{\textup{star}}y}{y^\intercal \Diag(\mu) y}:\, \begin{array}{l}
    y^\intercal \Diag(\mu) \mb 1_{N+1} = 0
    \end{array}},
\end{align}
with objective value at most 
\begin{align*}
    \frac{2 \Xi + 2k^2\Xi \left(\frac{2}{1+\sqrt{2}}\right)^k}{1 - 4k(k+1)\sqrt{\Xi}} &= \frac{1}{2\sqrt{2}}(1+\sqrt{2})^{-k}\left(\frac{\frac{4\sqrt{2}}{286} + \frac{4\sqrt{2}}{286} k^2 \left(\frac{2}{1+\sqrt{2}}\right)^k}{1 - 4k(k+1)\sqrt{\frac{1}{286}(1+\sqrt{2})^{-k}}}\right).
\end{align*}
One may check (see \nextmathematica) that the expression in parentheses is $<1$ for all $k \geq 1$.

On the other hand, \eqref{eq:var_char_lambda2} is the variational characterization of the second smallest eigenvalue of
\begin{align*}
    &\Diag(\mu)^{-1/2} \mathcal{L}_\textup{star}\Diag(\mu)^{-1/2}\\
    &\qquad=
    \frac{1}{2\sqrt{2}}(1+\sqrt{2})^{-k}\left(\begin{smallmatrix}
    k+1\\
    & 1\\
    && 1\\
    &&&\ddots\\
    &&&&1
    \end{smallmatrix}\right)^{-1/2}\left(\begin{smallmatrix}
    N & -1 & - 1 & \dots & -1\\
    -1 & 1\\
    -1 &  & 1\\
    \vdots & & & \ddots\\
    -1 & &  && 1
    \end{smallmatrix}\right)\left(\begin{smallmatrix}
        k+1\\
        & 1\\
        && 1\\
        &&&\ddots\\
        &&&&1
    \end{smallmatrix}\right)^{-1/2}\\
    &\qquad= \frac{1}{2\sqrt{2}}(1+\sqrt{2})^{-k} \left(\begin{smallmatrix}
        \tfrac{N}{k+1} & -1/\sqrt{k+1} & - 1/\sqrt{k+1} & \dots & -1/\sqrt{k+1}\\
    -1/\sqrt{k+1} & 1\\
    -1/\sqrt{k+1} &  & 1\\
    \vdots & & & \ddots\\
    -1/\sqrt{k+1} & &  && 1
    \end{smallmatrix}\right).
\end{align*}
The identity block in the bottom-right is $N\by N$ where $N\geq 2$. Thus, the second smallest eigenvalue of $\Diag(\mu)^{-1/2} \mathcal{L}_\textup{star}\Diag(\mu)^{-1/2}$ is at least $\frac{1}{2\sqrt{2}}(1+\sqrt{2})^{-k}$ by Cauchy's Interlacing Theorem, a contradiction.\qedhere

\end{proof}

    \noindent {\bf Acknowledgements.} Benjamin Grimmer's work was supported in part by the Air Force Office of Scientific Research under award number FA9550-23-1-0531.
    {\small
    
    \bibliographystyle{unsrt}

    }

    \appendix
    \renewcommand{\nextmathematica}{\stepcounter{mathematicaref}{\color{red}\texttt{Mathematica proof \Alph{section}.\arabic{mathematicaref}}}}

\section{Deferred Proofs and Calculations}

\subsection{Proof of Lemma~\ref{lem:helpful-bounds-Hk-muk}}\label{app:proof-of-asymptotics}
    First we verify $\alpha_k \leq \beta_{k+1}$, then we prove our bounds relating $H_k$ and $\mu_k$, and finally we show $\beta_k\leq \alpha_k$.
    The defining equation of $\alpha_k$ is that $\alpha_k$ is the unique root larger than $1$ of $q_k$. It is clear that $\beta_{k+1}\geq 1$, thus to show $\alpha_k \leq \beta_{k+1}$ suffices to show that $q_k(\beta_{k+1}) > 0$. We compute
    \[
        q_k(\beta_{k+1}) =  2(\beta_{k+1}-1)^2 + (\mu_k - 1) (\beta_{k+1}-1) - (\beta_{k+1}-1)(\mu_{k}-1) = 2(\beta_{k+1}-1)^2>0.
    \]

    To bound $H_k$, we first claim that the sum of all $\beta_i$'s in $\mathfrak{h}^{(k)}$ is given by $\sqrt{2}((1+\sqrt{2})^{k}-1)-2k$. This follows from noting each  $\beta_i$ appears in $\mathfrak{h}^{(k)}$ a total of $2(2^{k-i-1}-1)$ times and so the total value of $\beta_i$ terms in $\mathfrak{h}^{(k)}$ is
    \begin{align*}
        \sum_{i=0}^{k-2} 2(2^{k-i-1}-1)(1+(1+\sqrt{2})^{i-1}) = \sqrt{2}((1+\sqrt{2})^{k}-1)-2k.
    \end{align*}   
    Recall also that $\mu_k$ is the sum of all other entries in $\mathfrak{h}^{(k)}$ plus two. Thus,
    \begin{align*}
        H_k = 2\sqrt{2}((1+\sqrt{2})^k - 1)-4k + 4\sum_{i=0}^{k-1}\alpha_i + 2.
    \end{align*}
    To get an upper bound on this quantity, note that $\alpha_i\leq \beta_{i+1}= 1+(1+\sqrt{2})^i$. Thus,
    \begin{align*}
        H_k &\leq 2\sqrt{2}((1+\sqrt{2})^k - 1) + 4\sum_{i=0}^{k-1}(1+\sqrt{2})^i + 2\\
        &= 2\sqrt{2}((1+\sqrt{2})^k - 1) + 2\sqrt{2}((1+\sqrt{2})^k - 1) + 2\\
        &= 2 + 4\sqrt{2}((1+\sqrt{2})^k - 1)\\
        &\leq 4\sqrt{2} (1+\sqrt{2})^k . 
    \end{align*} 
    To get a lower bound, observe that $\alpha_i \geq 1$ for all $i$. Thus $H_k\geq 2\sqrt{2}((1+\sqrt{2})^k - 1)$.
    The first set of bounds for $\mu_k$ follows from the identity $\mu_k - 1 = H_k/2$.
    The final claimed inequality follows directly as $\mu_{k}-\mu_{k-1} \geq \sqrt{2}(1+\sqrt{2})^k - 2\sqrt{2}(1+\sqrt{2})^{k-1}=(2-\sqrt{2})(1+\sqrt{2})^{k-1}$.

    To conclude, we show $\beta_k\leq \alpha_k$ by showing $q_k(\beta_k)\leq 0$. We compute
    \[
        q_k(\beta_{k}) =  2(\beta_{k}-1)^2 + (\mu_k - 1) (\beta_{k}-\beta_{k+1}) = 2(\beta_{k}-1)^2 - \sqrt{2}(\mu_k - 1) (\beta_{k}-1)\leq 0
    \]
    where the inequality exactly follows from our lower bound on $\mu_k-1$.

\subsection{Attainment of Proposition~\ref{prop:product-rule}'s Bound by $\mathfrak{h}^{(k)}$} \label{app:product-tightness}
    We previously claimed in Section~\ref{subsec:tightness} that $\prod_{i=0}^{t-1} (\mathfrak{h}^{(k)}_i - 1) = 1$. We show this by induction below. The following lemma is useful in this calculation.
    \begin{lemma}
    \label{lem:mu_inductive}
        $\mu_{i+1} = \mu_i + 2(\alpha_{i}+\beta_{i+1}-2)$.
    \end{lemma}
    \begin{proof}
    It is clear that $\mu_{i+1} = \mu_i + 2 \alpha_i + 2\sum_{j=1}^{2^{i}-1} \pi^{(i)}_j,$
    which we have seen is equal to
    \[
        \mu_{i+1} = \mu_i + 2(\alpha_i + 2\beta_{i+1}-2). \qedhere
    \]
    \end{proof}

    As a base case, when $k = 1$, note that $\alpha_0$ is the positive root of the polynomial 
    \[
        2(x-1)^2 + (\mu_0 - 1) (x-1) - (\beta_{1}-1)(\mu_0-1) = x(2x-3),
    \]
    so that $\alpha_0 = \frac{3}{2}$, and $h^{(1)} = (\frac{3}{2}, 5, \frac{3}{2})$, which satisfies this equation.
    Now, assume that the equation holds for $\mathfrak{h}^{(k-1)}$, i.e. $\prod_{i=0}^{2^{k-1}-1} (\mathfrak{h}^{(k-1)}_i - 1) = 1$. Expand this expression to be
    \[
        \prod_{i=0}^{2^{k}-1} (\mathfrak{h}^{(k-1)}_i - 1) = \prod_{i=0}^{k-2} (\alpha_i - 1)^2 \prod_{i=0}^{k-3} (\beta_i - 1)^{2(2^{k-i-2}-1)} (\mu_{k-1} - 1) = 1.
    \]
    Computing the product of all of the $\beta_i - 1$, this simplifies to
    \[
        \prod_{i=0}^{k-2} (\alpha_i - 1)^2 = \frac{(1+\sqrt{2})^{(k-1)(k-2)}}{\mu_{k-1} - 1}.
    \]
    We similarly can expand
    \[
        \prod_{i=0}^{2^{k}-1} (\mathfrak{h}^{(k)}_i - 1) = \prod_{i=0}^{k-1} (\alpha_i - 1)^2 (1+\sqrt{2})^{-k(k-1)}  (\mu_k-1).
    \]
    Combining our previous expressions, we obtain by Lemma~\ref{lem:mu_inductive} that
    \begin{align*}
        \prod_{i=0}^{2^{k}-1} (\mathfrak{h}^{(k)}_i - 1) &= (\alpha_{k-1} - 1)^2 (1+\sqrt{2})^{-2(k-1)} \frac{\mu_{k}-1}{\mu_{k-1}-1}\\
                                              &= (\alpha_{k-1} - 1)^2 (1+\sqrt{2})^{-2(k-1)} \frac{\mu_{k-1}+2\alpha_{k-1}+2\beta_{k}-5}{\mu_{k-1}-1}.
    \end{align*}
    Now, we note that by the defining equation of $\alpha_{k-1}$ that 
    \[
        (\alpha_{k-1}-1)^2 = \frac{1}{2}\left((\beta_{k}-1)(\mu_{k-1}-1) -  (\mu_{k-1} - 1) (\alpha_{k-1}-1)\right).
    \]
    This gives
    \begin{align*}
        \prod_{i=0}^{2^{k}-1} (\mathfrak{h}^{(k)}_i - 1) &= (\alpha_{k-1} - 1)^2 (1+\sqrt{2})^{-2(k-1)} \frac{\mu_{k}-1}{\mu_{k-1}-1}\\
                                              &= \frac{1}{2}(1+\sqrt{2})^{-2(k-1)}(\mu_{k-1}+2\alpha_{k-1}+2\beta_{k}-5)(\beta_k - \alpha_{k-1})\\
                                              &= \frac{1}{2}(1+\sqrt{2})^{-2(k-1)}(2(\beta_{k}-1)^2 - q_{k-1}(\alpha_{k-1}))\\
                                              &= 1.
    \end{align*}

\subsection{Bounding the Superdiagonal Entries of $\lk$} \label{app:Superdiagonal-bound}
    In this section, we fix a $k\geq 1$ and let $\lambda$ denote the construction $\lambda^{(k)}$ given in \cref{subsec:lk_construction}. Our goal is to bound $\min_{i\in[0,t-1]}\lambda_{i,i+1}$
    below for use in proving \cref{lem:straightforward} via \cref{thm:straightforward_from_rank_one}.
    
    \begin{proposition}
        \label{prop:superdiagonal_bound}
    For all $k\geq 1$, it holds that
    \begin{align*}
        \min_{i\in[0,2^{k+1}-2]}\lambda^{(k)}_{i,i+1} \geq 
        \frac{2-\sqrt{2}}{8\sqrt{2}}(1+\sqrt{2})^{-2k+1}.
    \end{align*}
    \end{proposition}
    \begin{proof}
    We prove this by bounding $\lk_{i,i+1}$ for any $i\in[0,2^{k+1}-2]$ separately across our five cases. We will also make use of the easy that $\mu_{z(i)} - 1 \ge (1+\sqrt{2})^{z(i)+1}$
    
    \textbf{Case 1}: $i+1 <2^k$ and $i+1$ is not a power of 2.

    We use the definition of $\lk_i$ and the fact that $p(i)\geq 2$ and $z(i)\leq k-1$, and $\beta_{\ell} \ge 1$ for each $\ell$. Thus,
    \begin{align*}
        \lk_{i,i+1} &= \frac{\mu_{z(i)+1}-1}{(1+\sqrt{2})^{2(z(i)-p(i))+3}}\beta_{\nu(i+2)}\\
        &\geq \frac{\mu_{z(i)+1}-1}{(1+\sqrt{2})^{2z(i)-1}}\\
        &\geq \sqrt{2}\frac{(1+\sqrt{2})^{z(i)+1}}{(1+\sqrt{2})^{2z(i)-1}}\\
        &\geq \sqrt{2}(1+\sqrt{2})^3(1+\sqrt{2})^{-k}.
    \end{align*}

    \textbf{Case 2}: $i+1 > 2^k$ and $2^{k+1}-1 - i$ is not a power of 2. 
    We use the fact that $p(i)\geq 1$, $\beta_{\ell} \ge 1$ for all $\ell$, and $z(\rev(i))\leq k-1$.
    Thus,
    \begin{align*}
        \lk_{i,i+1} &= \frac{(1+\sqrt{2})^{2(p(i)+z(\rev(i))-k)-1}}{\mu_{z(\rev(i))+1}-1}\beta_{\nu(i+2)}\\
        &\geq \frac{(1+\sqrt{2})^{2z(\rev(i))}}{\mu_{z(\rev(i))+1}-1}(1+\sqrt{2})^{1-2k}\\
        &\geq \frac{1}{2\sqrt{2}}\frac{(1+\sqrt{2})^{2z(\rev(i))}}{(1+\sqrt{2})^{z(\rev(i))}} (1+\sqrt{2})^{-2k}\\
        &\geq \frac{1}{2\sqrt{2}}(1+\sqrt{2})^{-2k}.
    \end{align*}
    
    \textbf{Case 3}: $i = 2^{\ell}-1$ with $\ell < k$.
    
    In this case, we have $\lk_{0,1}=2$. Now, suppose $\ell\geq1$. Then,
    \begin{align*}
        \lk_{i,i+1} &= \frac{\mu_{\ell+1}-1}{2(1+\sqrt{2})^{2\ell}}\sqrt{2}\\
        &\geq (1+\sqrt{2})^{1-\ell}\\
        &\geq (1+\sqrt{2})^{-k}.
    \end{align*}
    
    \textbf{Case 4}: $i = 2^k-1$.
    If $k=0$, then this entry is $\beta_1/(\mu_k-1)$. Otherwise, this entry is $\beta_0/(\mu_k-1)$. As $\beta_1\geq\beta_0$, we may bound the general case as
    \begin{align*}
        \lk_{i,i+1} &\geq \frac{\sqrt{2}}{\mu_k-1}\\
        &\geq \frac{1}{2}(1+\sqrt{2})^{-k}.
    \end{align*}
    
    \textbf{Case 5}: $i = 2^{k+1}-2^{\ell}-1$ with $\ell < k$.
    First, suppose $\ell = 0$. Then, $\lk_{i,i+1}= 1/2$. Now, suppose $\ell\geq 1$. Then, the leftmost entry of $w_\ell$ is either
    $\beta_0/\sqrt{\mu_\ell-1}$ or $\beta_1/\sqrt{\mu_{\ell}-1}$ depending on if $\ell =0$, $\ell=1$, or $\ell>1$.
    In either case, we may bound
    \begin{align*}
        \lk_{i,i+1}&\geq \left(\frac{1}{\sqrt{\mu_{\ell}-1}} - \frac{1}{\sqrt{\mu_{\ell+1}-1}} \right)\frac{\sqrt{2}}{\sqrt{\mu_\ell - 1}}\\
        &= \frac{\sqrt{2}}{(\mu_\ell - 1)}\frac{\mu_{\ell+1} - \mu_\ell}{\sqrt{\mu_{\ell+1}-1}\left(\sqrt{\mu_{\ell+1}-1} + \sqrt{\mu_{\ell}-1}\right)}\\
        &\geq \frac{\mu_{\ell+1} - \mu_\ell}{\sqrt{2}(\mu_{\ell} - 1)(\mu_{\ell+1} - 1)}\\
        &\geq \frac{(2-\sqrt{2})(1+\sqrt{2})^{\ell}}{8\sqrt{2}(1+\sqrt{2})^{\ell}(1+\sqrt{2})^{\ell + 1}}\\
        &\geq \frac{(2-\sqrt{2})}{8\sqrt{2}}(1+\sqrt{2})^{-\ell-1}\\
        &\geq \frac{(2-\sqrt{2})}{8\sqrt{2}}(1+\sqrt{2})^{-k}. \qedhere
    \end{align*}
    \end{proof}

    \subsection{Bounding the Edge Weights in the Caterpillar graph}
    \label{app:edge_weight_bound}
    In \cref{subsec:lambda2_bound}, we required lower bounds on specific entries of $\lambda^{(k)}$. These lower bounds were stated in \cref{lem:lower_bound_edge_weights} and are proved below.
    \begin{proof}[Proof of \cref{lem:lower_bound_edge_weights}]
        Fix $k\geq 1$ and let $\lambda=\lambda^{(k)}$.
    
        First, suppose $\ell \in[1,k]$. We expand the definition of $\lambda_{2^\ell - 1, 2^{\ell-1}-1}$ and use \cref{lem:helpful-bounds-Hk-muk}:
        \begin{align*}
            \lambda_{2^\ell - 1, 2^{\ell-1}-1} &= \frac{(\mu_\ell - 1)}{2(1+\sqrt{2})^{(\ell-1)}} - 1\\
            &\geq \frac{\sqrt{2}(1+\sqrt{2})^\ell}{2(1+\sqrt{2})^{(\ell-1)}} - 1\\
            & = \frac{1}{\sqrt{2}}.
        \end{align*}
        Next, suppose $j\in (2^{\ell-1}-1, 2^{\ell}-1)$.
        We begin by claiming that all entries in $\sigma^{(\ell)}$ are at least $\frac{1}{2(1+\sqrt{2})^{2(\ell-1)}}$. 
        We do so by induction. First, $\sigma^{(1)}$ is empty so that there is nothing to show. 
        Next, for $\ell> 1$, the last $2^{\ell-2}-1$ entries of $\sigma^{(\ell)}$ coincide with the entries of $\sigma^{(\ell-1)}$. By induction the minimum value in those entries is at least $\frac{1}{2(1+\sqrt{2})^{2(\ell-2)}}$.
        Noting that all $\beta$ values are at least one, we deduce that the remaining entries in $\sigma^{(\ell)}$ are bounded below by $\frac{1}{2(1+\sqrt{2})^{2(\ell-1)}}$. This shows the claim.
        
        Now, we expand the definition of $\lambda_{2^\ell - 1, 2^{\ell-1}-1}$ and use \cref{lem:helpful-bounds-Hk-muk}:
        \begin{align*}
            \lambda_{2^\ell - 1, j} &= (\mu_{\ell}-1)\sigma^{(\ell)}_{j- (2^{\ell-1}-1)}\\
            &\geq \frac{\sqrt{2}(1+\sqrt{2})^{\ell}}{2(1+\sqrt{2})^{2(\ell-1)}}\\
            &=\frac{(1+\sqrt{2})^2}{\sqrt{2}} (1+\sqrt{2})^{-\ell}.
        \end{align*}
        
        Finally, suppose $j\in(2^k-1, 2^{k+1}-1]$. We begin by claiming that the entries in $w_\ell$ are at least $\frac{1}{\sqrt{\mu_\ell - 1}}$. We do this by induction. The base case holds as $w_0=[1]$ and $\mu_0 = 2$. For $\ell \geq 1$, note that the last $2^{\ell-1}$ entries of $w_\ell$ coincide with $w_{\ell-1}$. By induction the values of these entries are at least $\frac{1}{\sqrt{\mu_{\ell-1}-1}}$. Noting that all $\beta$ values are at least one, we deduce that the remaining entries in $w_\ell$ are at least $\frac{1}{\sqrt{\mu_\ell - 1}}$. Thus,
        \begin{align*}
            \lambda_{2^{k}-1, j} &= \frac{(w_k)_{j - (2^{k}-1)}}{\sqrt{\mu_k-1}} \geq \frac{1}{\mu_k-1}\geq \frac{1}{2\sqrt{2}(1+\sqrt{2})^k}.\qedhere
        \end{align*}
        \end{proof}
        
\section{Useful Supporting Identities and Properties}\label{sec:funFacts}
\subsection{Algebraic Properties of $\lk$}
Here are two recurrence relations for $\lk$ that are useful in various calculations.
They say that certain entries (or rows) of $\lk$ are simply scalar multiples of other entries (or rows).
\begin{lemma}
    \label{lem:recurrence}
    Let $2^{z} - 1< i, j < 2^{z+1}-1 \in \N$ with $0\leq z < k$. If $k > z' > z$, then 
    \[
        \lk_{2^{z'}+i, 2^{z'}+j} = (1+\sqrt{2})^{2(z - z' + 1)}\frac{\mu_{z'+1}-1}{\mu_{z+1}-1}\lk_{i,j}.
    \]
\end{lemma}
\begin{proof}
    The claim holds if $i=j$. In the remainder assume $i\neq j$.
    Note that $\lk_{i,j} \neq 0$ if and only if $-2^{\nu(i+1) - 1} \le j - i \le 2^{\nu(i+1)}$.
    As $(2^{z'}+j) - (2^{z'}+i) = j-i$ and $\nu(2^{z'} + i+1) = \nu(i+1)$ we deduce that $\lk_{i,j} \neq 0$ and and only if $\lk_{2^{z'}+i,2^{z'}+j} \neq 0$.

    If $\lk_{i,j} = 0$, then the claim holds. In the remainder, assume $\lk_{i,j} \neq 0$.
    In this case, our definitions imply that
    \[
        \lk_{i,j} = \frac{\mu_{z+1}-1}{(1+\sqrt{2})^{2(z-p(i))+3}} (\rho_{\nu(i+1)})_{2^{\nu(i+1)-1} - i  + j},
    \]
    and since $\nu(2^{z'} + i + 1) = \nu(i + 1)$, and the number of one's in the binary expansion of $2^{z'} + i$ is exactly one more than the number of one's in the binary expansion of $i$, we see that 
    \[
        \lk_{2^{z'} + i,2^{z'} + j} = \frac{\mu_{z'+1}-1}{(1+\sqrt{2})^{2(z'-p(i)-1)+3}} (\rho_{\nu(i+1)})_{2^{\nu(i+1)-1} - i  + j}.
    \]
    Comparing the two expressions yields our result.
\end{proof}

\begin{lemma}
    \label{lem:case_2_scaling}
Suppose either:
\begin{itemize}
    \item $2^{z}-1< r, \ell < 2^{z+1}-1$ for some $z\in[0,k)$, or
    \item $2^{z}-1= r < \ell < 2^{z+1}-1$ for some $z\in[0,k)$.
\end{itemize}
Let $k>z'>z$, $\ell'= \ell+ 2^{z'}$ and $r'= r + 2^{z'}$, then
\begin{align*}
    \lk_{\rev(r'),\rev(\ell')} = (1+\sqrt{2})^{2(z'-z-1)}\frac{\mu_{z+1}-1}{\mu_{z'+1}-1}\lk_{\rev(r),\rev(\ell)}.
\end{align*}
\end{lemma}
\begin{proof}
In the first case, we have that both $\lk_{\rev(r)}$ and $\lk_{\rev(r')}$
are defined according to Case 2 and that $\nu(\rev(r')+1)= \nu(r'+1) = \nu(r+1)= \nu(\rev(r)+1)$. 
Thus, these two rows (after the natural re-indexing) are scalar multiples of $\rho_{\nu(r+1)}$ (and hence of each other).
Using the identities \cref{lem:rev_identities}, we have
\begin{gather*}
    \lk_{\rev(r)}\sim \frac{(1+\sqrt{2})^{2(z-p(r)-\nu(r+1))+3}}{\mu_{z+1}-1}\rho_{\nu(r+1)},\\
    \lk_{\rev(r')}\sim \frac{(1+\sqrt{2})^{2(z'-p(r)-\nu(r+1))+1}}{\mu_{z'+1}-1}\rho_{\nu(r+1)}.
\end{gather*}
Comparing the two coefficients proves the claim.

In the second case, we have that $\ell>r$ so that $\rev(\ell)<\rev(r)$ and $\rev(\ell') < \rev(r')$. We have that $\lk_{\rev(r)}$ is defined according to Case 5 and $\lk_{\rev(r')}$ is defined according to Case 2. The nonzero portion left of the diagonal of each row is given by
\begin{align*}
    \text{left portion of }\lk_{\rev(r)} &= \left[\frac{(1+\sqrt{2})^{z - 1}}{\mu_{z+1} - 1},\frac{1}{\mu_{z+1} - 1}\left((1+\sqrt{2})^{2z}\sigma^{(z)} - \frac{\pi^{(z-1)}}{2}\right)\right]\\
    &=\frac{1}{\mu_{z+1} - 1}\left[(1+\sqrt{2})^{z - 1},(1+\sqrt{2})^{2z}\sigma^{(z)} - \frac{\pi^{(z-1)}}{2}\right],\\
    \text{left portion of }\lk_{\rev(r')} &= \frac{(1+\sqrt{2})^{2(z'-z)-1}}{\mu_{z'+1} - 1}\left[(1+\sqrt{2})^{z-2},(1+\sqrt{2})^{2z-1}\sigma^{(z)}-\frac{\pi^{(z-1)}}{2(1+\sqrt{2})}\right]\\
     &= \frac{(1+\sqrt{2})^{2(z'-z-1)}}{\mu_{z'+1} - 1}\left[(1+\sqrt{2})^{z-1},(1+\sqrt{2})^{2z}\sigma^{(z)}-\frac{\pi^{(z-1)}}{2}\right].
\end{align*}
Comparing the two coefficients proves the claim.
\end{proof}

\subsection{Algebraic Properties of $\mu$}
There are various algebraic properties of $\mu$ that we will use in this paper.

\begin{lemma}
    \label{lem:sqrtmu}
    For all $\ell\geq0$, it holds that
    \begin{gather*}
        \mu_{\ell} = 2\frac{(\alpha_{\ell}-1)^2}{\beta_{\ell+1}-\alpha_{\ell}} + 1,\\
       \mu_{\ell+1} = 2\frac{(\beta_{\ell+1}-1)^2}{\beta_{\ell+1}-\alpha_\ell}+1.
    \end{gather*}
    In particular, $\frac{\sqrt{\mu_{\ell}-1}}{\alpha_{\ell}-1} =  \frac{\sqrt{\mu_{\ell+1}-1}}{\beta_{\ell+1}-1}$.
\end{lemma}
\begin{proof}
    The defining equation of $\alpha_{\ell}$ is that
    \begin{align*}
        q_i(\alpha_{\ell}) &=  2(\alpha_{\ell}-1)^2 + (\mu_{\ell} - 1) (\alpha_{\ell}-1) - (\beta_{\ell+1}-1)(\mu_{\ell}-1)\\
        &= 2(\alpha_{\ell}-1)^2 - (\mu_{\ell}-1)(\beta_{\ell+1} - \alpha_{\ell}) = 0.
    \end{align*}
    Solving this equation for $\mu_{\ell}$ shows that
    \[
        \mu_{\ell} = 2\frac{(\alpha_{\ell}-1)^2}{\beta_{\ell+1}-\alpha_{\ell}} + 1.
    \]
    Recall that $\mu_{\ell+1} = \mu_{\ell} + 2((\alpha_{\ell}-1) + (\beta_{\ell+1}-1)) $, so
    \begin{align*}
        \mu_{\ell+1} &= 2\frac{(\alpha_{\ell}-1)^2}{(\beta_{\ell+1}-1)-(\alpha_{\ell}-1)} + 2((\alpha_{\ell}-1) + (\beta_{\ell+1}-1))+ 1\\
        &= 2\frac{(\alpha_{\ell}-1)^2 + \left((\beta_{\ell+1}-1)^2 - (\alpha_{\ell}-1)^2\right)}{(\beta_{\ell+1}-1)-(\alpha_{\ell}-1)} + 1\\
        &= 2\frac{(\beta_{\ell+1}-1)^2}{\beta_{\ell+1}-\alpha_{\ell}} + 1.
    \end{align*}
    It follows that 
    \[
        \frac{\mu_{\ell}-1}{(\alpha_{\ell}-1)^2} = 
        \frac{\mu_{\ell+1}-1}{(\beta_{\ell+1}-1)^2},
    \]
    and taking the square root of both sides implies the last claim.
\end{proof}

\begin{lemma}
\label{lem:relating_muprev_to_mu_and_beta}
    Suppose $k \geq 1$, then $2(\beta_k-1)+ \sqrt{(\mu_{k-1} - 1)(\mu_k - 1)} = \mu_k - 1$, or equivalently, $2(1+\sqrt{2})^{k-1}+ \sqrt{(\mu_{k-1} - 1)(\mu_k - 1)} = \mu_k - 1$.
\end{lemma}
\begin{proof}
     The claimed identity is equivalent to
    \[
         (\mu_{k-1}-1)(\mu_k-1) = (\mu_k-1-2(\beta_k-1))^2.
    \]
    By \cref{lem:mu_inductive}, we have $\mu_{k-1} = \mu_k - 2 (\alpha_{k-1} +\beta_k-2)$. Applying this identity and combining, we get that this is equivalent to
    \[
    2(\alpha_{k-1}-1)^2 + (\mu_i - 1) (\alpha_{k-1}-1) - (\beta_{i+1}-1)(\mu_i-1).
    \]
    which is the defining equation for $\alpha_{k-1}$.
\end{proof}

\subsection{Properties of the $\rev$ operation}
\begin{lemma}
\label{lem:rev_identities}
    Suppose $0\leq r \leq 2^k-1$. Then, $\rev(\rev(r))=r$ and 
    \begin{gather*}
        \nu(\rev(r)+1) =\nu(r+1),\\
        z(2^{k+1}-2-\rev(r))= z(\rev(\rev(r)))=z(r),\\
        p(\rev(r)) = k-p(r)-\nu(r+1) + 2.
    \end{gather*}
    In particular, $p(\rev(r))+ z(r) -k = z(r) -p(r) -\nu(r+1) + 2$.
\end{lemma}
\begin{proof}
We recall that $\rev(r)= 2^{k+1}-2-r$.
For the first identity, note
\begin{align*}
    \nu(\rev(r)+1) &= \nu(2^{k+1}-2-r+1) = \nu(2^{k+1}-(r+1)).
\end{align*}
Then, recall that the 2-adic valuation for the sum or difference of two numbers with different 2-adic valuations is the smaller of two. As $\nu(r+1)\leq k$, we have that the above quantity is equal to $\nu(r+1)$.

Next, note
\begin{align*}
    z(2^{k+1}-2-\rev(r))&= z(\rev(\rev(r)))=z(r).
\end{align*}

Finally, $p(\rev(r))$ is the number of ones in the binary expansion of $2^{k+1}-1-r
$. This is equivalent to $k+1$ minus the number of ones in the binary expansion of $r$. Now, consider the binary expansion of $r+1$. The smallest position for which the binary expansion of $r+1$ is equal to one, i.e., $\nu(r+1)$, is the same as the smallest position for which the binary expansion of $r$ is equal to zero.
The difference in the number of ones in their binary expansion is then $\nu(r+1)-1$. We have deduced that $p(\rev(r))= k+1 - p(r-1) = k+1- (p(r) +\nu(r+1)-1) = k-p(r)-\nu(r+1) + 2$.
\end{proof}

\section{The Support of $\lk$}
\label{sec:supportFacts}
The support of $\lk$ has a rich combinatorial structure, which we need to make  use of extensively in our computations. We record some facts about this support and their proofs here.
For now, let us fix $k$, and let $\lambda$ refer to $\lk$.

From our definition of $\lambda$, $\lambda_{i,j} \neq 0$ if and only if $i > j > i - 2^{\nu(i+1)-1}$ or $i + 2^{\nu(i+1)} > j > i$.

It is useful to us to understand for a fixed $j$, which are the $i$ where $\lambda_{i,j} \neq 0$.
For a given $j \le 2^{k+1}-1$, we let 
\[
    S_j^- = \{i < j : \lambda_{i,j} \neq 0\}
    \qquad\text{and}\qquad
    S_j^+ = \{i > j : \lambda_{i,j} \neq 0\}.
\]

\begin{lemma}
    Suppose that $j\in[1,2^k - 1]$ has the binary expansion $j = \sum_{a=0}^{z} b_{a} 2^{a}$, where $b_i \in \{0,1\}$ and $b_z = 1$, then
    \[
        S_j^- = \set{\sum_{a = r}^{z} b_a 2^a - 1 : r \in [0,z]}.
    \]
\end{lemma}
\begin{proof}
    We begin by showing that if $i = \sum_{a = r}^{z} b_a 2^a - 1 $ for some $r \in [0,z]$ where $b_r = 1$, then $i\in S_j^-$. Note that $\nu(i+1) = r$. This implies that
    \[
        i <\sum_{a = 0}^{z} b_a 2^a = j,
    \] and that 
    \begin{align*}
        i + 2^{\nu(i+1)}&\ge \sum_{a = \nu(i+1)}^{z} b_a 2^a + (2^{\nu(i+1)} - 1)\\
        &= \sum_{a = \nu(i+1)}^{z} b_a 2^a + \sum_{a=0}^{\nu(i+1)-1} 2^a \\
        &\ge \sum_{a = 0}^{z} b_a 2^a\\
        &= j.
    \end{align*}

    Now, we show the reverse direction, i.e., if $i < j$ and $i + 2^{\nu(i+1)} \ge j$, then $i = \sum_{a = \nu(i+1)}^{z} b_a 2^a - 1$.
Note that $z(i+1) \le z$ since $i < j$.
    This implies that the binary expansion of $i+1$ can be expressed as
    \[
        i+1 = \sum_{a = \nu(i+1)}^{z} b_a' 2^a,
    \]
    so that
    \[
        i = \sum_{a = \nu(i+1)}^{z} b_a' 2^a - 1.
    \]
    For the sake of contradiction, suppose that $b_a' \neq b_a$ for some $a \ge \nu(i+1)$. In this case, let $a^*$ be the largest $a$ so that $b_a' \neq b_a$. 
    If $b_{a^*}' = 0$ while $b_{a^*} = 1$, then 
    \[
        i = \sum_{a = \nu(i+1)}^{a^*-1} b_a' 2^a + \sum_{a = a^*+1}^{z} b_a 2^a - 1,
    \]
    so that 
    \[
        i + 2^{\nu(i+1)} \le 2^{a^*} + \sum_{a = a^*+1}^{z} b_a 2^a - 1 < j.
    \]
     If $b_{a^*}' = 1$, while $b_{a^*} = 0$, then 
    \[
        i \ge \sum_{a = a^*+1}^{z} b_a 2^a + 2^{a^*} - 1 =  \sum_{a = a^*+1}^{z} b_a 2^a + \sum_{a=0}^{a^*-1}2^a \ge j.
    \]
  
    In either case, we reach a contradiction.
\end{proof}

\begin{lemma}
    Suppose $2^z - 1 \leq j < 2^{z+1} - 1$ with $z<k$.
    Then, $i\in S_j^+$ if and only if $i\in [j+1, 2^{k+1}-2]$ and  $i - j \le 2^{\nu(i+1)-1}$.
    In particular, if $i \in S_j^+$, then $j < i \le 2^{z+1}-1$.
    If in addition $j = 2^z-1$, then $S_j^+$ is the singleton set $\{2^{z+1}-1\}$.
\end{lemma}
\begin{proof}
    This is clear from the support of $\lambda$.
\end{proof}

\begin{lemma}
    Suppose that $2^z - 1 \le j < 2^{z+1}-1$ with $z < k$. Let $z<\ell<k$.
    We then have that 
    \[
        S_{j + 2^{\ell}}^+ = \{2^{\ell} + i : i \in S_j^+\} \cup \{2^{\ell+1}-1\}.
    \]
\end{lemma}
\begin{proof}
    First, suppose that $i \in S_{2^{\ell} + j}^+$. Because $2^{\ell} - 1 \le j < 2^{\ell+1}-1$, we then have that $2^{\ell} + j < i \le 2^{\ell+1}-1$ by the previous lemma.

    Suppose that $i \neq 2^{\ell+1}-1$, then let $i' = i - 2^{\ell}$, so $\nu(i'+1) = \nu(i+1)$.
    Our previous lemma implies that $i - (2^{\ell}+j) \le 2^{\nu(i+1)-1}$, so $i' - j \le 2^{\nu(i+1)-1} =  2^{\nu(i'+1)-1}$.
    Therefore, $i' \in S_j^+$.

    Now, suppose that $i \in S_j^+$, then let $i' = 2^{\ell}+i$.
    As before, $\nu(i+1) < \ell$, so $\nu(i'+1) = \nu(i+1)$ and 
    \[
        i' - (2^{\ell} + j) = i - j \le 2^{\nu(i+1) - 1}.
    \]
    This implies that $i' \in S_{j+2^{\ell}}^+$.
\end{proof}

     \section{Proof of \Cref{thm:rowcolconstraint}} \label{sec:lambdadRowColSum}
In this section, we will show that $\lk$ satisfies the first main condition of \Cref{thm:straightforward_from_rank_one}.

\begin{theorem}
    \label{thm:rowcolconstraint}
    Suppose $k\geq 1$, then
    \[
        \sum_{\substack{i,j=0\\ i\neq j}}^t\lk_{i,j}a_{i,j} = a_{\star,2^{k+1}-1} - a_{\star,0}.
    \]
    Equivalently,
    \begin{itemize}
        \item The sum of the zeroth row of $\lk$ is one larger than the sum of the zeroth column of $\lk$.
        \item For $i = 1,\dots,2^{k+1}-2$, the sum of the $i^{\mathrm{th}}$ row of $\lk$ equals the sum of the $i^{\mathrm{th}}$ column of $\lk$.
        \item The sum of the $2^{k+1}-1$ row of $\lk$ is one less than the sum of the $2^{k+1}-1$ column of $\lk$.
    \end{itemize}
\end{theorem}
\begin{proof}
    The equivalence of the two statements follows from \cref{lem:lin_constraints}. Thus, it suffices to prove the three statements in the second claim.
    We show the first item in \Cref{lem:rowcol0}.
    We show the second item in \Cref{lem:rowcollow}, \Cref{lem:rowcolmid}, and \Cref{lem:rowcolhighrest}.
    We show the third item in \Cref{lem:rowcollast}.
\end{proof}

In the remainder of this section, we fix $k\geq 1$ and let $h = \mathfrak{h}^{(k)}$ and $\lambda = \lk$. Section~\ref{subsec:RowSums} computes the sums of each row of $\lambda$. Section~\ref{subsec:ColSums} computes the sums of each column of $\lambda$. Finally, Section~\ref{subsec:RowColSums} proves lemmas claimed above. Various algebraic identities involving the entries of $h$ will be used in this section and proven in \Cref{sec:funFacts}.

\subsection{Row Sums} \label{subsec:RowSums}
\subsubsection{Partial Row Sums}
\label{subsub:partial_rows}
Each row of $\lambda$ is composed of various components; we will enumerate their sums here.
\begin{lemma}
    \label{lem:sumpi}
    For $k\geq 0$,
    $\sum_{i=1}^{2^k-1} \pi^{(k)}_i = \beta_{k+1}-2$.
\end{lemma}
\begin{proof}
    First, note that $\pi^{(0)} = \emptyset$ so that $\sum_{i=1}^0 \pi^{(0)}_i = 0$. We also have $\beta_1 - 2 = 0$.

    Next, note that 
    \[\sum_{i=1}^{2^k-1} \pi^{(k)}_i = \sum_{i=1}^{2^{k-1}-1} \pi^{(k)}_i + \beta_{k-1} + \sum_{i=2^{k}+1}^{2^k-1} \pi^{(k)}_i.
    \]
    Note that the first $2^{k-1}-1$ entries of $\pi^{(k)}$ are identical to those of $\pi^{(k-1)}$, and the same holds for the last $2^{k-1}-1$ entries. By induction, we may conclude that
    \[
        \sum_{i=1}^{2^k-1} \pi^{(k)}_i = 2 (\beta_k - 2) + \beta_{k-1} = \beta_{k+1}-2.
    \]
    See {\nextmathematica} for a proof of the second identity.
\end{proof}

\begin{lemma}
    \label{lem:sumsigma}
    For $k\geq 1$, the sum of the entries in $\sigma^{(k)}$ is 
    \[
        \frac{1}{2}\left(1-(1+\sqrt{2})^{1-k}\right).
    \]
\end{lemma}
\begin{proof}
    We show this by induction: note that the sum of the entries in $\sigma^{(1)}$ is 0, and so is this expression.

    For $k > 1$, the sum of the entries in $\sigma^{(k)}$ is 
    \[
        \frac{1}{2}(1+\sqrt{2})^{-2(k-1)}\left(\sum_{i=1}^{2^{k-2}-1}\pi^{(k-2)}_i + \beta_k\right) + \Sigma_{k-1},
    \]
    where $\Sigma_{k-1}$ is the sum of the entries in $\sigma^{(k-1)}$. 

    By \cref{lem:sumpi}, and the induction hypothesis, this is
    \[
        \frac{1}{2}(1+\sqrt{2})^{-2(k-1)}\left((\beta_{k-1} - 2) + \beta_k\right) + 
        \frac{1}{2}\left(1-(1+\sqrt{2})^{2-k}\right).
    \]
    Note that $(\beta_{k-1} - 2) + \beta_k = (2+\sqrt{2})(1+\sqrt{2})^{k-2}$, so that this becomes
    \begin{align*}
        \frac{1}{2}(2+\sqrt{2})(1+\sqrt{2})^{-k}+
        \frac{1}{2}\left(1-(1+\sqrt{2})^{2-k}\right)&= 
        \frac{1}{2}\left(1-(1+\sqrt{2})^{1-k}\right).\qedhere
    \end{align*}
\end{proof}

\begin{lemma}
\label{lem:rho_sum}
    For $k\geq 0$,
    the sum of the entries in $\rho_k$ is 
    \[
        \frac{\beta_{k+1}^2}{2(1+\sqrt{2})} + (1+\sqrt{2})^{k-2}.
    \]
\end{lemma}
\begin{proof}
    A simple calculation shows that this holds for $k = 0$ (see \nextmathematica).
    Now suppose $k \geq 1$. Using \Cref{lem:sumpi} and \Cref{lem:sumsigma}, we see that the sum of the entries in $\rho_k$ is
    \[
        (1+\sqrt{2})^{k-2}+\left(\frac{1}{2}(1+\sqrt{2})^{2k-1}\left(1-(1+\sqrt{2})^{1-k}\right)-\frac{\beta_k-2}{2(1+\sqrt{2})}\right)+\left(\beta_{k+1}-2\right)+1
    \]
    This is equal to the claimed expression (see \nextmathematica).
\end{proof}
\begin{lemma}
    \label{lem:sum_of_wk}
    For $k\geq 0$, the sum of the entries of $w_k$ is $\sqrt{\mu_k-1}$.
\end{lemma}
\begin{proof}
    We proceed by induction: the sum of the entries of $w_0$ is $1 = \sqrt{\mu_0-1}$.
    
    By expanding the definition and applying \Cref{lem:sumpi} and the inductive hypothesis, the sum of the entries of $w_k$ is $\frac{\beta_k-2}{\sqrt{\mu_k-1}} + \frac{\beta_k}{\sqrt{\mu_k-1}}+\sqrt{\mu_{k-1}-1}$.
    This is equivalent to the claimed expression by \cref{lem:relating_muprev_to_mu_and_beta}.
\end{proof}

\subsubsection{Computing Row Sums}
\begin{lemma}
    \label{lem:row_sums}
    We will give the sum of the entries in each row, dividing into the cases above.

\textbf{Case 1}: $i+1 <2^k$ and $i+1$ is not a power of 2. The sum of the entries of $\lambda_i$ is
\[
    \frac{\mu_{z(i)+1}-1}{(1+\sqrt{2})^{2(z(i)-p(i))+3}} \left(\frac{\beta_{\nu(i+1)+1}^2}{2(1+\sqrt{2})} + (1+\sqrt{2})^{\nu(i+1)-2}\right).
\]

\textbf{Case 2}: $i+1 > 2^k$ and $2^{k+1}-1 - i$ is not a power of 2. The sum of the entries of $\lambda_i$ is
\[
    \frac{(1+\sqrt{2})^{2(p(i)+z(2^{k+1} - 2- i)-k)-1}}{\mu_{z(2^{k+1} - 2- i)+1}-1} \left(\frac{\beta_{\nu(i+1)+1}^2}{2(1+\sqrt{2})} + (1+\sqrt{2})^{\nu(i+1)-2}\right).
\]

\textbf{Case 3}: The sum of the entries of $\lambda_0$ is 2. For $i = 2^{\ell}-1$ with $0 < \ell < k$, 
the sum of the entries in $\lambda_i$ is
\[
     \frac{1}{2}(\mu_{\ell}-1)-1 + \frac{\mu_{\ell+1}-1}{2(1+\sqrt{2})^\ell}.
\]

\textbf{Case 4}: $i = 2^k-1$. The sum of the entries of $\lambda_i$ is

\[
     \frac{\mu_k-1}{2}.
\]

\textbf{Case 5}: $i = 2^{k+1}-2^{\ell}-1$ with $\ell < k$. The sum of the entries of $\lambda_i$ is
\[
    \frac{\beta_{\ell+1}^2}{2(\mu_{\ell+1}-1)}.
\]
    
\end{lemma}
\begin{proof}
Cases 1 and 2 follow directly by definition and \cref{lem:rho_sum}.
    The expressions for Cases 3 and 4 follow by adding up the partial sums computed in the previous subsection. See \nextmathematica{} and \nextmathematica.

    For Case 5, we combine the expressions for the partial sums (see \nextmathematica) to get the row sum in the form 
    \begin{align*}
    \frac{\beta_{\ell+1}^2}{2(\mu_{\ell+1}-1)} + \frac{1}{2(\mu_{\ell+1}-1)}((1+\sqrt{2})^{\ell-1} - (1+\sqrt{2})^{\ell+1} -2 (1+\sqrt{2})^{\ell}) + \left(1 - \sqrt{\frac{\mu_{\ell}-1}{\mu_{\ell+1} - 1}}\right).
    \end{align*}
    Our goal is to show that the portion of this expression after the $\frac{\beta_{\ell+1}^2}{2(\mu_{\ell+1}-1)}$ term is zero.
    By \Cref{lem:relating_muprev_to_mu_and_beta}, the final term is
    \begin{align*}
        1- \sqrt{\frac{\mu_{\ell}-1}{\mu_{\ell+1}-1}}= \frac{2(\beta_{\ell+1}-1)}{\mu_{\ell+1}-1}.
    \end{align*}
    Combining these expressions proves the claim (see \nextmathematica).
\end{proof}

\subsection{Column Sums} \label{subsec:ColSums}
\subsubsection{The Support of the $j$th Column}
\label{subsub:col_support_lemmas}
For a given $j \le 2^{k+1}-1$, we let 
\[
    S_j^- = \{i < j : \lambda_{i,j} \neq 0\}
    \qquad\text{and}\qquad
    S_j^+ = \{i > j : \lambda_{i,j} \neq 0\}
\]
denote the indices above $j$ and indices below $j$ in the support of the $j$th column.
The following lemmas give computational descriptions of these sets that will be useful in computing the column sums. We will give their proofs in \Cref{sec:supportFacts}

\begin{lemma}
    \label{lem:nonzero_entries_above}
    Suppose that $j\in[1,2^k - 1]$ has the binary expansion $j = \sum_{a=0}^{z} b_{a} 2^{a}$, where $b_i \in \{0,1\}$ and $b_z = 1$, then
    \[
        S_j^- = \set{\sum_{a = r}^{z} b_a 2^a - 1 : r \in [0,z]}.
    \]
\end{lemma}
\begin{lemma}
    \label{lem:recursive_support_case_2}
Suppose $2^z - 1 < j < 2^{z+1} -1$ with $z < k$. If $k>\ell= z+1$, then
\begin{align*}
    S^+_{\rev(j+2^\ell)} = \set{i-2^\ell:\, i\in S^+_{\rev(j)}}.
\end{align*}
On the other hand, if $k>\ell > z+1$, then
\begin{align*}
    S^+_{\rev(j+2^\ell)} = \set{\rev(2^\ell - 1)} \cup \set{i-2^\ell:\, i\in S^+_{\rev(j)}}.
\end{align*}
\end{lemma}

\begin{lemma}
    Suppose $2^z - 1 \leq j < 2^{z+1} - 1$ with $z<k$.
    Then, $i\in S_j^+$ if and only if $i\in [j+1, 2^{k+1}-2]$ and  $i - j \le 2^{\nu(i+1)-1}$.
    In particular, if $i \in S_j^+$, then $j < i \le 2^{z+1}-1$.
    If in addition $j = 2^z-1$, then $S_j^+$ is the singleton set $\{2^{z+1}-1\}$
\end{lemma}
\begin{lemma}
\label{lem:support_below_recursive_NE}
    Suppose that $2^z - 1 \le j < 2^{z+1}-1$ with $z < k$. Let $z<\ell<k$.
    We then have that 
    \[
        S_{j + 2^{\ell}}^+ = \{2^{\ell} + i : i \in S_j^+\} \cup \{2^{\ell+1}-1\}.
    \]
\end{lemma}

\subsubsection{Computing Column Sums}

\begin{lemma}
\label{lem:colsums_above_left}
    Fix $1\leq j \leq 2^{k}-1$.
    First suppose $j+1$ is not a power of $2$ and let $z$ so that $2^z - 1 < j < 2^{z+1}-1$. Let $p$ denote the number of ones in the binary expansion of $j+1$.
    Then,
    \[
        \sum_{i \in S_j^-} \lambda_{i,j} = \frac{1}{2} (1+\sqrt{2})^{-5+2p+\nu(j+1)-2z}\beta_{\nu(j+1)+2}(\mu_{z+1}-1).
    \]
    On the other hand, if $j = 2^{z}-1$ for some $z=1,\dots, k$, then
    \[
        \sum_{i \in S_j^-} \lambda_{i,j} = \frac{1}{2}(\mu_{z}-1).
    \]
\end{lemma}
\begin{proof}
    We begin with the case where $j+1$ is not a power of $2$. Let $j+1 = \sum_{a=0}^{z} b_{a} 2^{a}$ be the binary expansion of $j+1$.
    Since $2^{\nu(j+1)}$ is the largest power of 2 dividing $j+1$, it follows that
    \[
        j+1 = 2^z + \sum_{a=\nu(j+1)+1}^{z-1} b_a 2^a + 2^{\nu(j+1)},
    \]
    which implies that
    \[
        j = 2^z + \sum_{a=\nu(j+1)+1}^{z-1} b_a 2^a + \sum_{a = 0}^{\nu(j+1)-1} 2^a.
    \]

    Let $i \in S_j^-$.
By \cref{lem:nonzero_entries_above}, there are three cases: either $i + 1 = 2^z$; $i+1 = 2^z + \sum_{a=\ell}^{z-1} b_a 2^a$ for some $z > \ell > \nu(j+1)$, or $i+1 = 2^z + \sum_{a=\nu(j+1)+1}^{z-1} b_a 2^a + \sum_{a=\ell}^{\nu(j+1)-1} 2^a$ for $\ell < \nu(j+1)$.

    \textbf{Case (i)}:
    Let $i = 2^z-1$.
    As $j+1$ is not a power of $2$, we have that $\nu(j+1)<z$ and $\nu(j-i) = \nu(j+1)$. Thus, by definition of $\lambda$, we have
\[
        \lambda_{i,j} = \frac{\mu_{z+1}-1}{2(1+\sqrt{2})^{2z}}\beta_{\nu(j+1)}.
    \]

    \textbf{Case (ii)}:
    In this case, neither $j+1$ nor $i+1$ are powers of two. Furthermore, $i+1$ is given explicitly as
$i+1 = 2^z + \sum_{a=\ell}^{z-1} b_a 2^a$ for some $z > \ell > \nu(j+1)$. Then $0 < j - i < 2^{\ell}$, implying that
    \[
        \lambda_{i,j} = \frac{\mu_{z+1}-1}{(1+\sqrt{2})^{2(z-p_i)+3}} \beta_{\nu(j -i)} = 
        \frac{\mu_{z+1}-1}{(1+\sqrt{2})^{2(z-p_i)+3}} \beta_{\nu(j+1)},
    \]
    where $p_i$ is the number of ones in the binary expansion of $i+1$ and we note that $\nu(j - i) = \nu((j+1)-(i+1)) = \nu(j+1)$.
    Now, note that if we sum over all $i$ of the form of case 2, there is exactly one such term for each possible value of $p_i$ from $2$ through $p-1$. That is, if we add all such $\lambda_{i,j}$, we obtain (see \nextmathematica)
    \[
        \sum_{p_i=2}^{p-1} \frac{\mu_{z+1}-1}{(1+\sqrt{2})^{2(z-p_i)+3}} \beta_{\nu(j+1)} = 
        \frac{1}{2} \left(\left(1+\sqrt{2}\right)^{2 (p-z-2)}  - (1+\sqrt{2})^{-2z}\right)(\mu_{z+1}-1)\beta_{\nu(j+1)}.
    \]

    \textbf{Case (iii)}:
    If 
    \[
        i+1 = 2^z + \sum_{a=\nu(j+1)+1}^{z-1} b_a 2^a + \sum_{a=\ell}^{\nu(j+1)-1} 2^a,
    \]
    for some $0 \leq \ell < \nu(j+1)$, then we have that $j-i = 2^{\ell}$, so that
    \[
        \lambda_{i,j} = \frac{\mu_{z+1}-1}{(1+\sqrt{2})^{2(z-p_i)+3}}.
    \]
    Once again, if we add up all such terms, we collect one for each possible value of $p_i$ from $p$ to $p+\nu(j+1)-1$, yielding (see \nextmathematica)
    \[
        \sum_{p_i=p}^{p+\nu(j+1)-1} \frac{\mu_{z+1}-1}{(1+\sqrt{2})^{2(z-p_i)+3}} = \frac{1}{2}(1+\sqrt{2})^{2(p-z-2)}\left((1+\sqrt{2})^{2\nu(j+1)} - 1\right)\left(\mu_{z+1}-1\right).
    \]

    If $j+1$ is not a power of two, then adding the sums in the three cases yields (see \nextmathematica)
    \[
        \frac{1}{2} (1+\sqrt{2})^{(-5+2p+\nu(j+1)-2z)}\beta_{\nu(j+1)+2}(\mu_{z+1}-1),
    \]
    as desired.

    Now, we consider the case when $j = 2^{z}-1$. By \Cref{lem:nonzero_entries_above}, $S_j^- = \{2^{z} - 2^{\ell} - 1 : 0 \le \ell < z\}$. 
    If $i = 2^{z-1}-1$, then
    \[
        \lambda_{i,j} = \frac{\mu_{z}-1}{2(1+\sqrt{2})^{2(z-1)}}.
    \]
    If $i = 2^{z} - 2^{\ell} - 1$ for $0 \le \ell < z-1$, then
    \[
        \lambda_{i,j} = \frac{\mu_{z}-1}{(1+\sqrt{2})^{2\ell+1}}.
    \]
    Adding all of these terms yields $\frac{1}{2}(\mu_{z}-1)$ (see \nextmathematica).
\end{proof}

\begin{lemma}
    \label{lem:SplusSum}
    Fix $j$ so that $2^z - 1 < j < 2^{z+1}-1$ where $z<k$. If there are $p$ one's in the binary expansion of $j+1$,
    then
    \[
        \sum_{i \in S_j^+} \lambda_{i,j} =
        \frac{1}{2} \left(1+\sqrt{2}\right)^{2 (p-z-2)}\beta_{\nu(j+1)+2} (\mu_{z+1}-1) .
    \]
\end{lemma}
\begin{proof}
    We will show this by induction on the number of ones in the binary expansion of $j+1$. 

    If there are exactly 2 one's in the binary expansion of $j+1$, then
    \[
        j = 2^{z} + 2^{\ell} - 1,
    \]
    for some $\ell < z$.
    In this case, \cref{lem:support_below_recursive_NE} implies $S_j^+ = \{2^{z+1}-1, 2^z+2^{\ell+1}-1\}$.
    We consider two cases: either $2^z+2^{\ell+1}-1 = 2^{z+1}-1$, i.e. $\ell = z-1$, or it does not.

    If $\ell = z-1$, then $S_j^+ = \{2^{z+1}-1\}$, and 
    \[
        \sum_{i \in S_j^+} \lambda_{i,j} = \lambda_{2^{z+1}-1, 2^{z} + 2^{z-1} - 1} = 
        \frac{1}{2} \left(1+\sqrt{2}\right)^{-2z}\beta_{\ell+2} (\mu_{z+1}-1),
    \]
    as desired.

    On the other hand, if $\ell \neq z-1$, then $S_j^+ = \{2^{z+1}-1, 2^z+2^{\ell+1}-1\}$, where $\ell < z-1$. So,
    \[
        \sum_{i \in S_j^+} \lambda_{i,j} = \lambda_{2^{z+1}-1, 2^{z} + 2^{\ell} - 1} + \lambda_{2^z+2^{\ell+1}-1, 2^z+2^{\ell}-1}.
    \]
    From our definitions,
    \[
        \lambda_{2^{z+1}-1, 2^{z} + 2^{\ell} - 1} = \frac{1}{2}(1+\sqrt{2})^{-2z} \beta_{\ell}(\mu_{z+1}-1),
    \]
    and
    \[
        \lambda_{2^{z} +2^{\ell+1}-1, 2^{z} + 2^{\ell} - 1} = \frac{\mu_{z+1}-1}{(1+\sqrt{2})^{2(z-2)-\ell+4}}.
    \]
    Summing up these two terms gives
    \[
        \lambda_{2^{z+1}-1, 2^{z} + 2^{\ell} - 1} + \lambda_{2^z+2^{\ell+1}-1, 2^z+2^{\ell}-1} =
        \frac{1}{2} \left(1+\sqrt{2}\right)^{-2z}\beta_{\ell+2} (\mu_{z+1}-1)
    \]
    as can be seen in \nextmathematica.

    Now, we assume $p > 2$.
    Let $j = 2^z + j'$, where $j' < 2^z-1$ has $p-1$ one's in its binary expansion.
    We have by \cref{lem:support_below_recursive_NE} that $S_j^+ = \{i' + 2^z : i' \in S_{j'}^+\} \cup \{2^{z+1}-1\}$.
    Once again we consider two cases: either $2^{z+1}-1 \in \{i' + 2^z : i' \in S_{j'}^+\}$, or it is not.

    Assume that $2^{z+1}-1 \in \{i' + 2^z : i' \in S_{j'}^+\}$. This implies that $2^z-1 \in S_{j'}^+$, or equivalently that $z(j') = z-1$.
    In this case,
    \[
        \sum_{i \in S_j^+} \lambda_{i,j} = \sum_{i' \in S_{j'}^+} \lambda_{2^z + i', 2^z+j'}.
    \]
    Now, suppose $i'\in S^+_{j'}$ is not a power of 2. Since $z(i) = z(i') + 1$, and $p(i) = p(i') + 1$, \Cref{lem:recurrence} implies that
    \[
        \lambda_{2^z + i', 2^z+j'} = 
        \frac{\mu_{z+1}- 1}{\mu_{z}- 1}\lambda_{i', j'}.
    \]
    The only element of $S^+_{j'}$ and that is one less than a power of 2 is $i' = 2^z-1$. If $2^{z}-1 - 2^a > j' > 2^{z} - 1 - 2^{a+1}$, then $2^{z+1} - 1 - 2^a > j > 2^{z+1} - 1 - 2^a$, and
    \begin{align*}
        \lambda_{2^z + i', 2^z+j'} &= \lambda_{2^{z+1} - 1, 2^z+j'}\\
        &= \frac{1}{2}(\mu_{z+1}-1)(1+\sqrt{2})^{-2(a+1)}\beta_{\nu(j+1)}\\
        &= 
        \frac{\mu_{z+1}+ 1}{\mu_{z}+ 1}\lambda_{2^z-1, j'}.
    \end{align*}
    Therefore, 
    \[
        \sum_{i' \in S_{j'}^+} \lambda_{2^z + i', 2^z+j'} = \frac{\mu_{z+1}+ 1}{\mu_{z}+ 1}\sum_{i' \in S_{j'}^+} \lambda_{i', j'}=
        \frac{1}{2} \left(1+\sqrt{2}\right)^{2 (p-z-2)}\beta_{\nu(j+1)+2} (\mu_{z+1}-1),
    \]
    as desired.

    Now assume that $2^{z+1}-1 \not \in \{i' + 2^z : i' \in S_{j'}^+\}$. This is equivalent to $2^{z'+1}-1 > j' > 2^{z'}-1$ for some $z' < z-1$.
    Now, we have that 
    \[
        \sum_{i \in S_j^+} \lambda_{i,j} = \sum_{i' \in S_{j'}^+} \lambda_{2^z + i', 2^z+j'} + \lambda_{2^{z+1}-1, j}.
    \]
    By \Cref{lem:recurrence}, we note that for all $i' \in S_{j'}^+$ other than $i' = 2^{z'+1}-1$, 
    \[
        \lambda_{2^z + i', 2^z+j'} = 
        (1+\sqrt{2})^{2(1+z'-z)}\frac{\mu_{z+1}-1}{\mu_{z'+1}-1}\lambda_{i', j'}.
    \]
    It remains to consider $\lambda_{2^z+2^{z'+1}-1, 2^z+j'}$ and $\lambda_{2^{z+1}-1, 2^z+j'}$.

    Note that $2^{z+1} - 2^{z-1} > j > 2^{z+1} - 2^z$ and  that $j' + 1 < 2^{z'-1}$ is not a power of 2, so
    \[
        \lambda_{2^{z+1}-1, 2^z+j'} = \frac{1}{2}(1+\sqrt{2})^{-2z}\beta_{\nu(j'+1)}(\mu_{z+1}-1).
    \]

    Also, if $2^{z'+1}-1 - 2^a > j' > 2^{z'+1} - 1 - 2^{a+1}$, then
    \[
        \lambda_{2^z+2^{z'+1}-1, 2^z+j'} = \frac{\mu(z+1) - 1}{(1+\sqrt{2})^{2(z-2)+4}}\left((1+\sqrt{2})^{2(z'+1-a)}-1 \right)\beta_{\nu(j'+1)}.
    \]
    We see then that 
    \begin{align*}
        \lambda_{2^{z+1}-1, 2^z+j'} +  \lambda_{2^z+2^{z'+1}-1, 2^z+j'} &= 
         \frac{\mu(z+1) - 1}{(1+\sqrt{2})^{2z}}\left((1+\sqrt{2})^{2(z'+1-a)} \right)\beta_{\nu(j'+1)}\\
         &= 
         (1+\sqrt{2})^{2(1+z'-z)}\frac{\mu_{z+1}-1}{\mu_{z'+1}-1}\lambda_{2^{z'+1}-1, j'},
    \end{align*}
    \begin{align*}
        \lambda_{2^z+2^{z'}-1, 2^z+j'} &= \frac{\mu_{z+1}-1}{(1+\sqrt{2})^{2z-1}} (\rho_{z'})_{j' - 2^{z'}}\\
        &=
        \frac{\mu_{z+1}-1}{(1+\sqrt{2})^{2z-1}} 
        \left((1+\sqrt{2})^{2z'-1}(\sigma^{(z')})_{j' - 2^{z'}} - \frac{\beta_{\nu(j'+1)}}{2(1+\sqrt{2})}\right)\\
         &= 
        \frac{\mu_{z+1}-1}{(1+\sqrt{2})^{2z-1}} 
        \left(\frac{1}{\mu_{z+1}-1}(1+\sqrt{2})^{2z'-1}\lambda_{2^{z'}-1, j'}  - \frac{\beta_{\nu(j'+1)}}{2(1+\sqrt{2})}\right)\\
         &= 
        \frac{\mu_{z+1}-1}{(\mu_{z+1}-1)(1+\sqrt{2})^{2z-2z'-2}} \lambda_{2^{z'}-1, j'} 
         - \frac{\mu_{z+1}-1}{(1+\sqrt{2})^{2z}}\frac{\beta_{\nu(j'+1)}}{2}.
    \end{align*}

    Note that $\lambda_{2^z+2^{z'}-1, 2^z+j'} + \lambda_{2^{z+1}-1, 2^z+j'} = (1+\sqrt{2})^{2(1+z'-z)}\frac{\mu_{z+1}-1}{\mu_{z+1}-1} \lambda_{2^{z'}-1, j'}$, so that 
    \begin{align*}
        \sum_{i \in S_j^+} \lambda_{i,j} &= \sum_{i' \in S_{j'}^+} \lambda_{2^z + i', 2^z+j'} + \lambda_{2^{z+1}-1, j}\\
        &= (1+\sqrt{2})^{2(1+z'-z)}\frac{\mu_{z+1}-1}{\mu_{z+1}-1}  \sum_{i' \in S_{j'}^+} \lambda_{i', j'}\\
        &=
        \frac{1}{2} \left(1+\sqrt{2}\right)^{2 (p-z-2)}\beta_{\nu(j+1)+2} (\mu_{z+1}-1) .\qedhere
    \end{align*}
\end{proof}

\begin{lemma}
    \label{lem:rowcolhighpow2_expression}
    Suppose $j = 2^{k+1} - 2^\ell - 1$ for some $0\leq \ell < k$. Then, the sum of the entries in the $j$th column of $\lk$ is
    \begin{align*}
        \frac{\beta^2_{\ell+1}}{2(\mu_{\ell+1}-1)}.
    \end{align*}
\end{lemma}
\begin{proof}
Inspecting the support of $\lk$, we have that $S_j^+ = \varnothing$, so we need only consider the sum of the entries corresponding to element of $S_j^+$.

Note that the binary expansion of $j$ is given by $j = \sum_{a=0}^k b_a2^a$ where $b_a = 1$ if and only if $a\neq \ell$.
By \Cref{lem:nonzero_entries_above},
\begin{align*}
    S_j^{-} = \set{2^{k+1} - 2^\ell - 2^s - 1:\, s\in[0,\ell - 1]}\cup \set{2^{k+1} - 2^s - 1:\, s\in[\ell+1,k-1]}\cup\set{2^k - 1}.
\end{align*}

Suppose $i$ belongs to the first set of indices, i.e., $i = 2^{k+1} - 2^\ell - 2^s - 1$ for some $s\in[0,\ell -1]$.
In this case, $\lk_i$ is defined by Case 2.
We have $p(i) = (k - \ell) + (\ell-s) = k - s$, $\nu(i+1) = \nu(2^{k+1} - 2^\ell - 2^s) = s$,
and
$z(2^{k+1} - 1 - i) = z(2^\ell + 2^s) = \ell$.
Additionally, the $j$th column depends on the entry of $\lk_i$ that is $2^s$ entries to the right of the diagonal. The corresponding entry of $\rho_{\nu(i+1)}$ is 1. Thus, summing up the nonzero entries $\lk_{i,j}$ where $i$ falls in the first set of indices, we get
\begin{align*}
    \sum_{s = 0}^{\ell-1} \frac{(1+\sqrt{2})^{2((k-s)+\ell-k)-1}}{\mu_{\ell+1} - 1} &= \frac{(1+\sqrt{2})^{2\ell}}{2(\mu_{\ell+1} - 1)} \left(1 - (1+\sqrt{2})^{-2\ell}\right).
\end{align*}
This formula holds also in the case $\ell = 0$.

Now, suppose $i$ belongs to the second set of indices, i.e., $i = 2^{k+1} - 2^s - 1$ for some $s\in[\ell+1,k - 1]$. In this case, $\lk_i$ is defined by Case 5. Note that $\lk_{i,j}$ depends on the entry of $w_s$ that is the $(2^\ell + 1)$th entry from the right. As the $2^{\ell + 1}$ final entries of $w_s$ are given by $w_{\ell+1}$, it holds that this entry is given by $\frac{\beta_{\ell+1}}{\sqrt{\mu_{\ell+1}-1}}$ for all $w_s$. Summing up the entries of this form, we get
\begin{align*}
    \left(\frac{\beta_{\ell+1}}{\sqrt{\mu_{\ell+1}-1}}\right)\sum_{s=\ell+1}^{k-1} \left(\frac{1}{\sqrt{\mu_s - 1}} - \frac{1}{\sqrt{\mu_{s+1} - 1}}\right) = \left(\frac{\beta_{\ell+1}}{\sqrt{\mu_{\ell+1}-1}}\right)\left(\frac{1}{\sqrt{\mu_{\ell+1} - 1}} - \frac{1}{\sqrt{\mu_{k} - 1}}\right).
\end{align*}
This formula also holds in the case where $l = k - 1$.

The final entry in the $j$th column is the $\lk_{i,j}$ entry where $i = 2^k - 1$. Once again, the entry of $w_k$ that $\lk_{i,j}$ depends on is given by $\frac{\beta_{\ell+1}}{\sqrt{\mu_{\ell+1}-1}}$. Thus,
\begin{align*}
    \lk_{i,j} = \frac{1}{\sqrt{\mu_k - 1}}\frac{\beta_{\ell+1}}{\sqrt{\mu_{\ell+1}-1}}.
\end{align*}

Summing up all entries in the column gives
\begin{align*}
    \frac{(1+\sqrt{2})^{2\ell}}{2(\mu_{\ell+1} - 1)} \left(1 - (1+\sqrt{2})^{-2\ell}\right) + \frac{\beta_{\ell+1}}{\mu_{\ell+1}-1} &= \frac{(1+\sqrt{2})^{2\ell}  + 2 (1+\sqrt{2})^\ell + 1}{2(\mu_{\ell+1}-1)} = \frac{\beta_{\ell+1}^2}{2(\mu_{\ell+1}-1)}.\qedhere
\end{align*}
\end{proof}

\begin{lemma}
\label{lem:telescoping}
Suppose $0\leq r < \ell < 2^k-1$ so that $r+1$ is not a power of $2$ and so that $\ell > 2^{z(r) + 1}-1$. If $\rev(\ell)+1$ is not a power of 2 then
\begin{align*}
    \sum_{i = 0}^{\ell} \lk_{i,\rev(r)} = \frac{\beta_{\nu(r+1)}}{\sqrt{\mu_{z(\ell)+1}-1}\sqrt{\mu_{z(r)+1} - 1}}.
\end{align*}
If $\rev(\ell)+1$ is a power of 2, then 
\begin{align*}
    \sum_{i = 0}^{\ell} \lk_{i,\rev(r)} = \frac{\beta_{\nu(r+1)}}{\sqrt{\mu_{z(\ell)} - 1}\sqrt{\mu_{z(r)+1} - 1}}.
\end{align*}
\end{lemma}
\begin{proof}
If $\rev(\ell)+1$ is not a power of 2, then the only nonzero entries of the form $\lambda_{i, \rev(r)}$ with $i < \rev(\ell)$ are those of the form $\rev(2^a-1)$ with $a > z(\ell)$.
In this case,
\begin{align*}
    \sum_{i = 0}^{\ell} \lk_{i,\rev(r)} &= \sum_{a = z(\ell)+ 1}^{k} \lk_{\rev(2^a-1), \rev(r)}\\
    &= 
    \sum_{a = z(\ell)+1}^{k-1} \frac{\beta_{\nu(r+1)}}{\sqrt{\mu_{z(r)+1} - 1}}\left(\frac{1}{\sqrt{\mu_\tau - 1}} - \frac{1}{\sqrt{\mu_{\tau+1} - 1}}\right) + \frac{\beta_{\nu(r+1)}}{\sqrt{\mu_{z(r)+1} - 1}}\left(\frac{1}{\sqrt{\mu_k - 1}}\right)\\
    &= 
    \frac{\beta_{\nu(r+1}}{\sqrt{\mu_{z(\ell)+1} - 1}\sqrt{\mu_{z(r)+1} - 1}}.
\end{align*}
The calculation is identical if $\ell+1$ is a power of 2, except that we add $\lambda_{\ell,r}$.
\end{proof}

\begin{lemma}
\label{lem:sum_above_diagonal_case_2}
Suppose $0\leq r < 2^k-1$ and $r+1$ is not a power of $2$. Then,
\begin{align*}
    \sum_{i\in S^-_{\rev(r)}} \lk_{i,\rev(r)} &= \frac{1}{2(\mu_{z(r) + 1} - 1)}
    \left((1+\sqrt{2})^{2z(r)-\nu(r+1)-2p(r)+1}\beta_{\nu(r+1)+2} + \beta_{\nu(r+1)}\right).
\end{align*}
\end{lemma}
\begin{proof}
Let $\sum_{a=0}^{k} b_a2^a$ be the binary expansion of $\rev(r)$.
Equivalently if $\sum_{a=0}^k c_a2^a$ is the binary expansion of $r+1$, then $b_a = (1-c_a)$ for all $a\in[0,k]$.
Note that $c_k = 0$ so that $b_k = 1$.
The set
\begin{align*}
    S_{\rev(r)}^- = \set{\sum_{a=\tau}^k b_a 2^a - 1 :\, \begin{array}{l}
        \tau\in[0,k]\\
        b_\tau = 1
    \end{array}}.
\end{align*}
Suppose $\tau\in[0,k]$ and $b_\tau = 1$. Let $i = \sum_{a=\tau}^k b_a 2^a - 1$.
We have that $0\leq\nu(r+1)< z(r)<k$. 
We enumerate the possible values of $\lk_{i , \rev(r)}$ according to $\tau\in[z(r)+1,k]$, $\tau\in[\nu(r+1)+1,z(r)-1]$ and $\tau\in[0,\nu(r+1)-1]$.

Note that \Cref{lem:telescoping} implies that
\[
    \sum_{\tau = z(r)+1}^k \lambda_{\rev(2^{\tau}-1), \rev(r)} = \frac{\beta_{\nu(r+1)}}{\mu_{z(r) + 1} - 1}.
\]

Now, suppose $\tau \in[0, \nu(r+1)-1]$, then $\sum_{a=0}^{\tau - 1}b_a2^a + 1 = 2^{\tau}$. In this case  $\lk_i$ is defined according to Case 2, $p(i) = k+1 - p(r) - \tau$, $\nu(i+1) = \tau$, and $z(2^{k+1} - 2 - i)=z(r)$. Then,
\begin{align*}
    \lk_{i,\rev(r)} &= \frac{(1+\sqrt{2})^{2(z(r)- p(r) - \tau)+1}}{\mu_{z(r) + 1} - 1}.
\end{align*}
Summing up entries of this form gives
\begin{align*}
    \sum_{\tau = 0}^{\nu(r+1) -1}\frac{(1+\sqrt{2})^{2(z(r)- p(r) - \tau)+1}}{\mu_{z(r) + 1} - 1}
    &= \frac{(1+\sqrt{2})^{2(z(r)-p(r))+1}}{\mu_{z(r)+1}-1} \sum_{\tau = 0}^{\nu(r+1) - 1} (1+\sqrt{2})^{-2\tau}\\
    &= \frac{(1+\sqrt{2})^{2(z(r)-p(r))+2}}{2(\mu_{z(r)+1}-1)} \left(1 - (1+\sqrt{2})^{-2\nu(r+1)}\right).
\end{align*}

Now suppose $\tau \in[\nu(r+1)+1,z(r)-1]$ satisfies $b_\tau = 1$. In this case, $\lk_i$ is defined according to Case 2, $\nu(i+1) = \tau$, and $z(2^{k+1} - 2 - i)=z(r)$. Then,
\begin{align*}
    \lk_{i,\rev(r)}
&= \frac{(1+\sqrt{2})^{2(p(i) + z(r) - k)-1}}{\mu_{z(r)+1} - 1}\beta_{\nu(r+1)}.
\end{align*}
When we sum up over entries of this form, the value of $p(i)$ is in bijection with $[k - z(r)+1, k-\nu(r+1)-p(r)+1]$. Thus, the final part of this summation is
\begin{align*}
    \sum_{p=k-z(r)+1}^{k-\nu(r+1)-p(r)+1}\frac{(1+\sqrt{2})^{2(p + z(r) - k)-1}}{\mu_{z(r)+1} - 1}\beta_{\nu(r+1)}
&= \frac{\beta_{\nu(r+1)}}{2(\mu_{z(r)+1}-1)} \left((1+\sqrt{2})^{2(z(r)-\nu(r+1)-p(r))+2} - 1\right).
\end{align*}

Finally, summing up all entries gives the desired claim (see \nextmathematica).
\end{proof}

\begin{lemma}\label{lem:case2_below_diagonal_sum}
    Let $0\leq r < 2^k - 1$ and suppose $r+1$ is not a power of $2$. Then,
    \begin{align*}
        \sum_{i=S^+_{\rev(r)}} \lk_{i,\rev(r)} &= \frac{1}{2(\mu_{z(r)+1} - 1)}\bigg((1+\sqrt{2})^{2(z(r) - p(r) - \nu(r+1) + 1)}\beta^2_{\nu(r+1)+1}\\
        &\qquad\qquad + (2 - \beta_{\nu(r+1)+2})(1+\sqrt{2})^{2z(r) - 2p(r) - \nu(r+1) + 1} - \beta_{\nu(r+1)}\bigg).
    \end{align*}
    \end{lemma}
    \begin{proof}
    We will induct on $p(r)$. As $r+1$ is not a power of two, we have $p(r)\geq 2$.
    
    First, suppose $p(r) = 2$.
    We compute directly that $S_{\rev(r)}^+ = \set{\rev(2^{z(r)} - 1)}$.
    Then
    \begin{align*}
        &\lk_{\rev(2^{z(r)} - 1),\rev(r)}\\
        &= \begin{cases}
        \frac{(1+\sqrt{2})^{z(r)-1}}{\mu_{z(r)+1}-1} &\text{if }\nu(r+1) + 1 =z(r)\\
        \frac{1}{2(\mu_{z(r)+1}-1)}\left((1+\sqrt{2})^{2(z(r)-\nu(r+1)-1)}\beta_{\nu(r+1)+2} - \beta_{\nu(r+1)}\right) &\text{if }\nu(r+1) + 1 < z(r).
        \end{cases}
    \end{align*}
    One can verify that the second expression coincides with the first expression when $\nu(r+1)+1 = z(r)$ (see \nextmathematica). Thus,
    \begin{align*}
        \lk_{\rev(2^{z(r)} - 1),\rev(r)} &= \frac{1}{2(\mu_{z(r)+1}-1)}\left((1+\sqrt{2})^{2(z(r)-\nu(r+1)-1)}\beta_{\nu(r+1)+2} - \beta_{\nu(r+1)}\right).
    \end{align*}
    We can also verify that this expression coincides with the claimed expression for $\sum_{i=S^+_{\rev(r)}} \lk_{i,\rev(r)}$ (see \nextmathematica).

    Now, suppose $p(r)\geq 3$ and let $r' = r - 2^z$.
    We have that $p(r') = p(r) - 1 \geq 2$ so that $r'+1$ is not a power of two. Let $z' = z(r')$. We will now apply \cref{lem:recursive_support_case_2}.
    There are two cases: where $z' = z - 1$ and where $z' < z-1$.

    In the first case, $z' = z-1$ and \cref{lem:recursive_support_case_2} states that $S^+_{\rev(r)} = \set{i - 2^z:\, i \in S^+_{\rev(r')}}$ so that
    \begin{align*}
        \sum_{i\in S^+_{\rev(r)}}\lk_{i,\rev(r)} &= \sum_{i\in S^+_{\rev(r')}}\lk_{i-2^z,\rev(r')-2^z}.
    \end{align*}
    For all $i\in S^+_{\rev(r')}$, it must hold that $ \rev(2^{z'+1} - 1)< \rev(r') <  i \leq  \rev(2^{z'} - 1)$.
    We may now apply \cref{lem:case_2_scaling} to get
    \begin{align*}
        \sum_{i\in S^+_{\rev(r)}}\lk_{i,\rev(r)} &= \sum_{i\in S^+_{\rev(r')}}\lk_{i-2^z,\rev(r')-2^z}\\
        &= \left((1+\sqrt{2})^{2(z-z' - 1)}\frac{\mu_{z'+1} - 1}{\mu_{z+1} - 1}\right)\sum_{i\in S^+_{\rev(r')}}\lk_{i,\rev(r')}\\
        &= \frac{1}{2(\mu_{z+1} - 1)}\bigg((1+\sqrt{2})^{2( z - p(r) - \nu(r+1) + 1)}\beta^2_{\nu(r+1)+1}\\
        &\qquad\qquad + (2 - \beta_{\nu(r+1)+2})(1+\sqrt{2})^{2z - 2p(r) - \nu(r+1) + 1} - \beta_{\nu(r+1)}\bigg).
    \end{align*}
    This is exactly the claimed expression.

    Now, consider the case $z' < z - 1$. In this case, \cref{lem:recursive_support_case_2} states that $S^+_{\rev(r)} = \set{\rev(2^z - 1)}\cup \set{i - 2^z:\, i \in S^+_{\rev(r')}}$. Again, for all $i\in S^+_{\rev(r')}$, it must hold that $ \rev(2^{z'+1} - 1)< \rev(r') <  i \leq  \rev(2^{z'} - 1)$.
    We may now apply \cref{lem:case_2_scaling} to get
    \begin{align*}
        \sum_{i\in S^+_{\rev(r)}}\lk_{i,\rev(r)} &= \lk_{\rev(2^z - 1), r} + \sum_{i\in S^+_{\rev(r')}}\lk_{i - 2^z, \rev(r') - 2^z}.
    \end{align*}
    We evaluate the two terms separately.
    First, note that $2^{z'} < r'+1 < 2^{z'+1}\leq 2^{z-1}$. Thus,
    \begin{align*}
        \lk_{\rev(2^z - 1), r} &= \frac{\beta_{\nu(r+1)}}{2(\mu_{z+1} - 1)}\left((1+\sqrt{2})^{2(z-z'-1)} - 1\right).
    \end{align*}
    For the second term, we apply the inductive hypothesis and \cref{lem:case_2_scaling} to get
    \begin{align*}
        \sum_{i\in S^+_{\rev(r')}}\lk_{i - 2^z, \rev(r') - 2^z} &= \frac{(1+\sqrt{2})^{2(z-z'-1)}}{2(\mu_{z+1}-1)}\bigg((1+\sqrt{2})^{2(z' - p(r) - \nu(r+1) + 2)}\beta^2_{\nu(r+1)+1}\\
        &\qquad\qquad + (2 - \beta_{\nu(r+1)+2})(1+\sqrt{2})^{2z' - 2p(r) - \nu(r+1) + 3} - \beta_{\nu(r+1)}\bigg).
    \end{align*}
    One can check that the sum of these two expressions is given by the claimed expression (see \nextmathematica).
\end{proof}

\subsection{Comparisons of Row and Column Sums}\label{subsec:RowColSums}
\begin{lemma}
    \label{lem:rowcol0}
    The sum of the zeroth row of $\lambda$ is one larger than the sum of the zeroth column of $\lambda$.
\end{lemma}
\begin{proof}
    The only entry of $\lambda_0$ is 2. 
    On the other hand, the only row which has an entry in column 0 is $\lambda_1$, and $\lambda_{1,0} = 1$, which implies the lemma.
\end{proof}
\begin{lemma}
    \label{lem:rowcollow}
    For $i = 1,\dots,2^k-2$, the sum of the $i$th row of $\lambda$ is equal to the sum of the $i$th column of $\lambda$.
\end{lemma}
\begin{proof}
    We first show this assuming $i = 2^z - 1$.
    It follows from our earlier work that the sum of the entries of the $i^{th}$ row is
    \[
        \frac{1}{2}(\mu_{z}-1)-1 + \frac{\mu_{z+1}-1}{2(1+\sqrt{2})^z}.
    \]
    On the other hand, the nonzero entries of the $i^{th}$ column are those indexed by $S_i^-$ and $2^{z+1}-1$.
    The sum of these entries is 
    \[
        \frac{1}{2}(\mu_{z}-1) + \left(\frac{\mu_{z+1}-1}{2(1+\sqrt{2})^z}-1\right).
    \]
    Thus, the row sum and the column sum are equal.
    
    We next show this assuming $i + 1$ is not a power of 2.
    If the number of one's in the binary expansion of $i+1$ is $p > 1$, then the sum of the entries in the $i^{th}$ row of $\lambda$ is
    \[
        (\mu_{z+1}-1)\left(\frac{1}{(1+\sqrt{2})^{2(z-p)+3}} \right)\left(\frac{\beta_{\nu(i+1)+1}^2}{2(1+\sqrt{2})} + (1+\sqrt{2})^{\nu(i+1)-2}\right).
    \]
    On the other hand, we have seen that the sum of the entries in the $i^{th}$ column is
    \begin{align*}
        &\sum_{j \in S_i^-} \lambda_{j,i} +
        \sum_{j \in S_i^+} \lambda_{j,i}\\
        &\quad= \frac{1}{2} (1+\sqrt{2})^{(-5+2p+\nu(i+1)-2z)}\beta_{\nu(i+1)+2}(\mu_{z+1}-1)
        + \frac{1}{2} \left(1+\sqrt{2}\right)^{2 (p-z-2)}\beta_{\nu(i+1)+2} (\mu_{z+1}-1).
\end{align*}
    We verify that these two expressions are equivalent in \nextmathematica.
\end{proof}
\begin{lemma}
    \label{lem:rowcolmid}
    Let $k \ge 1$ and $i = 2^k-1$.
    The sum of the $i$th row of $\lambda$ is equal to the sum of the $i$th column of $\lambda$.
\end{lemma}
\begin{proof}
By \cref{lem:row_sums,lem:colsums_above_left}, we have that the $(2^k - 1)$th row sum and column sum are both $\frac{\mu_k - 1}{2}$.\qedhere

\end{proof}

\begin{lemma}
    \label{lem:rowcolhighrest}
    Suppose $2^k \le j < 2^{k+1}-1$.
    Then, the sum of the entries in the $j$th column of $\lk$ is equal to the sum of the entries in the $j$th row of $\lk$.
\end{lemma}
\begin{proof}
First, suppose that $j = \rev(2^\ell - 1)$ for some $0\leq \ell<k$.
    By \cref{lem:rowcolhighpow2_expression}, we have that the column sum is 
\begin{align*}
    \frac{\beta_{\ell+1}^2}{2(\mu_{\ell+1}-1)}.
\end{align*}

By \Cref{lem:row_sums}, this is also the associated row sum.

Now, suppose $j = \rev(r)$ where $0\leq r <2^k -1$ and $r+1$ is not a power of $2$. Then, by \cref{lem:sum_above_diagonal_case_2,lem:case2_below_diagonal_sum}, we have that the $j$th column sum is
\begin{align*}
\frac{1}{2(\mu_{z(r) + 1} - 1)}
    &\bigg((1+\sqrt{2})^{2z(r)-\nu(r+1)-2p(r)+1}\beta_{\nu(r+1)+2} + (1+\sqrt{2})^{2(z(r) - p(r) - \nu(r+1) + 1)}\beta^2_{\nu(r+1)+1}\\
        &\qquad\qquad + (2 - \beta_{\nu(r+1)+2})(1+\sqrt{2})^{2z(r) - 2p(r) - \nu(r+1) + 1}\bigg).
\end{align*}
On the other hand, noting that $p(\rev(r)) = k - p(r) - \nu(r+1) + 2$, we have that the row sum is given by
\begin{align*}
    \frac{(1+\sqrt{2})^{2(z(r) - p(r) - \nu(r+1) + 2) - 1}}{\mu_{z(r)+1} - 1}\left(\frac{\beta^2_{\nu(r+1)+1}}{2(1+\sqrt{2})} + (1+\sqrt{2})^{\nu(r+1) - 2}\right).
\end{align*}
These expressions are equal as is verified in \nextmathematica.
\end{proof}

\begin{lemma}
    \label{lem:rowcollast}
    The sum of the $2^{k+1}-1$ row of $\lambda$ is one less than the sum of the $2^{k+1}-1$ column of $\lambda$.
\end{lemma}
\begin{proof}

The sum of the entries in the $(2^{k+1} -1)$th row is zero by definition. We now deal with the sum of the entries in the $j = (2^{k+1} - 1)$th column.
We have that $S_{2^{k+1}-1}^- = \set{2^{k+1} - 2^\ell - 1:\, \ell\in[0,k]}$.

First, suppose $i = 2^{k+1} - 2^\ell - 1$ for some $\ell = 0,\dots, k - 1$, then $\lk_i$ is defined according to Case 5.
Summing up all entries of this form, we get
\begin{align*}
    \sum_{\ell = 0}^{k-1}\left(\frac{1}{\sqrt{\mu_{\ell - 1}}} - \frac{1}{\sqrt{\mu_{\ell + 1} - 1}}\right) = \frac{1}{\sqrt{\mu_0 - 1}} - \frac{1}{\sqrt{\mu_k - 1}} = 1 - \frac{1}{\sqrt{\mu_k - 1}}.
\end{align*}

The final entry in the $(2^{k+1}-1)$th column is where $i = 2^k - 1$. Then $\lambda_{i,j} = \frac{1}{\sqrt{\mu_k - 1}}$ and so the column sum is one.
\end{proof}     \section{Proof of \Cref{thm:Mvalue}} \label{sec:lambdaM}
In this section, we will show that $\lk$ satisfies the the second main condition of \Cref{thm:straightforward_from_rank_one}.
\begin{theorem}
    \label{thm:Mvalue}
    For all $k\geq 1$, it holds that
    \[
        M(\lk,\mathfrak{h}^{(k)}) = \phi^{(k)}(\phi^{(k)})^{\intercal},
    \]
    where $\phi^{(k)}$ is defined in \Cref{eq:phiDef}.
\end{theorem}

For the remainder of the section, we fix $k\geq 1$ and let $h = \hk$, $\lambda = \lk$, $M = M(\lambda, h)$, and $\phi = \phi^{(k)}$.

\begin{proof}[Proof of \Cref{thm:Mvalue}]
    We perform the computation of $M$ entry by entry.
    The nature of the definition of $\lambda$ requires us to break this computation down into a number of distinct cases. We give a summary of the possible cases in \Cref{table:proof_references_M_entries}.

    We show that $M_{i,i} = \frac{1}{2}\phi_i^2$ for all $i\in[0,2^{k+1}-1]$ in \Cref{lem:smallDiag,lem:largeDiag,lem:medDiag,lem:M66}.

    We then need to consider the off-diagonal entries of $M$.
    Note that $M$ is symmetric, so we only need to consider $M_{i,j}$ where $i > j$.

    We will break the remaining entries of $M$ into cases, mirroring the cases in the definition of $\lambda$. 
    To reiterate, the cases we consider for an index $i$ are 
    \begin{description}
        \item[Case 1] $i < 2^k-1$ and $i+1$ is not a power of 2.
        \item[Case 2] $i > 2^k-1$ and $2^{k+1} - i - 1$ is not a power of 2.
        \item[Case 3] $i = 2^{\ell}-1$ for some $\ell < k$.
        \item[Case 4] $i = 2^{k}-1$.
        \item[Case 5] $i = 2^{k+1}-2^{\ell}-1$ for some $\ell < k$.
    \end{description}

    We summarize the possible cases and where they are proved in \Cref{table:proof_references_M_entries}.\qedhere
    
    \begin{table}
    \centering
     \begin{tabular}{ c|cccccc}
      & $i$ in Case 1 & $i$ in Case 2 & $i$ in Case 3 & $i$ in Case 4 & $i$ in Case 5 & $i=2^{k+1} - 1$\\
     \hline
     $j$ in Case 1 & \Cref{lem:M_11} & \Cref{lem:supportSimp} & \cref{lem:M_13} & 
     \cref{lem:M_14} & \Cref{lem:supportSimp} & \Cref{lem:supportSimp}\\
     $j$ in Case 2 & $\cdot$ & \Cref{lem:M_22} & \Cref{lem:supportSimp} & \cref{lem:M_24} & \cref{lem:M_25} & \cref{lem:M_6}\\
     $j$ in Case 3 &  $\cdot$ & $\cdot$ & \cref{lem:M33_M34} & \cref{lem:M33_M34} & \Cref{lem:supportSimp} & \Cref{lem:supportSimp}\\
     $j$ in Case 4 & $\cdot$  & $\cdot$ & $\cdot$ & \cref{lem:medDiag} & \cref{lem:M_45} & \cref{lem:M_6}\\
     $j$ in Case 5 & $\cdot$ & $\cdot$ & $\cdot$ & $\cdot$ & \cref{lem:M_55} & \cref{lem:M_6}\\
     $j = 2^{k+1} - 1$ & $\cdot$ & $\cdot$ & $\cdot$ & $\cdot$ & $\cdot$ & \cref{lem:M66} 
     \end{tabular}
     \caption{Pointers to the proofs of \cref{thm:Mvalue} in the separate cases.}
     \label{table:proof_references_M_entries}
    \end{table}

\end{proof}

We will make extensive use \cref{thm:rowcolconstraint} as well as the computed expressions for the partial row 
 and column sums from \cref{subsub:partial_rows} and the computational descriptions of $S^+_j$ and $S^-_j$ in \cref{subsub:col_support_lemmas}.

\subsection{Diagonal Entries of $M$}
In this subsection, we will consider the various diagonal entries of $M$.
We divide these into four cases: $M_{i,i}$ where $0\leq i < 2^k-1$, where $i = 2^k-1$, where $2^k-1<i<2^{k+1}-1$, and where $i=2^{k+1}-1$.

First, we present a lemma concerning the entries of $M$.
\begin{lemma}
Let $0\leq r\leq 2^{k+1}-1$. Then,
\begin{align*}
     M_{r,r} = \frac{1}{2}\sum_{\ell=0}^{2^{k+1}-1} \left(\lambda_{r,\ell} + \lambda_{\ell, r}\right) -  \sum_{\ell \in S_{r}^+} h_{r}\lambda_{\ell, r}.
\end{align*}
If additionally, $0 < i < 2^{k+1}-1$, then
    \[
            M_{r,r} = \sum_{\ell=0}^{2^{k+1}-1} \lambda_{r,\ell} - h_{r} \sum_{\ell \in S_{r}^+} \lambda_{\ell, r}.
    \]
\end{lemma}
\begin{proof}
We expand the entries of $M$ defined in \Cref{eq:dual-form} and note that $\lambda$ is zero on its $\star$th row and column:
\begin{align*}
    M_{r,r} &= \sum_{\substack{i,j=0\\i\neq j}}^{2^{k+1}-1} \lambda_{i,j} A_{i,j}(h)_{r,r} + \frac{1}{2}\sum_{\substack{i,j=0\\i\neq j}}^{2^{k+1}-1} \lambda_{i,j} (C_{i,j})_{r,r}.
\end{align*}
From this, it can be seen that 
\[
        M_{r,r} = \sum_{\substack{i,j=0\\i\neq j}}^{2^{k+1}-1} \lambda_{i,j} A_{i,j}(h)_{r,r} + \frac{1}{2}\sum_{\substack{i,j=0\\i\neq j}}^{2^{k+1}-1} \lambda_{i,j} (C_{i,j})_{r,r}.
\]
    If additionally, $0 < i < 2^{k+1}-1$, then by \Cref{thm:rowcolconstraint}, we have that 
    \[
        \sum_{\ell=0}^{2^{k+1}-1} \lambda_{r,\ell} = \sum_{\ell=0}^{2^{k+1}-1} \lambda_{\ell, r},
    \]
    from which the result follows.
\end{proof}
\begin{lemma}
\label{lem:smallDiag}
    If $i < 2^k-1$, then $M_{i,i} = \frac{1}{2}\phi_i^2=0$.
\end{lemma}
\begin{proof}
    When $i = 0$, we have that the $i$th row sum is $2$, the $i$th column sum is $1$, and $h_i= 3/2$. Thus,
    \begin{align*}
        M_{0,0} &= \frac{1}{2}\left(1+2\right) - \tfrac{3}{2}\cdot 1 = 0.
    \end{align*}
    
    Now, suppose $i>0$ and $2^z < i+1 < 2^{z+1}$ for some $z\in[1,k-1]$.
    By \Cref{lem:row_sums},
    \[
        \sum_{\ell=0}^{2^{k+1}-1} \lambda_{i,\ell}
        = 
        \left(\frac{\mu_{z+1}-1}{(1+\sqrt{2})^{2(z-p(i))+3}} \right)\left(\frac{\beta_{\nu(i+1)+1}^2}{2(1+\sqrt{2})} + (1+\sqrt{2})^{\nu(i+1)-2}\right).
    \]
    By \Cref{lem:SplusSum},
    \[
        h_{i} \sum_{j \in S_{i}^+} \lambda_{j, i} =  \beta_{\nu(i+1)} \frac{1}{2} \left(1+\sqrt{2}\right)^{2 (p(i)-z-2)}\beta_{\nu(i+1)+2} (\mu_{z+1}-1) .
    \]
    Subtracting these two expressions shows that $M_{i,i}=0$ (see \nextmathematica). 

    On the other hand, if $i = 2^{z}-1$ for some $z\in[1,k-1]$, then by \cref{lem:row_sums}
    \[
        \sum_{\ell=0}^{2^{k+1}-1} \lambda_{i,\ell} =
        \frac{1}{2}(\mu_{z}-1)-1+\frac{\mu_{z+1}-1}{2(1+\sqrt{2})^z},
    \]
    and
    \[
        h_i \sum_{j \in S_{i}^+} \lambda_{j, i} =  \alpha_{z}\lambda_{2^{z+1}-1,2^z-1}=\alpha_{z}\left(\frac{\mu_{z+1}-1}{2(1+\sqrt{2})^z}-1\right).
    \]
    We deduce that
    \[
        M_{i,i}=\frac{1}{2}(\mu_{z}-1)-1+\frac{\mu_{z+1}-1}{2(1+\sqrt{2})^z} - \alpha_{z}\left(\frac{\mu_{z+1}-1}{2(1+\sqrt{2})^z}-1\right).
    \]
    By \cref{lem:sqrtmu}, we have that
    \begin{gather*}
         \mu_{z} -1= 2\frac{(\alpha_z-1)^2}{\beta_{z+1}-\alpha_z},\\
         \mu_{z+1}-1= 2\frac{(\beta_{z+1}-1)^2}{\beta_{z+1}-\alpha_z}.
    \end{gather*}
    Substituting these expressions for $\mu_{z}$ and $\mu_{z+1}$ into our expression for $M_{i,i}$ shows that it is equal to zero (see \nextmathematica).\qedhere
    
\end{proof}

\begin{lemma}
\label{lem:medDiag}
It holds that
    \[M_{2^k-1,2^k-1} = \frac{1}{2}\phi_{2^k-1}^2=\mu_k-1.\]
\end{lemma}
\begin{proof}
By \Cref{lem:row_sums},
\[
    \sum_{\ell=0}^{2^{k+1}-1}  \lambda_{\ell, 2^k-1} = \frac{\mu_k-1}{2}.
\]
We also have that
\[
    h_{2^k-1} \sum_{\ell \in S_{2^k-1}^+} \lambda_{\ell, 2^k-1} = 0,
\]
since $S_{2^k-1}^+ = \varnothing$.
So, 
\[
    M_{2^k-1,2^k-1} = \sum_{\ell=0}^{2^{k+1}-1} \lambda_{2^k-1,\ell} - h_{2^{k}-1} \sum_{\ell \in S_{2^{k}-1}^+} \lambda_{\ell, 2^{k}-1} = \frac{1}{2}(\mu_k-1).\qedhere
\]
\end{proof}
\begin{lemma}
\label{lem:largeDiag}
    If $2^{k+1}-1 > i > 2^k-1$, then
    \[
        M_{i,i} = \frac{1}{2}\phi_{i}^2.
    \]
\end{lemma}
\begin{proof}
Let $i= \rev(r)$ for some $r\in[0,2^{k}-1)$.
We divide into cases: either $r+1$ is a power of two or it is not.

If $r=2^a-1$ for some $a\in[0,k-1]$, then by \Cref{lem:row_sums},
\[
    \sum_{\ell=0}^{2^{k+1}-1} \lambda_{\rev(2^{a}-1),\ell} = \frac{\beta^2_{a+1}}{2(\mu_{a+1}-1)}.
\]
We also have that $S^+_i = \varnothing$, so that 
        $M_{i,i} = \frac{\beta^2_{a+1}}{2(\mu_{a+1}-1)}$.
On the other hand, we also have that
\[
    \frac{1}{2}\phi_i^2 = \frac{\beta_{a+1}^2}{2(\mu_{a+1}-1)}.
\]

Now suppose that $i = \rev(r)$ for some $0 \leq r < 2^k-1$ where $r+1$ not a power of 2. Then
\begin{align*}
    \frac{\phi_{\rev(r)}^2}{2} &= \frac{\beta_{\nu(r+1)}^2}{2(\mu_{z(r) + 1} - 1)}.
\end{align*}

By \Cref{lem:row_sums}, and the identity the identity $p(\rev(r)) = k - p(r) - \nu(r+1) + 2$ (see \cref{lem:rev_identities}), we have 
\[
    \sum_{\ell=0}^{2^{k+1}-1}  \lambda_{\rev(r),\ell} = 
    \frac{(1+\sqrt{2})^{2(z(r) - p(r) - \nu(r+1) )+3}}{\mu_{z(r) + 1} - 1}\left(\frac{\beta^2_{\nu(r+1)+1}}{2(1+\sqrt{2})} + (1+\sqrt{2})^{\nu(r+1) - 2}\right).
\]
By \Cref{lem:sum_above_diagonal_case_2},
\[
    \sum_{\ell \in S_{\rev(r)}^-}  \lambda_{\ell, \rev(r)} = 
    \frac{1}{2(\mu_{z(r) + 1} - 1)}
    \left((1+\sqrt{2})^{2z(r)-\nu(r+1)-2p(r)+1}\beta_{\nu(r+1)+2} + \beta_{\nu(r+1)}\right).
\]

Combining these identities gives
\begin{align*}
    M_{\rev(r), \rev(r)} &=  \sum_{\ell=0}^{2^{k+1}-1}\lambda_{\rev(r),\ell} - \beta_{\nu(r+1)}\sum_{\ell\in S^+_{\rev(r)}}\lambda_{\ell,\rev(r)}\\
    &=  (1-\beta_{\nu(r+1)})\sum_{\ell=0}^{2^{k+1}-1}\lambda_{\rev(r),\ell} + \beta_{\nu(r+1)}\sum_{\ell\in S^-_{\rev(r)}}\lambda_{\ell,\rev(r)}\\
    &= \frac{\beta_{\nu(r+1)}^2}{2(\mu_{z(r)+1}-1)}.
\end{align*}
See \nextmathematica.

\end{proof}

\begin{lemma}
\label{lem:M66}
It holds that
\begin{align*}
    M_{2^{k+1}-1,2^{k+1}-1} = \frac{\phi_{2^{k+1}-1}^2}{2}.
\end{align*}
\end{lemma}
\begin{proof}
In this case,
\begin{align*}
    M_{2^{k+1}-1,2^{k+1}-1} &= \frac{1}{2}\left(\sum_{i<2^{k+1}-1} \lambda_{i,2^{k+1}-1}\right)= \frac{\phi_{2^{k+1}-1}^2}{2}.
\end{align*}
Here, we use the fact that the $(2^{k+1}-1)$th column sum is equal to one, the $(2^{k+1}-1)$th row is zero, and that $\phi_{2^{k+1}-1} = 1$.
\end{proof}

\subsection{Off-Diagonal Entries of $M$}

The following lemma gives a description for $M_{i,j}$ where $i\neq j$ and will be used repeatedly throughout this subsection.

\begin{lemma}
    \label{lem:M_entries}
Suppose $0\leq r<\ell \leq 2^{k+1} - 1$. Then
\begin{align*}
    M_{r,\ell} = \frac{1}{2}\left(
    \sum_{i=0}^r h_r\lambda_{i,\ell}-\sum_{i= \ell + 1}^{2^{k+1}-1} h_\ell \lambda_{i,r}-\lambda_{r,\ell}-\lambda_{\ell,r}\right).
\end{align*}
\end{lemma}
\begin{proof}
We expand
\begin{align*}
    M_{r,\ell} &= \sum_{\substack{i,j=0\\i\neq j}}^{2^{k+1}-1} \lambda_{i,j} A_{i,j}(h)_{r,\ell} + \frac{1}{2}\sum_{\substack{i,j=0\\i\neq j}}^{2^{k+1}-1} \lambda_{i,j} (C_{i,j})_{r,\ell}.
\end{align*}
We will use the definition of $A$ and $C$ to pull out the nonzero entries of each sum. The first sum becomes
\begin{align*}
    &\sum_{\substack{i=0\\i\neq r}}^{2^{k+1}-1} \lambda_{i,r} A_{i,r}(h)_{r,\ell}
    + \sum_{\substack{i=0\\i\neq \ell}}^{2^{k+1}-1} \lambda_{i,\ell} A_{i,\ell}(h)_{r,\ell}\\
    &\qquad= \sum_{i= \ell + 1}^{2^{k+1}-1} \lambda_{i,r} A_{i,r}(h)_{r,\ell}
    + \sum_{i=0}^r \lambda_{i,\ell} A_{i,\ell}(h)_{r,\ell}\\
    &\qquad= -\frac{1}{2}\sum_{i= \ell + 1}^{2^{k+1}-1} h_\ell \lambda_{i,r}
    + \frac{1}{2}\sum_{i=0}^r h_r\lambda_{i,\ell}.
\end{align*}
The second sum becomes
\begin{align*}
    \frac{1}{2}\left(\lambda_{r,\ell}(C_{r,\ell})_{r,\ell} + \lambda_{\ell,r}(C_{\ell,r})_{r,\ell}\right)&=-\frac{1}{2}\left(\lambda_{r,\ell}+\lambda_{\ell,r}\right).\qedhere
\end{align*}
\end{proof}

We consider a useful lemma, which can be shown by just considering the support of $\lambda$:
\begin{lemma}
    \label{lem:supportSimp}
    Fix $i < j< 2^{k+1}-1$.
    If $z(i) > z(j) + 1$ then $M_{i,j} = 0$.
    
    Also, if $z(i) = z(j) + 1$ and $i+1$ is not a power of 2, then $M_{i,j} = 0$.
\end{lemma}
\begin{proof}
    In all of the following calculations, we will consider the expression
    \[
        M_{i,j} = -\frac{1}{2}\left(\lambda_{i,j} + \lambda_{j, i}\right) -  \frac{1}{2}h_{i} \sum_{\ell = i+1}^{2^k-1} \lambda_{\ell,j} + \frac{1}{2}h_{j}  \sum_{\ell = *}^{j} \lambda_{\ell,i}.
    \]

    For the first statement, we will argue that if $z(j) < z(i) - 1$, or it is the case that $z(j) = z(i) - 1$ and $i+1$ is not a power of 2, then every term in this sum is 0.
    It is not hard to see by considering the binary expansion of an integer $\ell$ that $\ell + 2^{\nu(\ell+1)} \le 2^{z(\ell)} - 1$, and that $\ell - 2^{\nu(\ell+1) - 1} \ge 2^{z(\ell)} - 1$, unless $\ell = 2^{z(\ell)}-1$.
    Our characterization of the support of $\lambda$ then implies that for all $\ell \le j \le 2^{z(j)+1}-1$, $\lambda_{\ell, i} = 0$ and that for all $\ell \ge i > 2^{z(i)}-1$, and this clearly implies that $M_{i,j} = 0$.

\end{proof}

We begin by taking care of the easier cases.

\begin{lemma}
\label{lem:M_13}
    Suppose $0 < r < 2^{k} - 1$ and $r+1$ is not a power of $2$.
    Also let $0\leq \ell < k$. Then,
    \begin{align*}
        M_{r, 2^\ell - 1} = 0 = \frac{\phi_r \phi_{2^\ell - 1}}{2}.
    \end{align*}
\end{lemma}
\begin{proof}
We split the proof into cases depending on if $2^{\ell}-1<r$ or if $r<2^{\ell}-1$.

First suppose $2^\ell - 1 < r$. Then by \cref{lem:M_entries},
\begin{align*}
    M_{2^\ell - 1, r} &= \frac{1}{2}\left(
        \sum_{i=0}^{2^\ell - 1} h_{2^\ell - 1}\lambda_{i,r}-\sum_{i= r + 1}^{2^{k+1}-1} h_r \lambda_{i,2^\ell - 1}-\lambda_{2^\ell - 1,r}\right).
\end{align*}
If $r + 1> 2^{\ell+1}$, \Cref{lem:supportSimp} implies that this is 0. Thus, we may assume $r+1 < 2^{\ell + 1}$ in the remainder of this case. Then,
\begin{align*}
     M_{r, 2^\ell - 1}  &= \frac{1}{2}\left(
        (\alpha_\ell-1)\lambda_{2^\ell - 1,r}-\beta_{\nu(r+1)}\lambda_{2^{\ell+1}-1,2^\ell - 1}\right)\\
        &= \frac{1}{2}\left(
            (\alpha_\ell-1)\frac{\mu_{\ell+1} - 1}{2(1+\sqrt{2})^{2\ell}}\beta_{\nu(r +1)}-\beta_{\nu(r+1)}\left(\frac{\mu_{\ell+1}-1}{2(1+\sqrt{2})^\ell}-1\right)\right)\\
        &= \frac{\beta_{\nu(r+1)}}{2}\left(
            (\alpha_\ell-1)\frac{\mu_{\ell+1} - 1}{2(1+\sqrt{2})^{2\ell}}-\frac{\mu_{\ell+1}-1}{2(1+\sqrt{2})^\ell}+1\right)\\
        &= 0.
\end{align*}
The last line uses the expression for $\mu_{\ell+1}$ given in \cref{lem:sqrtmu} (see \nextmathematica).

Now, suppose $r<2^\ell - 1$. Then,
\begin{align*}
    M_{r,2^\ell - 1} &= \frac{1}{2}\left(
    \sum_{i=0}^r h_r\lambda_{i,2^\ell - 1}-\lambda_{r,2^\ell - 1}-\lambda_{2^\ell - 1,r}\right).
\end{align*}
If $r + 1 < 2^{\ell - 1}$, \Cref{lem:supportSimp} implies that this is 0. Thus, we may suppose $r+1 > 2^{\ell - 1}$ in the remainder of this case.
Let $2^{\tau-1} < 2^\ell - 1 - r \leq 2^{\tau}$ where $\tau\in[1, \ell - 1]$. Then,
\begin{align*}
    M_{r,2^\ell - 1} &= \frac{1}{2}\left(
        \beta_{\nu(r+1)}\sum_{s=\tau}^{\ell-1} \lambda_{2^\ell - 2^s - 1,2^\ell - 1}-\lambda_{r,2^\ell - 1}-\lambda_{2^\ell - 1,r}\right). 
\end{align*}
The inner summation is given by
\begin{align*}
    \sum_{s=\tau}^{\ell-1} \lambda_{2^\ell - 2^s - 1,2^\ell - 1} &= \sum_{s=\tau}^{\ell-2} \lambda_{2^\ell - 2^s - 1,2^\ell - 1} + \lambda_{2^{\ell-1} - 1,2^\ell - 1}\\
    &= \sum_{s=\tau}^{\ell-2}\frac{\mu_{\ell}-1}{(1+\sqrt{2})^{2s+1}} + \frac{\mu_\ell - 1}{2(1+\sqrt{2})^{2(\ell-1)}}\\
    &= \frac{\mu_\ell - 1}{2}(1+\sqrt{2})^{-2\tau}.
\end{align*}
Here, we have used the fact that $z(2^\ell-2^s-1)=\ell-1$ and $p(2^\ell-2^s-1)=\ell-s$.
We also compute
\begin{align*}
    \lambda_{r,2^\ell-1} &= \begin{cases}
        \frac{\mu_{\ell} - 1}{(1+\sqrt{2})^{2\tau + 1}} & \text{if } 2^{\ell} - 1 - r = 2^\tau\\
        0 & \text{else},
    \end{cases}\\
    \lambda_{2^\ell-1,r} &= \begin{cases}
        \frac{\mu_\ell - 1}{2(1+\sqrt{2})^{2(\tau+1)}} \beta_{\tau+2} & \text{if } 2^{\ell} - 1 - r = 2^\tau\\
        \frac{\mu_\ell - 1}{2(1+\sqrt{2})^{2\tau}} \beta_{\nu(r+1)} & \text{else}.
    \end{cases}
\end{align*}
When $2^\ell -1 - r = 2^\tau$, we have that
\begin{align*}
    M_{r, 2^\ell - 1} &= \frac{\mu_\ell - 1}{2(1+\sqrt{2})^{2\tau}}\left(\beta_\tau - \frac{2}{1+\sqrt{2}} -  \frac{\beta_{\tau+2}}{(1+\sqrt{2})^2} \right) = 0.
\end{align*}
See \nextmathematica.

When $2^\ell - 1 - r < 2^\tau$, we have that
\begin{align*}
    M_{r, 2^\ell - 1} &=
    \beta_{\nu(r+1)}\frac{\mu_\ell - 1}{2}(1+\sqrt{2})^{-2\tau}  - \frac{\mu_\ell - 1}{2(1+\sqrt{2})^{2\tau}} \beta_{\nu(r+1)} = 0.\qedhere
\end{align*}
\end{proof}

\begin{lemma}
\label{lem:M_14}
Suppose $0\leq r < 2^k -1$ and $r+1$ is not a power of $2$. Then,
\begin{align*}
    M_{r,2^k - 1} = 0 = \frac{\phi_{r} \phi_{2^{k} - 1}}{2}.
\end{align*}
\end{lemma}
\begin{proof}
The second identity holds as $\phi_r = 0$. We turn to the first identity. By \cref{lem:M_entries}, we have
\begin{align*}
    M_{r,2^k - 1} &= \frac{1}{2}\left(
        \beta_{\nu(r+1)}\sum_{i=0}^r \lambda_{i,2^k - 1} -\lambda_{2^k - 1,r}-\lambda_{r,2^k - 1}\right).
\end{align*}
Note that if $r< 2^{k-1}-1$, then every term in this expression is zero. Thus we will assume $r > 2^{k-1}-1$.
Now, we will write $2^k - (r+1) = \tau$. Note that $1\leq \tau \leq 2^{k-1} - 1$.

There are two remaining cases. First, suppose $\tau = 2^\ell$ for some $\ell\in[0,k-2]$.
In this case, 
$\beta_{\nu(r+1)}= \beta_{\ell}$ and
$r = 2^k - 2^\ell -1$ so that the inner summation is give by
\begin{align*}
    \sum_{s = \ell}^{k-1} \lambda_{2^k - 2^s - 1,2^k -1} &= \sum_{s=\ell}^{k-2} \lambda_{2^k - 2^s - 1,2^k -1} + \lambda_{2^{k-1} - 1, 2^k - 1}\\
    &= \sum_{s=\ell}^{k-2} \frac{\mu_{k}-1}{(1+\sqrt{2})^{2s+1}} + \frac{\mu_k - 1}{2(1+\sqrt{2})^{2(k-1)}}\\
    &=\frac{\mu_k - 1}{2(1+\sqrt{2})^{2\ell}}. 
\end{align*}
Additionally,
\begin{align*}
    \lambda_{2^k - 1,2^k - 2^\ell - 1} + \lambda_{2^k - 2^\ell - 1,2^k - 1} &= \frac{\mu_k - 1}{2}\left(\frac{\beta_{\ell+2}}{(1+\sqrt{2})^{2(\ell+1)}} + \frac{2}{(1+\sqrt{2})^{2\ell+1}}\right).\end{align*}
Combining these identities gives $M_{2^k - 2^\ell - 1,2^k - 1} = 0$ (see \nextmathematica).

The final case is when $\tau$ is not a power of $2$. Let $\ell$ so that $2^\ell < \tau < 2^{\ell+1}$. Note that we must have $\ell\in[0,k-2]$. The inner summation is then given by
\begin{align*}
    \sum_{s=\ell+1}^{k-1} \lambda_{2^k - 2^s - 1,2^k - 1}
    &=  \sum_{s=\ell+1}^{k-2} \lambda_{2^k - 2^s - 1,2^k -1} + \lambda_{2^{k-1} - 1, 2^k - 1}\\
    &= \sum_{s=\ell+1}^{k-2} \frac{\mu_{k}-1}{(1+\sqrt{2})^{2s+1}} + \frac{\mu_k - 1}{2(1+\sqrt{2})^{2(k-1)}}\\
    &= \frac{\mu_k - 1}{2(1+\sqrt{2})^{2(\ell+1)}}. 
\end{align*}
Next,
\begin{align*}
    \lambda_{2^k - 1,r} + \lambda_{r,2^k - 1} &= \lambda_{2^k - 1,r}
= \frac{(\mu_k - 1)}{2(1+\sqrt{2})^{2(\ell+1)}}\beta_{\nu(r+1)}.
\end{align*}
Combining these identities gives $M_{r,2^k - 1} = 0$.
\end{proof}

\begin{lemma}
\label{lem:M_24}
    Suppose $0\leq r< 2^k-1$ and $r+1$ is not a power of $2$. Then,
    \begin{align*}
        M_{2^k - 1,\rev(r)} = \frac{\phi_{\rev(r)} \phi_{2^k - 1}}{2}.
    \end{align*}
\end{lemma}
\begin{proof}
    Let $\tau$ so that $2^\tau < r+1 < 2^{\tau + 1}$. Note, we must have $\tau\in[0,k-1]$. 
    By \cref{lem:M_entries}, we have
    \begin{align*}
        M_{2^k - 1,\rev(r)} &= \frac{1}{2}\left(
            \sum_{i=0}^{2^k -1} h_{2^k -1}\lambda_{i,\rev(r)}-\lambda_{2^k -1,\rev(r)}\right)\\
        &= \frac{\mu_k - 1}{2}\lambda_{2^k - 1,\rev(r)}\\
        &= \frac{\sqrt{\mu_k - 1}}{2}(w_k)_{2^k-r-1}.
    \end{align*}
    On the other hand, $\phi_{2^k - 1} = \sqrt{\mu_k - 1}$ and $\phi_{\rev(r)} = (w_k)_{2^k-r-1}$.
\end{proof}

\begin{lemma}
\label{lem:M_25}
Suppose $0<r<2^k - 1$ and $r+1$ is not a power of $2$. Also let $0\leq \ell < k$. Then,
\begin{align*}
    M_{\rev(r),\rev(2^\ell - 1)} = \frac{\phi_{\rev(r)}\phi_{\rev(2^\ell-1)}}{2}.
\end{align*}
\end{lemma}
\begin{proof}
We split the proof into three cases: where $\rev(r) > \rev(2^{\ell} - 1)$,
where $\rev(2^{\ell}+2^{\ell-1}-1) \leq \rev(r) < \rev(2^\ell - 1)$
and 
where $\rev(r) < \rev(2^{\ell}+2^{\ell-1}-1)$.

First, suppose $\rev(r) > \rev(2^{\ell} - 1)$. Then,
\begin{align*}
    M_{\rev(2^\ell -1),\rev(r)} &= \frac{1}{2}\left(
    \alpha_\ell \sum_{i=0}^{\rev(2^\ell - 1)} \lambda_{i,\rev(r)}-\lambda_{\rev(2^\ell -1),\rev(r)}\right).
\end{align*}
We deal with the inner summation separately. Let $s\in[1,\ell-1]$ so that $2^s < r+1 < 2^{s+1}$. Then,
\begin{align*}
    \sum_{i=0}^{\rev(2^\ell -1)} \lambda_{i,\rev(r)} &= \sum_{\psi = \ell}^{k} \lambda_{\rev(2^\psi - 1), \rev(r)}
    = \frac{\beta_{\nu(r+1)}}{\sqrt{\mu_{\ell}-1}\sqrt{\mu_{s+1} - 1}}.
\end{align*}
We also compute
\begin{align*}
    \lambda_{\rev(2^\ell - 1),\rev(r)} &= \left(\frac{1}{\sqrt{\mu_\ell - 1}} - \frac{1}{\sqrt{\mu_{\ell+1}-1}}\right)\frac{\beta_{\nu(r+1)}}{\sqrt{\mu_{s+1}-1}}.
\end{align*}
Combining these two identities gives
\begin{align*}
    M_{\rev(2^\ell -1),\rev(r)} &= \left(\frac{\beta_{\nu(r+1)}}{2\sqrt{\mu_{s+1}-1}}\right)\left(\frac{\alpha_\ell-1}{\sqrt{\mu_{\ell}-1}} + \frac{1}{\sqrt{\mu_{\ell+1}-1}}\right).
\end{align*}
The second term in parentheses is equal to $\beta_{\ell+1}/\sqrt{\mu_{\ell+1}-1}$ by \cref{lem:sqrtmu}.
We compare this with
\begin{align*}
    \phi_{\rev(2^\ell - 1)} &= \frac{\beta_{\ell+1}}{\sqrt{\mu_{\ell+1}-1}} \qquad\text{and}\qquad
    \phi_{\rev(r)} = \frac{\beta_{\nu(r+1)}}{\sqrt{\mu_{s+1} - 1}}.
\end{align*}
This completes the proof assuming $\rev(r)>\rev(2^{\ell}-1)$.

Now, suppose $\rev(r)<\rev(2^{\ell}+2^{\ell-1}-1)$. Let $s\in[\ell,k-1]$ so that $2^{s} < r+1 < 2^{s+1}$. Then,
\begin{align*}
    M_{\rev(r),\rev(2^\ell -1)} &= \frac{\beta_{\nu(r+1)}}{2}
        \sum_{\psi = s+1}^k \lambda_{\rev(2^\psi - 1),\rev(2^\ell -1)}\\
        &= \frac{\beta_{\nu(r+1)}}{2\sqrt{\mu_{s+1}-1}}\frac{\beta_{\ell+1}}{\sqrt{\mu_{\ell+1} - 1}}.
\end{align*}

In the final case, suppose $\rev(2^{\ell}+2^{\ell-1}-1)\leq  \rev(r) < \rev(2^{\ell}-1)$.
\begin{align*}
    M_{\rev(r),\rev(2^\ell - 1)} &= \frac{1}{2}\left(\beta_{\nu(r+1)}
        \sum_{i=0}^{\rev(r)} \lambda_{i,\rev(2^\ell - 1)}-\lambda_{\rev(r),\rev(2^\ell - 1)}-\lambda_{\rev(2^\ell - 1),\rev(r)}\right).
\end{align*}
Let $s\in[1,\ell-1]$ so that $2^\ell + 2^{s-1} < r + 1 \leq 2^\ell + 2^s$. We deal with the inner summation next:
\begin{align*}
    \sum_{i=0}^{\rev(r)} \lambda_{i,\rev(2^\ell - 1)} &= \sum_{\psi = s}^{\ell - 1} \lambda_{\rev(2^\ell + 2^\psi - 1),\rev(2^\ell - 1)} + \sum_{\psi = \ell+1}^{k} \lambda_{\rev(2^\psi - 1),\rev(2^\ell - 1)}\\
    &= \sum_{\psi = s}^{\ell - 1} \frac{(1+\sqrt{2})^{2(\ell-\psi)-1}}{\mu_{\ell+1}-1} + \frac{\beta_{\ell+1}}{\mu_{\ell+1} - 1}\\
    &= \frac{(1+\sqrt{2})^{2(\ell-s)}-1}{2(\mu_{\ell+1}-1)} + \frac{\beta_{\ell+1}}{\mu_{\ell+1} - 1}.
\end{align*}
Next, we compute
\begin{align*}
    \lambda_{\rev(r),\rev(2^\ell - 1)} &= \begin{cases}
    \frac{(1+\sqrt{2})^{2(\ell-s)- 1}}{\mu_{\ell+1} - 1} & \text{if } r= 2^\ell + 2^s - 1\\
    0 & \text{else}
    \end{cases}
\end{align*}
and
\begin{align*}
    \lambda_{\rev(2^\ell - 1),\rev(r)} &= \begin{cases}
    \frac{1}{2(\mu_{\ell+1} - 1)}\left((1+\sqrt{2})^{2(\ell-s-1)}\beta_{s+2} - \beta_{s}\right) & \text{if } r= 2^\ell + 2^s - 1\\
    \frac{\beta_{\nu(r+1)}}{2(\mu_{\ell+1} - 1)}\left((1+\sqrt{2})^{2(\ell-s)} - 1\right) & \text{else}.
    \end{cases}
\end{align*}
Combining these identities and using the fact that $\beta_s = \beta_{\nu(r+1)}$ in the case where $r= 2^\ell + 2^s - 1$ gives
\begin{align*}
    M_{\rev(r),\rev(2^\ell - 1)} &= \frac{\beta_{\ell+1}\beta_{\nu(r+1)}}{2(\mu_{\ell+1} - 1)}.\qedhere
\end{align*}
\end{proof}

\begin{lemma}
\label{lem:useful_identity_for_case33_case34}
Suppose $0\leq \ell\leq k-1$. Then,
\begin{align*}
    \lambda_{2^{\ell}-1,2^{\ell + 1} - 1}h_{2^{\ell}-1} -\left(\lambda_{2^{\ell} -1,2^{\ell+1} - 1}+\lambda_{2^{\ell+1} -1,2^{\ell} - 1}\right) = 0.
\end{align*}
\end{lemma}
\begin{proof}
We expand
\begin{align*}
    \text{LHS}
    &=\lambda_{2^{\ell}-1,2^{\ell+1} - 1}(\alpha_{\ell} - 1) - \lambda_{2^{\ell+1} -1,2^{\ell} - 1}\\
    & = \frac{\mu_{\ell+1} - 1}{2(1+\sqrt{2})^{2\ell}}(\alpha_{\ell} - 1) - \frac{(\mu_{\ell+1} - 1)}{2(1+\sqrt{2})^{\ell}} + 1\\
    & = \frac{1}{2(1+\sqrt{2})^{2\ell}}\left((\mu_{\ell+1} - 1)(\alpha_{\ell}-\beta_{\ell+1}) + 2(1+\sqrt{2})^{2\ell}\right).\end{align*}
Combining this expression for the left-hand side with our expression for $\mu_{\ell+1}$ in \cref{lem:sqrtmu} shows that it is equal to zero (see \nextmathematica).
\end{proof}

\begin{lemma}
\label{lem:M33_M34}
Suppose $0\leq\ell< \tau\leq k$. Then,
\begin{align*}
    M_{2^\ell - 1,2^\tau -1} = 0=\frac{\phi_{2^\ell - 1}\phi_{2^\tau - 1}}{2}.
\end{align*}
\end{lemma}
\begin{proof}
As $\ell < k$, we have that $\phi_{2^\ell - 1} = 0$. Thus, it suffices to show that $M_{2^\ell - 1, 2^\tau -1}  = 0$.

If $\ell=\tau-1$, then this follows from \cref{lem:useful_identity_for_case33_case34}. On the other hand, if $\ell<\tau - 1$, then this follows from \cref{lem:supportSimp}.\qedhere

\end{proof}

\begin{lemma}
\label{lem:M_45}
    Suppose $0\leq\ell<k$. Then,
    \begin{align*}
        M_{2^k - 1,\rev(2^\ell -1)} = \frac{\phi_{2^k - 1}\phi_{\rev(2^\ell-1)}}{2}.
    \end{align*}
\end{lemma}
\begin{proof}
    By \cref{lem:M_entries},
    \begin{align*}
        M_{2^k - 1, \rev(2^\ell -1)}
        &= \frac{1}{2}\Bigg(
    \sum_{i=0}^{2^k-1} h_{2^k-1}\lambda_{i,\rev(2^\ell-1)}-\sum_{i= \rev(2^\ell-1)+1}^{2^{k+1}-1} h_{\rev(2^\ell-1)} \lambda_{i,2^{k}-1}\\
    &\hspace{6em}-\lambda_{2^k-1,\rev(2^\ell-1)}-\lambda_{\rev(2^\ell-1),2^k-1}\bigg).
    \end{align*}

    The only nonzero entry in the first two summations is
    \begin{align*}
    h_{2^{k} - 1}\lambda_{2^k - 1,\rev(2^\ell-1)}
        &= \frac{\mu_k}{\sqrt{\mu_k - 1}}\frac{\beta_{\ell+1}}{\sqrt{\mu_{\ell+1} - 1}}.
    \end{align*}

    The two remaining terms are
    \begin{align*}
        &\lambda_{\rev(2^\ell -1),2^k - 1}+\lambda_{2^k - 1, \rev(2^\ell -1)}= \frac{\beta_{\ell+1}}{\sqrt{\mu_k-1}\sqrt{\mu_{\ell+1}-1}}.
    \end{align*}

    Substituting both expressions back in gives
    \begin{align*}
         M_{2^k - 1, \rev(2^\ell -1)} &= \frac{1}{2}\sqrt{\mu_k-1}\frac{\beta_{\ell+1}}{\sqrt{\mu_{\ell+1}-1}} = \frac{\phi_{2^k - 1}\phi_{\rev(2^\ell - 1)}}{2}.\qedhere
    \end{align*}
\end{proof}

\begin{lemma}
\label{lem:M_55}
Suppose $0\leq r < \tau <k$. Then
\begin{align*}
    M_{2\rev(2^\tau - 1), \rev(2^r - 1)} = \frac{\phi_{\rev(2^\tau - 1)} \phi_{\rev(2^r - 1)}}{2}.
\end{align*}
\end{lemma}
\begin{proof}
We have
\begin{align*}
    &M_{\rev(2^\tau - 1), \rev(2^r - 1)}\\
    &\qquad=\frac{1}{2}\left(
        \alpha_\tau\sum_{i=0}^{\rev(2^\tau - 1)} \lambda_{i,\rev(2^r - 1)}-\lambda_{\rev(2^\tau - 1),\rev(2^r - 1)}\right)\\
    &\qquad = \frac{1}{2}\left(
        \alpha_\tau\sum_{i=0}^{\rev(2^\tau - 1)} \lambda_{i,\rev(2^r - 1)}-  \left(\frac{1}{\sqrt{\mu_\tau - 1}} - \frac{1}{\sqrt{\mu_{\tau + 1} - 1}}\right)\frac{\beta_{r+1}}{\sqrt{\mu_{r+1} - 1}}\right).
\end{align*}
We deal with the inner summation next. Note that $\rev(2^r - 1)= \sum_{i=0}^k b_a2^a$ where $b_a = 0$ if and only if $a = r$. Thus, by \cref{lem:nonzero_entries_above}, the relevant entries in the summation correspond to indices $i = \rev(2^\ell -1)$ where $\ell \in[\tau,k]$. Thus, the inner summation is given by
\begin{align*}
    \sum_{i=0}^{\rev(2^\tau - 1)} \lambda_{i,\rev(2^r - 1)} &= \sum_{\ell = \tau}^k \lambda_{\rev(2^\ell - 1),\rev(2^r - 1)}\\
    &= \frac{1}{\sqrt{\mu_{\tau} - 1}}\frac{\beta_{r+1}}{\sqrt{\mu_{r+1} - 1}}.
\end{align*}
Plugging this back in gives
\begin{align*}
    M_{\rev(2^\tau - 1), \rev(2^r - 1)}
    &= \frac{1}{2}\left(
        \alpha_\tau\frac{1}{\sqrt{\mu_{\tau} - 1}}\frac{\beta_{r+1}}{\sqrt{\mu_{r+1} - 1}} -  \left(\frac{1}{\sqrt{\mu_\tau - 1}} - \frac{1}{\sqrt{\mu_{\tau + 1} - 1}}\right)\frac{\beta_{r+1}}{\sqrt{\mu_{r+1} - 1}}\right)\\
    &= \frac{\beta_{r+1}}{2\sqrt{\mu_{r+1}-1}}\left(\frac{\alpha_\tau-1}{\sqrt{\mu_\tau - 1}} + \frac{1}{\sqrt{\mu_{\tau + 1} - 1}}\right)\\
    &= \frac{\beta_{r+1}}{2\sqrt{\mu_{r+1}-1}}\frac{\beta_{\tau+1}}{\sqrt{\mu_{\tau+1} - 1}}.
\end{align*}
Here, on the last line we have used \cref{lem:sqrtmu}. This completes the proof as $\phi_{\rev(2^r-1)} = \frac{\beta_{r+1}}{\sqrt{\mu_{r+1} - 1}}$ and $\phi_{\rev(2^\tau - 1)} = \frac{\beta_{\tau+1}}{\sqrt{\mu_{\tau+1} - 1}}$.
\end{proof}

\begin{lemma}
\label{lem:M_6}
Suppose $0\leq s \leq 2^{k+1}-2$. Then
\begin{align*}
    M_{s,2^{k+1}-1} = \frac{\phi_s}{2}=\frac{\phi_s \phi_{2^{k+1}-1}}{2}.
\end{align*}
\end{lemma}
\begin{proof}
The second equality follows from the fact that $\phi_{2^{k+1}-1} = 1$.

We turn to the first equality. Let $0\leq s\leq 2^{k+1}-2$. First, note if $0\leq s \leq 2^{k}-2$, then $\phi_s=0$ and $M_{s,2^{k+1}-1}=0$ by \cref{lem:supportSimp}.
Thus, we may assume $2^{k}-1\leq s\leq 2^{k+1}-2$ in the remainder of this proof. We will break the rest of the proof into cases depending on how $\lambda_{s}$ is defined, i.e., Cases 2, 4, and 5.

\textbf{Case 2}: Suppose $s=\rev(r)$ for some $0\leq r< 2^{k}-1$ for which $r+1$ is not a power of $2$.
By \cref{lem:M_entries},
\begin{align*}
    M_{\rev(r),2^{k+1}-1} &= \frac{1}{2}\left(
        \sum_{i=0}^{\rev(r)} h_{\rev(r)}\lambda_{i,2^{k+1}-1}-\lambda_{\rev(r),2^{k+1}-1}\right)\\
        &=\frac{h_{\rev(r)}}{2}
            \sum_{i=0}^{\rev(r)} \lambda_{i,2^{k+1}-1}.
\end{align*}
By \cref{lem:nonzero_entries_above}, the nonzero summands in the sum correspond to indices $i = \rev(2^\tau - 1)$ where $z(r) + 1 \leq\tau \leq k$. We also have that $h_s = \beta_{\nu(r + 1)}$.
Combining these identities gives
\begin{align*}
    M_{\rev(r),2^{k+1}-1} &= \frac{\beta_{\nu(r+1)}}{2}\sum_{\tau=z(r) + 1}^k \lambda_{\rev(2^\tau - 1), 2^{k+1}-1}\\
    &= \frac{\beta_{\nu(r+1)}}{2}\left(\sum_{\tau=z(r) + 1}^{k-1} \lambda_{\rev(2^\tau - 1), 2^{k+1}-1} + \lambda_{2^k - 1, 2^{k+1}-1}\right)\\
    &= \frac{\beta_{\nu(r+1)}}{2}\frac{1}{\sqrt{\mu_{z(r) + 1} - 1}}.
\end{align*}
Note that $\phi_{\rev(r)} = \frac{\beta_{\nu(r+1)}}{\sqrt{\mu_{z(r)+1} - 1}}$.

\textbf{Case 4}: Suppose $s = 2^k - 1$. By \cref{lem:M_entries},
\begin{align*}
    M_{2^k-1,2^{k+1}-1} &= \frac{1}{2}\left(
    \sum_{i=0}^{2^k-1} g_{2^k-1}\lambda_{i,2^{k+1}-1}-\lambda_{2^k-1,2^{k+1}-1}\right)\\
    &=\frac{1}{2}(\mu_k-1) \lambda_{2^k-1,2^{k+1}-1}\\
    &= \frac{\sqrt{\mu_k - 1}}{2}.
\end{align*}
On the other hand, $\phi_{2^k-1} = \sqrt{\mu_k - 1}$.

\textbf{Case 5}: Suppose $s = \rev(2^\ell-1)$ for some $0\leq \ell < k$. By \cref{lem:M_entries},
\begin{align*}
    M_{\rev(2^\ell-1),2^{k+1}-1} &= \frac{1}{2}\left(
        \sum_{i=0}^{\rev(2^\ell - 1)} h_{\rev(2^\ell - 1)}\lambda_{i,2^{k+1}-1} - \lambda_{\rev(2^\ell - 1),2^{k+1}-1}\right)\\
        &= \frac{1}{2}\left(\alpha_\ell \sum_{i=0}^{\rev(2^\ell - 1)}\lambda_{i,2^{k+1}-1} - \frac{1}{\sqrt{\mu_\ell - 1}} + \frac{1}{\sqrt{\mu_{\ell+1}-1}}\right).
\end{align*}
We handle the inner summation below. By \cref{lem:nonzero_entries_above}, the relevant indices in the summation are given by $i =\rev(2^\tau - 1)$ where $\ell \leq \tau \leq k$. Thus,
\begin{align*}
    \sum_{i=0}^{\rev(2^\ell - 1)}\lambda_{i,2^{k+1}-1} &= 
    \sum_{\tau = \ell}^k \lambda_{\rev(2^\tau - 1),2^{k+1}-1}\\
    &=\sum_{\tau = \ell}^{k-1} \lambda_{\rev(2^\tau - 1),2^{k+1}-1} + \lambda_{2^k - 1,2^{k+1}-1}\\
    &= \frac{1}{\sqrt{\mu_\ell - 1}}.
\end{align*}
Plugging back in gives
\begin{align*}
    M_{\rev(2^\ell - 1),2^{k+1}-1} &= \frac{1}{2}\left(\frac{\alpha_\ell - 1}{\sqrt{\mu_\ell - 1}} + \frac{1}{\sqrt{\mu_{\ell+1}-1}}\right)\\
    &= \frac{\beta_{\ell+1}}{2\sqrt{\mu_{\ell+1}-1}},
\end{align*}
where we have used \cref{lem:sqrtmu} in the second line.
Note that $\phi_{\rev(2^\ell - 1)} = \frac{\beta_{\ell+1}}{\sqrt{\mu_{\ell+1} - 1}}$.
\end{proof}

\subsection{Off-Diagonal Entries in Case $(1,1)$}

Our goal for this subsection will be to prove \Cref{lem:M_11}, stating that the off-diagonal entries $M_{i,j}$ are zero for all $0\leq j< i< 2^k-1$, where neither $i+1$ nor $j+1$ are powers of two (equivalently, where $p(i),p(j)\geq2$).
We will prove \cref{lem:M_11} inductively on the value of $\min(p(i),p(j))$.
We will make use of the following lemma

\begin{lemma}
\label{lem:secondDigit}
    Fix $i > j$ such that $i+1$ and $j+1$ are both not powers of 2. If $z(i) = z(j)$, so that $i = 2^z + i'$ and $j = 2^z + j'$, and $z(j') < z(i') -1$, then $M_{i,j} = 0$.
\end{lemma}
\begin{proof}
    Note that $\lambda_{i,j} = 0$, since 
    \[i - 2^{\nu(i+1)-1} = 2^{z} + i' -  2^{\nu(i'+1)-1} \ge 2^{z} + 2^{z(i')} > 2^z + j' = j.\]
   Similarly, $\lambda_{j,i} = 0$, since
    \[j + 2^{\nu(j+1)} = 2^{z} + j' +  2^{\nu(j'+1)} \le 2^{z} + 2^{z(j')+1} < 2^z + i' = i.\]
    If $\ell < 2^{z(i)}-1$, then as above, $\lambda_{\ell, i} = 0$, so let $2^{z(i)} - 1 \le \ell \le j$. We then have that $\ell = 2^{z(i)} + \ell'$, where $\ell' \le j'$. If $\ell' \neq 0$, then $\nu(\ell'+1) = \nu(\ell+1)$, and
    \[
        \ell + 2^{\nu(\ell+1)} = 2^{z(i)} + \ell' + 2^{\nu(\ell'+1)} \le 2^{z(i)} + 2^{z(\ell'+1)+1} < 2^{z(i)} + i'.
    \]
    Therefore, the only nonzero term in the sum $\sum_{\ell = *}^{j} \lambda_{\ell,i}$ is $\lambda_{2^{z(i)}-1, i}$.
    Similarly, if $\ell > i$, and $\ell = 2^{z(i)} + \ell'$, where $\ell' < 2^{z(i)}-1$, then 
    \[
        \ell - 2^{\nu(\ell+1)-1} = 2^{z(i)} + \ell' - 2^{\nu(\ell'+1)-1} \ge 2^{z(i)} + 2^{z(\ell'+1)-1} > j.
    \]
    This implies that $\sum_{\ell = i+1}^{2^k-1} \lambda_{\ell,j} = \lambda_{2^{z(i)+1}-1,j}$. We conclude that
    \[
        M_{i,j} = -\frac{1}{2}h_{i} \lambda_{2^{z+1}-1,j} + \frac{1}{2}h_{j}  \lambda_{2^z-1,i}.
    \]
    Now, note that because $i+1$ is not a power of 2,
    \[
        \lambda_{2^z-1,i} = \frac{\mu_{z+1}-1}{2(1+\sqrt{2})^{2z}} \beta_{\nu(i+1)}.
    \]
    Also note that because $j < 2^{z} + 2^{z-1}-1$, and $j$ is not a power of 2,
    \[
        \lambda_{2^{z+1}-1,j} = \frac{\mu_{z+1}-1}{2(1+\sqrt{2})^{2z}} \beta_{\nu(j+1)}.
    \]
    The result follows by noting that $h_i = \beta_{\nu(i+1)}$ and  $h_j = \beta_{\nu(j+1)}$, so that 
    \[
        M_{i,j} = -\frac{1}{2}\beta_{\nu(i+1)} \frac{\mu_{z+1}-1}{2(1+\sqrt{2})^{2z}} \beta_{\nu(j+1)} + \frac{1}{2}\beta_{\nu(j+1)}  \frac{\mu_{z+1}-1}{2(1+\sqrt{2})^{2z}} \beta_{\nu(j+1)} = 0. \qedhere
    \]
\end{proof}
The following lemma will be the base case for the subsequent inductive proof and itself requires nontrivial calculations.
\begin{lemma}
    \label{lem:base_case11}
    Suppose $0\leq j<i<2^k-1$ satisfy $\min \{p(i), p(j)\} = 2$. Then, $M_{i,j}=0$.
\end{lemma}
\begin{proof}
   If $z(i) \neq z(j)$ then the result follows from \Cref{lem:supportSimp}. 
    So, assume that $i = 2^z + i'$ and $j = 2^z + j'$ where $i', j' < 2^z$.
     \Cref{lem:secondDigit} shows that if $z(i') > z(j') + 1$, then $M_{i,j} = 0$. From now on, we will assume that $z(i') \le z(j') + 1$.

   We consider three cases, either $p(i) = p(j) = 2$; $p(i) = 2$ and $p(j) > 2$, and $p(i) > 2$ and $p(j) = 2$.

   \textbf{Case (i): $p(i) = p(j) = 2$}

   In this case, $i = 2^z + 2^r - 1$ and $j = 2^z + 2^t - 1$ for some $t < r < z$.
   By the assumption that $j<i$ and that $z(i')> z(j'+1)$, we have that $t = r-1$.
    Now, considering the support of $\lambda$, we have that
    \begin{align*}
        M_{i,j} &= -\frac{1}{2}\left(\lambda_{i,j} + \lambda_{j, i}\right) -  \frac{1}{2}h_{i} \sum_{\ell = i+1}^{2^{k+1}-1} \lambda_{\ell,j} + \frac{1}{2}h_{j}  \sum_{\ell = 0}^{j} \lambda_{\ell,i}\\
        &=-\frac{1}{2}\left(\lambda_{i,j} + \lambda_{j, i}\right) -  \frac{1}{2}h_{i} \lambda_{2^{z+1}-1,j} + \frac{1}{2}h_{j}  \left(\lambda_{j,i} + \lambda_{2^z-1,i}\right).
    \end{align*}

We note that:
    \begin{gather*}
                \lambda_{2^z-1,i} = 
        \lambda_{2^z-1,2^z + 2^r - 1} = \frac{\mu_{z+1}-1}{2(1+\sqrt{2})^{2z}}\beta_{r}\\
        \lambda_{2^{z+1}-1,j} = 
        \lambda_{2^{z+1}-1,2^z + 2^{r-1} - 1} = \frac{\mu_{z+1}-1}{2(1+\sqrt{2})^{2z}}\beta_{r-1}\\
        \lambda_{i,j} = 
        \lambda_{2^{z}+2^r-1,2^z + 2^{r-1} - 1} = \frac{\mu_{z+1}-1}{(1+\sqrt{2})^{2(z-2)-r+5}}\\
        \lambda_{j,i} = 
        \lambda_{2^{z}+2^{r-1}-1,2^z + 2^{r} - 1} = \frac{\mu_{z+1}-1}{(1+\sqrt{2})^{2(z-2)+3}}
    \end{gather*}
    Combining these values shows that $M_{i,j}=0$ (see \nextmathematica).

    \textbf{Case (ii): $p(i) = 2$, and $p(j) > 2$}

    Let $i = 2^z + 2^r - 1$, and let $j = 2^z + j'$, where $2^{r-1} - 1 < j' < 2^r-1$.
    We begin by noting
    \[
        \sum_{\ell = i+1}^{2^{k+1}-1} \lambda_{\ell,j} = \lambda_{2^{z+1}-1, j} = \frac{(\mu_{z+1}-1)}{2(1+\sqrt{2})^{2z}} \beta_{\nu(j+1)}.
    \]

    We break the remainder of the proof into two cases: where $j' = 2^r - 2^a - 1$ for some $a\geq 0$ or where $2^r-2^a -1< j' < 2^r - 2^{a-1} -1$ for some $a\geq 1$.

    For now, suppose $j'= 2^r - 2^a -1$ for some $a\geq 0$. Then,
by \Cref{lem:nonzero_entries_above}, $S_i^- = \{\ell \le i: \lambda_{\ell,i} \neq 0\} = \{2^z + 2^r - 2^b - 1: b \in [0,r]\}$. This means that 
\begin{align}
        \label{eq:useful_sum_above_11_base}
        \sum_{\ell = 0}^{j} \lambda_{\ell,i} &= \sum_{b = a}^r \lambda_{2^z + 2^r - 2^b - 1, i}\nonumber\\
        &= \sum_{b = a}^{r-1}\frac{\mu_{z+1}-1}{(1+\sqrt{2})^{2(z-(r-b+1)) + 3}} + \lambda_{2^z-1, i}\nonumber\\
        &= \frac{\mu_{z+1}-1}{2}(1+\sqrt{2})^{-2z}((1+\sqrt{2})^{2(r-a)} - 1) + \frac{\mu_{z+1}-1}{2(1+\sqrt{2})^{2z}}\beta_{r}.
\end{align}
Additionally, we have:
    \begin{gather*}
        \lambda_{i,j} = 
\frac{\mu_{z+1}-1}{(1+\sqrt{2})^{2(z-2)+3}}\left(\frac{1}{2}(1+\sqrt{2})^{2(r-a)-3}\beta_{a+2} - \frac{\beta_{a}}{2(1+\sqrt{2})}\right),\\
        \lambda_{j,i} 
= \frac{\mu_{z+1}-1}{(1+\sqrt{2})^{2(z-(1 + r - a)) + 3}}   .     
    \end{gather*}        
    Combining these expressions show (see \nextmathematica)
    \begin{align*}
        M_{i,j}= \frac{1}{2}\left(\beta_{a}\sum_{\ell=0}^{j}\lambda_{\ell,j} - \beta_r\sum_{\ell=i+1}^{2^{k+1}-1}\lambda_{\ell,j} - \lambda_{i,j}-\lambda_{j,i}\right)=0.
    \end{align*}

    Now, we consider the other subcase, in which, for some $a \ge 1$,
    \[
        2^r - 2^{a} - 1 < j' < 2^r - 2^{a-1} - 1.
    \]

    Again, we have that
    \begin{align*}
        \sum_{\ell=0}^{j}\lambda_{\ell,i}=\sum_{b=a}^r \lambda_{2^z+2^r-2^b-1,i},
    \end{align*}
    so that this sum is given by \eqref{eq:useful_sum_above_11_base}.
    Next, we have
    \begin{gather*}
        \lambda_{i,j} = 
        \frac{\mu_{z+1}-1}{2(1+\sqrt{2})^{2(z-2)+4}}((1+\sqrt{2})^{2(r-a+1)-2}-1)\beta_{\nu(j+1)},\\
        \lambda_{j,i} = 0.
    \end{gather*}
    Combining these expressions shows (see \nextmathematica)
    \[
        M_{i,j} = \frac{1}{2}\left(\beta_{\nu(j+1)}  \sum_{\ell = 0}^{j} \lambda_{\ell,i} -  \beta_{r} \sum_{\ell = i+1}^{2^{k+1}-1} \lambda_{\ell,j}  -\left(\lambda_{i,j} + \lambda_{j, i}\right)\right)= 0.
    \]

    \textbf{Case (iii): $p(j) = 2$ and $p(i) > 2$}
    
    Let $j = 2^z + 2^r - 1$.
    There are two subcases: where $z(i') = r+1$ and that where $z(i') = r$.

    If $z(i') = r+1$, then $i = 2^z + 2^{r+1} + i''$, where $i'' > 0$ by the assumption that $p(i)>2$. In this case,

\begin{align*}
        M_{i,j} &= \frac{1}{2}\left(\beta_r\sum_{\ell=0}^j\lambda_{\ell,i}   - \beta_{\nu(i+1)} \sum_{\ell=i+1}^{2^{k+1}-1}\lambda_{\ell,j}-\lambda_{i,j}-\lambda_{j,i}\right)\\
        &=\frac{1}{2}\left(\beta_r\lambda_{2^z-1,i}-\beta_{\nu(i+1)}\lambda_{2^{z+1}-1,j}\right).
    \end{align*}
    Combining this expression with the identities
    \begin{gather*}
        \lambda_{2^z-1,i}= \frac{\mu_{z+1}-1}{2(1+\sqrt{2})^{2z}}\beta_{\nu(i+1)}\\
        \lambda_{2^{z+1}-1,j} = \frac{\mu_{z+1}-1}{2(1+\sqrt{2})^{2z}}\beta_{r},
    \end{gather*}
    shows that $M_{i,j}=0$.
    
    Now, consider the case in which $z(i') = r$.
    In this case, we have
    \begin{gather*}
        \lambda_{j,i} = \frac{\mu_{z+1}-1}{(1+\sqrt{2})^{2(z-2)+3}} \beta_{\nu(i+1)},\\
        \lambda_{i,j} = 0.
    \end{gather*}
    It can be seen by \Cref{lem:nonzero_entries_above} that $S_i^- \cap \{0,\dots, j\} = \{j, 2^z-1\}$. That is,
    \[
        \sum_{\ell = 0}^{j} \lambda_{\ell,i} = \lambda_{j,i} + \lambda_{2^z-1, i} = \frac{\mu_{z+1}-1}{(1+\sqrt{2})^{2(z-2)+3}} \beta_{\nu(i+1)} + \frac{\mu_{z+1}-1}{2(1+\sqrt{2})^{2z}} \beta_{\nu(i+1)}.
    \]

    Similarly, it can be seen that 
    \begin{align*}
        \sum_{\ell = i+1}^{2^{k+1}-1} \lambda_{\ell,j}  &= \begin{cases}
            \lambda_{2^z + 2^{r+1} - 1,j} + \lambda_{2^{z+1}-1, j} & \text{ if } r < z-1\\
            \lambda_{2^{z+1}-1, j} & \text{ if } r = z-1\\
        \end{cases}\\
        &= \begin{cases}
            \frac{\mu_{z+1} - 1}{(1+\sqrt{2})^{2(z-2)- r+4}} + \frac{(\mu_{z+1}-1)}{2(1+\sqrt{2})^{2z}}\beta_{r} & \text{ if } r < z-1\\
            \frac{(\mu_{z+1}-1)}{2(1+\sqrt{2})^{2z}}\beta_{r+2} & \text{ if } r = z-1.\\
        \end{cases}
    \end{align*}
    One may check (see \nextmathematica) that the two expressions coincide when $r = z-1$. Thus, we may take the first expression in all cases.

    Combining the previous expressions, we have that
    \[
        M_{i,j} = \frac{1}{2}\left(\beta_r \sum_{\ell = 0}^{j} \lambda_{\ell,i} - \beta_{\nu(i+1)} \sum_{\ell = i+1}^{2^{k+1}-1} \lambda_{\ell,j} - \lambda_{j,i}\right) = 0.
    \]
    See \nextmathematica.
\end{proof}

\begin{lemma}
    \label{lem:M_11}
    Suppose $0\leq j<i < 2^k-1$, where neither $i+1$ nor $j+1$ is a power of $2$. Then, $M_{i,j}=0$.
\end{lemma}
\begin{proof}
    By \Cref{lem:supportSimp}, if $z(i) \neq z(j)$, then $M_{i,j} = 0$.
    We may thus assume $z(i)=z(j)<k$ and let $z$ denote their common value.

    We will show the result by induction on $\min\{p(i),p(j)\}$. 
    The base case, where $\min \{p(i), p(j)\} = 2$ is shown in \Cref{lem:base_case11}.
    In the remainder, we assume that $p(i),p(j)\geq 3$.

    Now, set $i' = i - 2^z$ and $j' = j - 2^z$.
    We see that $p(i') = p(i) - 1$ and $p(j') = p(j) - 1$ by considering the binary expansion of $i+1$ and $j+1$. Thus, neither $i'+1$ nor $j'+1$ is a power of two.

    We show the claim directly if $z(i')\neq z(j')$. In this case, it holds that $z(j') < z(i')$. In the sum
    \begin{align*}
        2M_{i,j} = -\left(\lambda_{i,j} + \lambda_{j, i}\right) -  h_{i} \sum_{\ell = i+1}^{2^{k+1}-1} \lambda_{\ell,j} + h_{j}  \sum_{\ell = 0}^{j} \lambda_{\ell,i},
    \end{align*}
    there are only two nonzero terms:
    \begin{align*}
        2M_{i,j} &= -h_{i} \lambda_{2^{z+1}-1,j} + h_{j} \lambda_{2^z-1,i}\\
            &= -\beta_{\nu(i+1)} \left(\frac{\mu_{z+1}-1}{2(1+\sqrt{2})^{2z}}\beta_{\nu(j+1)}\right) + \beta_{\nu(j+1)}  \left(\frac{\mu_{z+1}-1}{2(1+\sqrt{2})^{2z}}\beta_{\nu(i+1)}\right)\\
            &= 0.
    \end{align*}
    Here, the expression for $\lambda_{2^{z+1}-1,j}$ follows from the fact that $z(j') < z$.

    Now, assume that $z(i') = z(j')$ and denote their common value by $z'$.  By the inductive hypothesis, we may assume that 
    \[
        2M_{i',j'} = -\left(\lambda_{i',j'} + \lambda_{j', i'}\right) - h_{i'} \sum_{\ell = i'+1}^{2^{k+1}-1} \lambda_{\ell,j'} + h_{j'}  \sum_{\ell = 0}^{j'} \lambda_{\ell,i'} = 0.
    \]    Noting that $\lambda_{\ell,j'} = 0$ when $\ell \geq 2^{z'+1}$, and $\lambda_{\ell,i'} = 0$ when $\ell \leq 2^{z'}-2$, we can rewrite this as
    \[
        -\left(\lambda_{i',j'} + \lambda_{j', i'}\right) -  h_{i'} \sum_{\ell = i'+1}^{2^{z'+1}-1} \lambda_{\ell,j'} + h_{j'}  \sum_{\ell = 2^{z'}-1}^{j'} \lambda_{\ell,i'} = 0.
    \]
    We now wish to compare $M_{i,j}$ to $M_{i',j'}$.
    For this, we divide the summation defining $M_{i,j}$ into parts and then rearrange:
    \begin{align*}
        2M_{i,j} &= -\left(\lambda_{i,j} + \lambda_{j, i}\right) - h_{i} \sum_{\ell = i+1}^{2^z + 2^{z'+1} - 2} \lambda_{\ell,j} -
                               h_{i} \sum_{\ell = 2^z + 2^{z'+1}-1}^{2^{z+1} - 1} \lambda_{\ell,j}\\
                                &\qquad+ h_{j}  \sum_{\ell = 2^{z} - 1}^{2^{z} + 2^{z'} - 1} \lambda_{\ell,i} +
                               h_{j}  \sum_{\ell = 2^{z} + 2^{z'}}^{j} \lambda_{\ell,i}\\
                        &= \left(-\left(\lambda_{i,j} + \lambda_{j, i}\right)-h_{i} \sum_{\ell = i+1}^{2^z + 2^{z'+1} - 2} \lambda_{\ell,j}+ h_{j}  \sum_{\ell = 2^{z} + 2^{z'}}^{j} \lambda_{\ell,i}\right) \\
                               &\qquad + h_{j}  \sum_{\ell = 2^{z} - 1}^{2^{z} + 2^{z'} - 1} \lambda_{\ell,i}
                                -h_{i} \sum_{\ell = 2^z + 2^{z'+1} - 1}^{2^{z+1}-1} \lambda_{\ell,j}.
    \end{align*}
    The term in parentheses can be rewritten
    \begin{align*}
        &-\left(\lambda_{i,j} + \lambda_{j, i}\right)- h_{i} \sum_{\ell = i+1}^{2^z + 2^{z'+1} - 2} \lambda_{\ell,j}+ h_{j}  \sum_{\ell = 2^{z} + 2^{z'}}^{j} \lambda_{\ell,i}\\
        &\qquad =
        -\left(\lambda_{2^z+i',2^z+j'} + \lambda_{2^z+j', 2^z+i'}\right)- h_{2^z+i'} \sum_{\ell = 2^z+i'+1}^{2^z + 2^{z'+1} - 2} \lambda_{\ell,2^z+j'}+ h_{2^z+j'}  \sum_{\ell = 2^{z} + 2^{z'}}^{2^z+j'} \lambda_{\ell,2^z+i'}.
    \end{align*}
    Note that $h_{2^z+i'} = h_{i'}$, since $i'+1$ is not a power of 2, and $\nu(2^z+ i'+1) = \nu(i'+1)$. Similarly, $h_{2^z+j'} = h_{j'}$. We finally recall \Cref{lem:recurrence}, which shows that this expression is the same as
    \begin{align*}
        &(1+\sqrt{2})^{2(z'-z+1)} \frac{\mu_{z+1} - 1}{\mu_{z'+1}-1}\left(-\left(\lambda_{i',j'} + \lambda_{j', i'}\right)- h_{i'} \sum_{\ell = i'+1}^{2^{z'+1} - 2} \lambda_{\ell,j'}+h_{j'}  \sum_{\ell = 2^{z'}}^{j'} \lambda_{\ell,i'}\right)\\
        &\qquad= 
        (1+\sqrt{2})^{2(z'-z+1)} \frac{\mu_{z+1} - 1}{\mu_{z'+1}-1}\left(2M_{i',j'}+h_{i'}\lambda_{2^{z'+1}-1, j'} -  h_{j'}\lambda_{2^{z'}-1, i'}\right)
        \\
        &\qquad=(1+\sqrt{2})^{2(z'-z+1)} \frac{\mu_{z+1} - 1}{\mu_{z'+1}-1}\left(h_{i'}\lambda_{2^{z'+1}-1, j'} -  h_{j'}\lambda_{2^{z'}-1, i'}\right).
\end{align*}
    
    The next term we consider is
    \[
        h_{j}  \sum_{\ell = 2^{z} - 1}^{2^{z} + 2^{z'} - 1} \lambda_{\ell,i} = h_j (\lambda_{2^z-1,i} + \lambda_{2^{z} + 2^{z'} - 1,i}).
    \]

    The last term we need to consider is
    \[
        -h_{i} \sum_{\ell = 2^z + 2^{z'+1} - 1}^{2^{z+1} - 1} \lambda_{\ell,j}.
    \]
    Here, there are two cases: either $z' = z-1$ or it does not.
    
    Suppose that $z'  = z-1$, then
    \[
        -h_{i} \sum_{\ell = 2^z + 2^{z'+1} - 1}^{2^{z+1} - 1} \lambda_{\ell,j} = -h_{i} \lambda_{2^{z+1}-1,j},\]
    and combining, we obtain that 
    \begin{align*}
        2M_{i,j} &= 
        \frac{\mu_{z+1} - 1}{\mu_{z}-1}\left(h_{i'}\lambda_{2^{z}-1, j'} -  h_{j'}\lambda_{2^{z-1}-1, i'}\right)\\
        &\qquad+ h_j (\lambda_{2^z-1,i} + \lambda_{2^{z} + 2^{z-1} - 1,i})  - h_{i} \lambda_{2^{z+1}-1,j}\\
        &= h_i \left(\frac{\mu_{z+1} - 1}{\mu_{z}-1} \lambda_{2^{z}-1, j'} - \lambda_{2^{z+1}-1,j}\right) + h_j\left(\lambda_{2^z-1,i} + \lambda_{2^{z} + 2^{z-1} - 1,i} - \frac{\mu_{z+1} - 1}{\mu_{z}-1}\lambda_{2^{z-1}-1, i'}\right).
    \end{align*}
Here, we use the fact that $h_i= h_{i'}= \beta_{\nu(i+1)}$ and similarly  $h_j= h_{j'}= \beta_{\nu(j+1)}$.
    The first term in parentheses is zero by the definition of $\sigma$. 
    
    The second term in parantheses is also zero upon plugging in the various values of $\lambda$ (see \nextmathematica):
    \begin{gather*}
        \lambda_{2^{z-1}-1, i'} = \frac{\mu_{z} - 1}{2(1+\sqrt{2})^{2(z-1)}}\beta_{\nu(i+1)},\\
        \lambda_{2^z-1,i} = \frac{\mu_{z+1}-1}{2(1+\sqrt{2})^{2z}} \beta_{\nu(i+1)},\\
        \lambda_{2^{z} + 2^{z-1} - 1,i} = \frac{\mu_{z+1}-1}{(1+\sqrt{2})^{2z-1}} \beta_{\nu(i+1)}.
    \end{gather*}
    We deduce that $M_{i,j}=0$.

    If $z' < z-1$, then
    \begin{align*}
        -h_{i} \sum_{\ell = 2^z + 2^{z'+1} - 1}^{2^{z+1} - 1} \lambda_{\ell,j} &= -h_{i} (\lambda_{2^{z+1}-1,j} + \lambda_{2^{z} + 2^{z'+1}-1,j}).
    \end{align*}
    So that 
    \begin{align*}
        2M_{i,j} &= 
        (1+\sqrt{2})^{2(z'-z+1)} \frac{\mu_{z+1} - 1}{\mu_{z'+1}-1}\left(h_{i'}\lambda_{2^{z'+1}-1, j'} -  h_{j'}\lambda_{2^{z'}-1, i'}\right)\\
        &\qquad + h_j (\lambda_{2^z-1,i} + \lambda_{2^{z} + 2^{z'} - 1,i}) - h_{i} (\lambda_{2^{z+1}-1,j} + \lambda_{2^{z} + 2^{z'+1}-1,j}) \\
        &= 
        h_{i} \left((1+\sqrt{2})^{2(z'-z+1)} \frac{\mu_{z+1} - 1}{\mu_{z'+1}-1}\lambda_{2^{z'+1}-1, j'} - (\lambda_{2^{z+1}-1,j} + \lambda_{2^{z} + 2^{z'+1}-1,j}) \right)\\
        &\qquad -  h_{j}\left((1+\sqrt{2})^{2(z'-z+1)} \frac{\mu_{z+1} - 1}{\mu_{z'+1}-1}\lambda_{2^{z'}-1, i'} - (\lambda_{2^z-1,i} + \lambda_{2^{z} + 2^{z'} - 1,i}) \right).
    \end{align*}
    The second term is identically zero due to the identities (see \nextmathematica)
    \begin{gather*}
        \lambda_{2^{z'}-1, i'} = \frac{\mu_{z'+1}-1}{2(1+\sqrt{2})^{2z'}}\beta_{\nu(i+1)},\\
        \lambda_{2^z-1,i} = \frac{\mu_{z+1}-1}{2(1+\sqrt{2})^{2z}}\beta_{\nu(i+1)},\\
        \lambda_{2^{z} + 2^{z'} - 1,i} = \frac{\mu_{z+1}-1}{(1+\sqrt{2})^{2z - 1}}\beta_{\nu(i+1)}.
    \end{gather*}
    It remains to show that the first term is also zero. Let $r = 2^{z'+1}-1 - j'$. We have that $1\leq r < 2^{z'}$ and that $\tau= z(r-1) \in[0,z'-1]$. There are two final cases: where $r = 2^\tau$ and where $2^{\tau}<r< 2^{\tau+1}$.

    In the first case, we additionally have that $\nu(j+1) = \tau$ and
    \begin{align*}
        \lambda_{2^{z'+1}-1,j'} &= \frac{(\mu_{z'+1} -1)\beta_{\tau+2}}{2(1+\sqrt{2})^{2(\tau+1)}},\\
        \lambda_{2^{z+1} - 1, j} &= \frac{(\mu_{z+1}-1)\beta_{\tau}}{2(1+\sqrt{2})^{2z}},\\
        \lambda_{2^z+2^{z'+1}-1, j} &= \frac{\mu_{z+1} - 1}{2(1+\sqrt{2})^{2z}} \left((1+\sqrt{2})^{2(z'-\tau)}\beta_{\tau+2} - \beta_\tau\right).
    \end{align*}

    Otherwise, we have $\nu(j+1) = \nu(r)$ and
    \begin{align*}
        \lambda_{2^{z'+1}-1,j'} &= \frac{(\mu_{z'+1} - 1)\beta_{\nu(j+1)}}{2(1+\sqrt{2})^{2(\tau+1)}},\\
        \lambda_{2^{z+1} - 1, j} &= \frac{(\mu_{z+1} - 1)\beta_{\nu(j+1)}}{2(1+\sqrt{2})^{2z}},\\
        \lambda_{2^z+2^{z'+1}-1, j} 
        &= \frac{(\mu_{z+1}-1)\beta_{\nu(j+1)}}{2(1+\sqrt{2})^{2z}}\left((1+\sqrt{2})^{2(z'-\tau)} - 1\right).
    \end{align*}
    In both cases, plugging the relevant $\lambda$ values into the first term in parentheses shows that it is equal to zero. See \nextmathematica{} and \nextmathematica.\qedhere

\end{proof}

\subsection{Off-Diagonal Entries in Case $(2,2)$}

Our goal for this subsection will be to prove \Cref{lem:M_13}, stating that the off-diagonal entries of $M_{i,j}$ with $i,j < 2^k-1$ are all 0.
We start with a simplifying lemma:
\begin{lemma}
    \label{lem:supportSimp2}
    Letting $h = \mathfrak{h}^{(k)}$ and $\lambda = \lambda^{(k)}$, and fix $j < i \le 2^k-1$ so that neither $i+1$ nor $j+1$ are powers of 2. If $z(i) \neq z(j)$, then $M(h, \lambda)_{\rev(i), \rev(j)} = \frac{1}{2}\phi_{\rev(i)} \phi_{\rev(i)}$.
\end{lemma}
\begin{proof}
    If $z(i) \neq z(j)$, then $\lambda_{\rev(i),\rev(j)} = \lambda_{\rev(j),\rev(i)} = 0$. For $\ell > \rev(i)$, $\lambda_{\ell, \rev(j)} = 0$. 

    \Cref{lem:telescoping} implies that $\sum_{\ell = 0}^{\rev{j}} \lambda_{\ell, \rev(i)} = \frac{\beta_{\nu(\rev(i)+1)}}{\sqrt{\mu_{z(i)+1}-1}\sqrt{\mu_{z(j)+1}-1}}$. Therefore,
    \begin{align*}
        M(h, \lambda)_{\rev(i),\rev(j)} &= \frac{1}{2}h_{j}  \sum_{\ell=0}^{\rev(i)} \lambda_{\rev(2^a-1), \rev(i)}\\
            &= \frac{\beta_{\nu(j+1)}\beta_{\nu(i+1)}}{\sqrt{\mu_{z(i)+1}-1}\sqrt{\mu_{z(j)+1}-1}}\\
            &= \frac{1}{2}\phi_{\rev(i)} \phi_{\rev(i)}.
    \end{align*}
    Here, we use the fact that the series is telescoping to simplify the computation.
\end{proof}
The following lemma will be the base case for a subsequent inductive proof.
\begin{lemma}
    \label{lem:base_case22}
    Let $r < \ell < 2^k-1$, with $\min(p(r), p(\ell)) = 2$, then 
    \[
        M_{\rev(\ell),\rev(r)} = \frac{1}{2}\phi_{\rev(r)}\phi_{\rev(\ell)}.
    \]
\end{lemma}
\begin{proof}
    If $z(r) \neq z(\ell)$, then \Cref{lem:telescoping} implies the result, so assume that $z(r) = z(\ell)$. Let their common value be $z$.
    
    We consider two cases: where $p(\ell) = 2$, and where $p(r) = 2,$ and $p(\ell) \ge 3$.

    \textbf{Case (i): $p(\ell) = 2$}
    
    Let $\ell = 2^z + 2^b - 1$, and let $r = 2^z + r'$ with $b > z(r')$. 
    In this case,
    $\lambda_{\rev(r), \rev(\ell)} = 0$
    because $\rev(r) - 2^{\nu(\rev(r)+1)-1} = 2^k - 2^z - (r' + \lfloor 2^{\nu(r'+1)-1} \rfloor) > 2^k - 2^z - 2^b - 1$.

    By definition,
    \[
        \lambda_{\rev(\ell), \rev(r)} = \frac{(1+\sqrt{2})^{2(p(\rev(\ell)) + z - k)-1}}{\mu_{z+1}-1}\beta_{\nu(r+1)}.
    \]
    Noting that $p(\rev(\ell)) = k-b$, so
    \[
        \lambda_{\rev(\ell), \rev(r)} = \frac{(1+\sqrt{2})^{2(z - b)-1}}{\mu_{z+1}-1}\beta_{\nu(r+1)}.
    \]

    It can be seen by considering the support of column $\rev(\ell) = 2^{k+1} - 2^z - 2^b - 1$  that 
    \begin{align*}
        \sum_{c = \rev(r)+1}^{2^{k+1}-1} \lambda_{c, \rev(\ell)} &=  \lambda_{\rev(2^z-1), \rev(\ell)}\\
        &= \frac{1}{2(\mu_{z+1}-1)}((1+\sqrt{2})^{2(z-b-1)}\beta_{b+2}-\beta_b).
    \end{align*}

    Note that the binary expansion of 
    \[
    \rev(r) = 2^{k+1} - r - 2 = 2^{k+1} - 2^z - r' - 2
    \]
    is of the form
    \[
        \sum_{t = z}^{k+1} 2^t + \sum_{t=z(r')}^{z-1} 2^t + \dots,
    \]
    so by \Cref{lem:nonzero_entries_above}, we have that 
    \[
        S^-_{\rev(r)} \cap \{0, \dots, \rev(\ell)\} = \{
        \sum_{t = z+1}^{k+1} 2^t + \sum_{t=t_0}^{z-1} 2^t - 1 : t_0 \ge b\} \cup \{
        \sum_{t = t_0}^{k+1} 2^t - 1 : t_0 \ge z\}.
    \]
    Therefore,
    
    \begin{align*}
        \sum_{c = 0}^{\rev(\ell)} \lambda_{c, \rev(r)} &=  \sum_{t = z+1}^{k} \lambda_{\rev(2^t-1), \rev(r)} + \sum_{t = b}^{z-1} \lambda_{\rev(2^z+2^t-1), \rev(r)}\\
        &=  \frac{\beta_a}{\mu_{z+1}-1} + \sum_{t = b}^{z-1} \frac{(1+\sqrt{2})^{2(p(\rev(2^z+2^t-1))+z-k)-1}}{\mu_{z+1}-1}\beta_{\nu(r+1)}\\
        &= \frac{1}{2}\left((1+\sqrt{2})^{2(z-b)}+1\right)\frac{\beta_{\nu(r+1)}}{\mu_{z+1}-1}.
    \end{align*}
    Here, in the second line, we use \Cref{lem:telescoping} to simplify the first summation. The second summation uses \cref{lem:rev_identities} to write $p(\rev(2^z + 2^t -1)) = k - t$ (see \nextmathematica).

    Combining, we obtain that
    \begin{align*}
        M_{\rev(\ell), \rev(r)} &= 
        \frac{1}{2(\mu_{z+1}-1)}
        \bigg(
         -((1+\sqrt{2})^{2(z - b)-1})\beta_{\nu(r+1)} - \frac{\beta_{\nu(r+1)}}{2}\left((1+\sqrt{2})^{2(z-b-1)}\beta_{b+2}-\beta_b\right)\\
         &\qquad+\beta_b  \left(\frac{1}{2}\beta_{2(z-b)+1}\beta_a\right)\bigg)\\
         &= \frac{1}{2(\mu_{z+1}-1)}\beta_{\nu(r+1)}\beta_b\\
         &= \frac{1}{2}\phi_{\rev(\ell)} \phi_{\rev(r)}.
    \end{align*}
    See \nextmathematica.

    \textbf{Case (ii): $p(r) = 2$ and $p(\ell) \ge 3$}
    Write $r = 2^z + 2^a -1$ where $z>a$.
    By \cref{lem:M_entries}, we have
    \begin{align*}
        \frac{1}{2}\left(
            \beta_{\nu(\ell+1)}\sum_{i=0}^{\rev(\ell)}\lambda_{i,\rev(r)}-\beta_{\nu(r+1)}\sum_{i= \rev(r) + 1}^{2^{k+1}-1} \lambda_{i,\rev(\ell)}-\lambda_{\rev(\ell),\rev(r)}-\lambda_{\rev(r),\rev(\ell)}\right).
    \end{align*}
    
    Let $z'= z(\ell - 2^z)$. Note that we must have $z'\geq a$ as otherwise $r>\ell$. Now, there are two cases $z'>a$ and $z'= a$.
    
    First, suppose $z'>a$. The set
    \begin{align*}
        S_{\rev(r)}^- \cap \set{0,\dots, \rev(\ell)} = \set{\rev(2^z + 2^\tau -1):\, \tau \in[z'+1, z-1]} \cup \set{\rev(2^\tau - 1):\, \tau\in[z+1,k]}.
    \end{align*}
    Thus, the first summation is given by
    \begin{align*}
        &\sum_{\tau\in[z+1,k]}\lambda_{\rev(2^\tau -1), \rev(r)} + \sum_{\tau\in[z'+1,z-1]}\lambda_{\rev(2^z+2^\tau -1), \rev(r)}\\
        &\qquad= \sum_{\tau\in[z+1,k-1]}\left(\frac{1}{\sqrt{\mu_\tau-1}}-\frac{1}{\sqrt{\mu_{\tau+1}-1}}\right)\frac{\beta_{\nu(r+1)}}{\sqrt{\mu_{z+1}-1}} + \frac{1}{\sqrt{\mu_{k}-1}}\frac{\beta_{\nu(r+1)}}{\sqrt{\mu_{z+1}-1}}\\
        &\qquad\qquad+ \sum_{\tau\in[z'+1,z-1]}\lambda_{\rev(2^z+2^\tau -1), \rev(r)}\\
        &\qquad=\frac{\beta_{\nu(r+1)}}{\mu_{z+1}-1} + \sum_{\tau\in[z'+1,z-1]}\frac{(1+\sqrt{2})^{2(z-\tau)-1}}{\mu_{z+1}-1}\beta_{\nu(r+1)}\\
        &\qquad = \frac{\beta_{\nu(r+1)}}{\mu_{z+1}-1}\left(1+\sum_{\tau\in[z'+1,z-1]}(1+\sqrt{2})^{2(z-\tau)-1}\right)\\
        &\qquad = \begin{cases}
            \frac{\beta_{\nu(r+1)}}{\mu_{z+1}-1} & \text{if }z=z'+1\\
            \frac{\beta_{\nu(r+1)}}{2(\mu_{z+1}-1)}\left((1+\sqrt{2})^{2(z-z'-1)}+1\right)&\text{else}.
        \end{cases}
    \end{align*}
    
    The only possible nonzero term in the second summation occurs in row $\rev(2^z-1)$. Thus, the second summation is equal to 
    \begin{align*}
        \sum_{i=\rev(r)+1}^{2^{k+1}-1}\lambda_{i,\rev(\ell)} &= \lambda_{\rev(2^z-1), \rev(\ell)}\\
        &= \begin{cases}
        0 & \text{if }z=z'+1\\
        \frac{\beta_{\nu(\ell+1)}}{2(\mu_{z+1}-1)}\left((1+\sqrt{2})^{2(z-z'-1)}-1\right)& \text{else} .
        \end{cases}
    \end{align*}
    
    We also have that
    \begin{gather*}
        \lambda_{\rev(\ell),\rev(r)} = \lambda_{\rev(r),\rev(\ell)} = 0
    \end{gather*}
    in this case.
    
    In both cases, plugging our expressions for the two summations back into our expression for $M_{\rev(r),\rev(\ell)}$ shows that
    \begin{align*}
        M_{\rev(r),\rev(\ell)} = \frac{\beta_{\nu(r+1)}\beta_{\nu(\ell+1)}}{2(\mu_{z(r)+1}-1)}
    \end{align*}
    as desired.

    Now, suppose $z'=a$ and let $z'' = z(\ell - 2^z - 2^{z'})$.
    We further distinguish two cases: where $z'' = \nu(\ell+1)$ and where $z'' > \nu(\ell + 1)$.
    
    If $z'' = \nu(\ell + 1)$, then $\ell = 2^z + 2^{z'} + 2^{z''} - 1$ and
    \begin{align*}
        S^-_{\rev(r)} \cap \set{0,\dots, \rev(\ell)}  &=\set{\rev(2^z + 2^{z'} + 2^\tau - 1):\, \tau\in[z'',z'-1]}\\
        &\qquad \cup \set{\rev(2^z + 2^\tau - 1):\, \tau\in[z'+1,z-1]}\\
        &\qquad \cup \set{\rev(2^\tau -1) :\, \tau \in[z+1,k]}.
    \end{align*}
    Thus,
    \begin{align*}
        \sum_{i=0}^{\rev(\ell)}\lambda_{i,\rev(r)} &= \sum_{\tau = z''}^{z'-1}\lambda_{\rev(2^z+2^{z'} + 2^\tau - 1),\rev(r)} + \sum_{\tau = z'+1}^{z-1} \lambda_{\rev(2^z + 2^\tau -1),\rev(r)} + \sum_{\tau = z+1}^k \lambda_{\rev(2^\tau - 1),\rev(r)}\\
        &= \sum_{\tau = z''}^{z'-1}\frac{(1+\sqrt{2})^{2(z-\tau) - 3}}{\mu_{z+1} - 1} + \sum_{\tau = z'+1}^{z-1}\frac{(1+\sqrt{2})^{2(z-\tau) - 1}}{\mu_{z+1} - 1}\beta_{\nu(r+1)} \\
        &\qquad\qquad + \sum_{\tau = z+1}^{k-1}\left(\frac{1}{\sqrt{\mu_\tau -1}} - \frac{1}{\sqrt{\mu_{\tau + 1} -1}}\right)\frac{\beta_{\nu(r+1)}}{\sqrt{\mu_{z+1} - 1}} + \frac{\beta_{\nu(r+1)}}{\sqrt{\mu_k - 1}\sqrt{\mu_{z+1} - 1}}\\
        &= \frac{(1+\sqrt{2})^{2z - 2}}{2(\mu_{z+1} - 1)}\left((1+\sqrt{2})^{-2z''} - (1+\sqrt{2})^{-2z'}\right)  + \frac{\beta_{\nu(r+1)}}{\mu_{z+1} - 1}\\
        &\qquad\qquad +
        \begin{cases}
        0 & \text{if }z= z'+1\\
        \frac{\beta_{\nu(r+1)}}{2(\mu_{z+1} - 1)}\left((1+\sqrt{2})^{2(z-z'-1)} - 1\right)& \text{else}
        \end{cases}
.    \end{align*}

    The only possible nonzero term in the second summation is
    \begin{align*}
        \lambda_{\rev(2^z - 1), \rev(\ell)} &= \begin{cases}
        0 & \text{if } z = z'+1\\
        \frac{\beta_{\nu(\ell+1)}}{2(\mu_{z+1} - 1)}\left((1+\sqrt{2})^{2(z-z'-1)} - 1\right) & \text{else}.
        \end{cases}
    \end{align*}
    
    We deduce that regardless of whether $z = z'+1$, that
    \begin{align}
        \label{eq:M22_base_case_intermediate}
        &\beta_{\nu(\ell+1)}\sum_{i=0}^{\rev(\ell)}\lambda_{i,\rev(r)}-\beta_{\nu(r+1)}\sum_{i= \rev(r) + 1}^{2^{k+1}-1} \lambda_{i,\rev(\ell)}\\
        &\qquad\qquad= \beta_{\nu(\ell+1)}\frac{(1+\sqrt{2})^{2z - 2}}{2(\mu_{z+1} - 1)}\left((1+\sqrt{2})^{-2z''} - (1+\sqrt{2})^{-2z'}\right)  + \frac{\beta_{\nu(r+1)}\beta_{\nu(\ell+1)}}{\mu_{z+1} - 1}.
    \end{align}
    
    We also compute
    \begin{align*}
        \lambda_{\rev(r),\rev(\ell)} &= \begin{cases}
        \frac{(1+\sqrt{2})^{2(z - z')-3+z'}}{\mu_{z+1} - 1} &\text{if } z' = z'' + 1\\
        \frac{(1+\sqrt{2})^{2(z - z')-1}}{\mu_{z+1}-1}\left(\frac{1}{2}(1+\sqrt{2})^{2(z'-z'')-3}\beta_{z''+2} - \frac{\beta_{z''}}{2(1+\sqrt{2})}\right) &\text{else},
        \end{cases}\\
        \lambda_{\rev(\ell),\rev(r)} &= \frac{(1+\sqrt{2})^{2(z - z'')-3}}{\mu_{z+1} - 1}.
    \end{align*}
    
    One may check (see \nextmathematica) that the two expressions for $\lambda_{\rev(r),\rev(\ell)}$ coincide when $z' = z'' + 1$, thus we may take the second expression even when $z' = z'' + 1$. Finally, subtracting $\lambda_{\rev(r),\rev(\ell)}$ and $\lambda_{\rev(\ell),\rev(r)}$ from \eqref{eq:M22_base_case_intermediate} gives
    \begin{align*}
        M_{\rev(r),\rev(\ell)}&=\frac{1}{2}\bigg(\beta_{z''}\frac{(1+\sqrt{2})^{2z - 2}}{2(\mu_{z+1} - 1)}\left((1+\sqrt{2})^{-2z''} - (1+\sqrt{2})^{-2z'}\right)  + \frac{\beta_{\nu(r+1)}\beta_{\nu(\ell+1)}}{\mu_{z+1} - 1}\\
        &\qquad\qquad - \frac{(1+\sqrt{2})^{2(z - z'')-3}}{\mu_{z+1} - 1} - \frac{(1+\sqrt{2})^{2(z - z')-1}}{\mu_{z+1}-1}\left(\frac{1}{2}(1+\sqrt{2})^{2(z'-z'')-3}\beta_{z''+2} - \frac{\beta_{z''}}{2(1+\sqrt{2})}\right)\bigg)\\
        &= \frac{1}{2}\frac{\beta_{\nu(r+1)}\beta_{\nu(\ell+1)}}{\mu_{z+1} - 1},
    \end{align*}
    as desired (see \nextmathematica).
    
    The final case is when $z'=a$ and $z'' > \nu(\ell+1)$.
    In this case
    \begin{align*}
        S^-_{\rev(r)} \cap \set{0,\dots, \rev(\ell)}  &=\set{\rev(2^z + 2^{z'} + 2^\tau - 1):\, \tau\in[z''+1,z'-1]}\\
        &\qquad \cup \set{\rev(2^z + 2^\tau - 1):\, \tau\in[z'+1,z-1]}\\
        &\qquad \cup \set{\rev(2^\tau -1) :\, \tau \in[z+1,k]}.
    \end{align*}
    Thus,
    \begin{align*}
    \sum_{i=0}^{\rev(\ell)}\lambda_{i,\rev(r)} &= \sum_{\tau=z''+1}^{z'-1} \lambda_{\rev(2^z +2^{z'} + 2^\tau - 1), \rev(r)} + \sum_{\tau=z'+1}^{z-1} \lambda_{\rev(2^z + 2^\tau -1), \rev(r)} + \sum_{\tau = z+1}^k \lambda_{\rev(2^\tau - 1), \rev(r)}\\
    &= \sum_{\tau=z''+1}^{z'-1} \frac{(1+\sqrt{2})^{2(z-\tau)-3}}{\mu_{z+1} - 1} + \sum_{\tau = z'+1}^{z-1}\frac{(1+\sqrt{2})^{2(z- \tau)-1}}{\mu_{z+1} - 1}\beta_{\nu(r+1)} + \frac{\beta_{\nu(r+1)}}{\mu_{z+1} - 1}.
    \end{align*}
    
    The only possible nonzero term in $S^+_{\rev(\ell)}\cap\set{\rev(r)+1,\dots,2^{k+1}-1}$ is
    \begin{align*}
        \lambda_{\rev(2^z - 1), \rev(\ell)} &= \begin{cases}
        0 & \text{if } z = z'+1\\
        \frac{\beta_{\nu(\ell+1)}}{2(\mu_{z+1} - 1)}\left((1+\sqrt{2})^{2(z-z'-1)} - 1\right) & \text{else}.
        \end{cases}
    \end{align*}
    
    We also compute
    \begin{gather*}
        \lambda_{\rev(r),\rev(\ell)} = \begin{cases}
        0 & \text{if } z' = z'' + 1\\
        \frac{(1+\sqrt{2})^{2(z-z'-1)}}{2(\mu_{z+1} - 1)}\left((1+\sqrt{2})^{2(z'-z'') - 2} - 1\right)\beta_{\nu(\ell+1)} & \text{else}
        \end{cases},\\
        \lambda_{\rev(\ell),\rev(r)} = 0.
    \end{gather*}
    
    Combining these expressions with our expression for $M_{\rev(r),\rev(\ell)}$ gives
    \begin{align}
        M_{\rev(r),\rev(\ell)} &= \frac{1}{2}\Bigg(
        \beta_{\nu(\ell+1)}\left(\sum_{\tau=z''+1}^{z'-1} \frac{(1+\sqrt{2})^{2(z-\tau)-3}}{\mu_{z+1} - 1} + \beta_{z'}\sum_{\tau = z'+1}^{z-1}\frac{(1+\sqrt{2})^{2(z- \tau)-1}}{\mu_{z+1} - 1}+ \frac{\beta_{\nu(r+1)}}{\mu_{z+1} - 1}\right)\nonumber\\
        &\qquad\qquad- \beta_{z'} \lambda_{\rev(2^z - 1), \rev(\ell)}
        - \lambda_{\rev(r),\rev(\ell)}
        \Bigg)\nonumber\\
        &= \frac{\beta_{\nu(\ell+1)}\beta_{\nu(r+1)}}{2(\mu_{z+1} - 1)} + \frac{1}{2}\left[\beta_{\nu(\ell+1)}\sum_{\tau=z''+1}^{z'-1} \frac{(1+\sqrt{2})^{2(z-\tau)-3}}{\mu_{z+1} - 1} - \lambda_{\rev(r),\rev(\ell)}\right] \label{eq:M22_base_case_pr2_pell4}\\
        &\qquad\qquad +\frac{\beta_{z'}}{2}\left[\beta_{\nu(\ell+1)}\sum_{\tau = z'+1}^{z-1}\frac{(1+\sqrt{2})^{2(z- \tau)-1}}{\mu_{z+1} - 1} - \lambda_{\rev(2^z - 1), \rev(\ell)}\right].\nonumber
    \end{align}
    It remains to show that the two square-bracketed terms in \eqref{eq:M22_base_case_pr2_pell4} are zero.
    
    For the first square-bracketed term, note that if $z' = z''+1$, then the summation is empty and $\lambda_{\rev(r),\rev(\ell)}=0$. Else, the square-bracketed term is
    \begin{align*}
        \frac{\beta_{\nu(\ell+1)}}{2(\mu_{z+1} - 1)}\left(
        \sum_{\tau=z''+1}^{z'-1} 2(1+\sqrt{2})^{2(z-\tau)-3} - (1+\sqrt{2})^{2(z-z'-1)}\left((1+\sqrt{2})^{2(z'-z'') - 2} - 1\right)\right).
    \end{align*}
    This is identically zero as can be seen in \nextmathematica.
    
    For the second square-bracketed term, note that if $z=z'+1$, then summation is empty and $\lambda_{\rev(2^z - 1), \rev(\ell)}=0$. Else, the square-bracketed term is
    \begin{align*}
        \frac{\beta_{\nu(\ell+1)}}{2(\mu_{z+1} - 1)}\left(\sum_{\tau = z'+1}^{z-1}2(1+\sqrt{2})^{2(z- \tau)-1} - (1+\sqrt{2})^{2(z-z'-1)} + 1\right).
    \end{align*}
    This is identically zero as can be seen in \nextmathematica.
        
\end{proof}

\begin{lemma}
    \label{lem:M_22}
Suppose $0\leq r < \ell < 2^k-1$ and neither $r+1$ nor $\ell+1$ a power of $2$. Then,
\[
    M_{\rev(\ell),\rev(r)} = \frac{1}{2}\phi_{\rev(r)}\phi_{\rev(\ell)}.
\]
\end{lemma}
\begin{proof}
We induct on $\min(p(r),p(\ell))$. The base case in which $\min(p(r), p(\ell)) = 2$ is given in \Cref{lem:base_case22}.

Now, suppose $\min(p(r),p(\ell))\geq 3$. If $z(\ell) \neq z(r)$, then we may directly apply \Cref{lem:supportSimp2}.

In the remainder of the proof, we may assume without loss of generality that $z\coloneqq z(\ell)=z(r)$ and $\min(p(r),p(\ell))\geq 3$.

Let $\ell'= \ell-2^z$ and $r'= r-2^z$. We will handle the case $z(\ell')\neq z(r')$ directly. In this case
\begin{align*}
    &M_{\rev(\ell),\rev(r)}\\
    &\qquad= \frac{1}{2}\bigg(
    \beta_{\nu(\ell+1)}\sum_{i=0}^{\rev(\ell)} \lambda_{i,\rev(r)}-\beta_{\nu(r+1)}\sum_{i= \rev(r) +1}^{2^{k+1}-1} \lambda_{i,\rev(\ell)}\\
    &\qquad\qquad\qquad -\lambda_{\rev(\ell),\rev(r)}-\lambda_{\rev(r),\rev(\ell)}\bigg).
\end{align*}
The entries $\lambda_{\rev(r),\rev(\ell)}$ and $\lambda_{\rev(\ell),\rev(r)}$ are both zero. The relevant indices in the first summation are given by the two sets
\begin{align*}
    \set{\rev(2^z + 2^\tau - 1:\, \tau\in[z(\ell')+1, z-1]} \cup \set{\rev(2^\tau - 1):\, \tau\in[z+1,k]}.
\end{align*}
The only nonzero entry in the second summation occurs at entry $\lambda_{\rev(2^z-1),\rev(\ell)}$. Thus,
\begin{align*}
    M_{\rev(r),\rev(\ell)}
    &=\frac{\beta_{\nu(\ell+1)}}{2}\sum_{\tau = z(\ell')+1}^{z-1} \lambda_{\rev(2^z+2^\tau-1), \rev(r)} + \frac{\beta_{\nu(\ell+1)}}{2}\sum_{\tau=z+1}^k \lambda_{\rev(2^\tau - 1), \rev(r)}\\
    &\qquad - \frac{\beta_{\nu(r+1)}}{2}\lambda_{\rev(2^z-1),\rev(\ell)}\\
    &= \frac{\beta_{\nu(\ell+1)}\beta_{\nu(r+1)}}{2(\mu_{z+1} - 1)} + \frac{\beta_{\nu(r+1)}}{2}\left(\frac{\beta_{\nu(\ell+1)}}{\mu_{z+1} - 1}\sum_{\tau = z(\ell')+1}^{z-1} (1+\sqrt{2})^{2(z-\tau) - 1} - \lambda_{\rev(2^z-1),\rev(\ell)}\right).
\end{align*}
Here, we have used \Cref{lem:telescoping} to simplify the second summation on the first line. The goal is to now show that the term in parentheses is zero. When $z(\ell') = z-1$, the summation is empty and the term $\lambda_{\rev(2^z-1),\rev(\ell)}$ is zero. On the other hand, if $z(\ell') < z-1$, then
\begin{align*}
    \lambda_{\rev(2^z - 1), \rev(\ell)} = \frac{\beta_{\nu(\ell+1)}}{2(\mu_{z+1} - 1)}\left((1+\sqrt{2})^{2(z-z(\ell')-1)}-1 \right).
\end{align*}
Substituting this expression shows that the term in parentheses is zero (see \nextmathematica).

Thus, in the remainder we may additionally assume $z'\coloneqq z(\ell')=z(r')$.
We will repeatedly use the fact that $\nu(r+1)= \nu(2^z + r' + 1 )= \nu(r'+1)$ and similarly for $\ell$ and $\ell'$.

By the inductive hypothesis, it holds that 
\begin{align*}
    &\frac{\beta_{\nu(r+1)}\beta_{\nu(\ell+1)}}{2(\mu_{z'+1}-1)}\\
    &\qquad = M_{\rev(\ell'), \rev(r')}\\
    &\qquad=\frac{1}{2}\bigg(
        \beta_{\nu(\ell+1)}\sum_{i=2^k-1}^{\rev(\ell')} \lk_{i,\rev(r)}-\beta_{\nu(r+1)}\sum_{i= \rev(r)+1}^{2^{k+1}-2^{z'}-1} \lk_{i,\rev(\ell')}\\
        &\qquad\qquad\qquad -\lk_{\rev(\ell'),\rev(r')}-\lk_{\rev(r'),\rev(\ell')}\bigg)\\
    &\qquad=\frac{1}{2}\bigg(
        \beta_{\nu(\ell+1)}\sum_{i=2^k-1}^{2^{k+1}-2^{z'+1}-1} \lk_{i,\rev(r')}+\beta_{\nu(\ell+1)}\sum_{i=2^{k+1}-2^{z'+1}}^{\rev(\ell')} \lk_{i,\rev(r')}\\
        &\qquad\qquad\qquad-\beta_{\nu(r+1)}\sum_{i= 2^{k+1}-r' - 1}^{2^{k+1}-2^{z'}-2} \lk_{i,\rev(\ell')}-\beta_{\nu(r+1)}\lk_{2^{k+1}-2^{z'}-1,\rev(\ell')}\\
        &\qquad\qquad\qquad -\lk_{\rev(\ell'),\rev(r')}-\lk_{\rev(r'),\rev(\ell')}\bigg)\\
    &\qquad=\frac{1}{2}\bigg(
        \frac{\beta_{\nu(\ell+1)}\beta_{\nu(r+1)}}{\mu_{z'+1}-1}+\beta_{\nu(\ell+1)}\sum_{i=2^{k+1}-2^{z'+1}}^{\rev(\ell')} \lk_{i,\rev(r')}\\
        &\qquad\qquad\qquad-\beta_{\nu(r+1)}\sum_{i= 2^{k+1}-r' - 1}^{2^{k+1}-2^{z'}-2} \lk_{i,\rev(\ell')}-\beta_{\nu(r+1)}\lk_{2^{k+1}-2^{z'}-1,\rev(\ell')}\\
        &\qquad\qquad\qquad -\lk_{\rev(\ell'),\rev(r')}-\lk_{\rev(r'),\rev(\ell')}\bigg).
\end{align*}
Here, we have used the fact that 
\[\sum_{i=2^k-1}^{2^{k+1}-2^{z'+1}-1} \lk_{i,2^{k+1}-2-r'} = 
 \frac{\beta_{\nu(r'+1)}}{\mu_{z'+1}-1},
\]
from \Cref{lem:telescoping}.

Rearranging this identity, we have that
\begin{align*}
    &\beta_{\nu(\ell+1)}\sum_{i=2^{k+1}-2^{z'+1}}^{2^{k+1}-2-\ell'} \lk_{i,2^{k+1}-2-r'}-\beta_{\nu(r+1)}\sum_{i= 2^{k+1}-r' - 1}^{2^{k+1}-2^{z'}-2} \lk_{i,2^{k+1}-2-\ell'}\\
    &\qquad\qquad = \beta_{\nu(r+1)}\lk_{2^{k+1}-2^{z'}-1,2^{k+1}-2-\ell'} + \lk_{2^{k+1}-2-\ell',2^{k+1}-2-r'}+\lk_{2^{k+1}-2-r',2^{k+1}-2-\ell'}.
\end{align*}

We will now rearrange terms in the summation for $M_{\rev(\ell), \rev(r)}$:
\begin{align*}
    &M_{\rev(\ell), \rev(r)}\\
    &\qquad=\frac{1}{2}\bigg(
        \beta_{\nu(\ell+1)}\sum_{i=2^k-1}^{2^{k+1}-2^{z}-2^{z'+1}-1} \lk_{i,2^{k+1}-2-r}+\beta_{\nu(\ell+1)}\sum_{i=2^{k+1}-2^{z}-2^{z'+1}}^{2^{k+1}-2-\ell} \lk_{i,2^{k+1}-2-r}\\
    &\qquad\qquad\qquad-\beta_{\nu(r+1)}\sum_{i= 2^{k+1}-r - 1}^{2^{k+1}-2^{z}-2^{z'}-2} \lk_{i,2^{k+1}-2-\ell}-\beta_{\nu(r+1)}\sum_{i= 2^{k+1}-2^z-2^{z'} - 1}^{2^{k+1}-2^{z}-1} \lk_{i,2^{k+1}-2-\ell}\\
    &\qquad\qquad\qquad -\lk_{2^{k+1}-2-\ell,2^{k+1}-2-r}-\lk_{2^{k+1}-2-r,2^{k+1}-2-\ell}\bigg)\\
    &\qquad= \bigg[\frac{1}{2}\bigg(
        \beta_{\nu(\ell+1)}\sum_{i=2^{k+1}-2^{z}-2^{z'+1}}^{2^{k+1}-2-\ell} \lk_{i,2^{k+1}-2-r}-\beta_{\nu(r+1)}\sum_{i= 2^{k+1}-r - 1}^{2^{k+1}-2^{z}-2^{z'}-2} \lk_{i,2^{k+1}-2-\ell}\\
        &\qquad\qquad\qquad -\lk_{2^{k+1}-2-\ell,2^{k+1}-2-r}-\lk_{2^{k+1}-2-r,2^{k+1}-2-\ell}\bigg)\bigg]\\
        &\qquad\qquad\qquad+
    \left[\frac{\beta_{\nu(\ell+1)}}{2}\sum_{i=2^k-1}^{2^{k+1}-2^{z}-2^{z'+1}-1} \lk_{i,2^{k+1}-2-r}\right]-\left[\frac{\beta_{\nu(r+1)}}{2}\sum_{i= 2^{k+1}-2^z-2^{z'} - 1}^{2^{k+1}-2^{z}-1} \lk_{i,2^{k+1}-2-\ell}\right].
\end{align*}
We deal with the first square-bracketed term above. Applying \cref{lem:case_2_scaling} and the identity we get from the inductive hypothesis, we may simplify this term to
\begin{align*}
    &\frac{1}{2}\bigg(
        \beta_{\nu(\ell+1)}\sum_{i=2^{k+1}-2^{z}-2^{z'+1}}^{2^{k+1}-2-\ell} \lk_{i,2^{k+1}-2-r}-\beta_{\nu(r+1)}\sum_{i= 2^{k+1}-r - 1}^{2^{k+1}-2^{z}-2^{z'}-2} \lk_{i,2^{k+1}-2-\ell}\\
        &\qquad\qquad\qquad -\lk_{2^{k+1}-2-\ell,2^{k+1}-2-r}-\lk_{2^{k+1}-2-r,2^{k+1}-2-\ell}\bigg)\\
    &\qquad= \frac{(1+\sqrt{2})^{2(z-z'-1)}(\mu_{z'+1}-1)}{2(\mu_{z+1}-1)}\bigg(
        \beta_{\nu(\ell+1)}\sum_{i=2^{k+1}-2^{z'+1}}^{2^{k+1}-2-\ell'} \lk_{i,2^{k+1}-2-r'}-\beta_{\nu(r+1)}\sum_{i= 2^{k+1}-r' - 1}^{2^{k+1}-2^{z'}-2} \lk_{i,2^{k+1}-2-\ell'}\\
        &\qquad\qquad\qquad -\lk_{2^{k+1}-2-\ell',2^{k+1}-2-r'}-\lk_{2^{k+1}-2-r',2^{k+1}-2-\ell'}\bigg)\\
    &\qquad= \frac{(1+\sqrt{2})^{2(z-z'-1)}(\mu_{z'+1}-1)}{2(\mu_{z+1}-1)}\left(\beta_{\nu(r+1)}\lk_{2^{k+1}-2^{z'}-1,2^{k+1}-2-\ell'}\right).
\end{align*}

The third square-bracketed term is
\begin{align*}
    &\frac{\beta_{\nu(r+1)}}{2}\sum_{i= 2^{k+1}-2^z-2^{z'} - 1}^{2^{k+1}-2^{z}-1} \lk_{i,2^{k+1}-2-\ell}\\
    &\qquad = \frac{\beta_{\nu(r+1)}}{2}\left(\lk_{2^{k+1} - 2^z - 1, 2^{k+1} - 2 - \ell} + \lk_{2^{k+1} - 2^z - 2^{z'} - 1, 2^{k+1} - 2 - \ell}\right).
\end{align*}

There are two cases for the second term: either $z' = z - 1$ or $z' < z - 1$.

Suppose $z'=z-1$. Then, the second term is
\begin{align*}
    \frac{\beta_{\nu(\ell+1)}}{2}\sum_{i=2^k-1}^{2^{k+1}-2^{z+1}-1} \lk_{i,2^{k+1}-2-r} &= \frac{\beta_{\nu(\ell+1)}\beta_{\nu(r+1)}}{2(\mu_{z+1} - 1)}
\end{align*}
and $\lk_{2^{k+1} - 2^z - 1, 2^{k+1} - 2 - \ell}=0$.
Thus, summing up all of our terms gives
\begin{align*}
    &\frac{(1+\sqrt{2})^{2(z-z'-1)}(\mu_{z'+1}-1)}{2(\mu_{z+1}-1)}\left(\beta_{\nu(r+1)}\lk_{2^{k+1}-2^{z'}-1,2^{k+1}-2-\ell'}\right) + \frac{\beta_{\nu(\ell+1)}\beta_{\nu(r+1)}}{2(\mu_{z+1} - 1)}\\
    &\qquad\qquad - \frac{\beta_{\nu(r+1)}}{2}\left(\lk_{2^{k+1} - 2^z - 1, 2^{k+1} - 2 - \ell} + \lk_{2^{k+1} - 2^z - 2^{z'} - 1, 2^{k+1} - 2 - \ell}\right)\\
    &\qquad = \frac{\beta_{\nu(\ell+1)}\beta_{\nu(r+1)}}{2(\mu_{z+1} - 1)}.
\end{align*}

Now, suppose $z' < z - 1$.
Then, the second term is
\begin{align*}
    \frac{\beta_{\nu(\ell+1)}}{2}\sum_{i=2^k-1}^{2^{k+1}-2^{z}-2^{z'+1}-1} \lk_{i,2^{k+1}-2-r} &= \frac{\beta_{\nu(\ell+1)}\beta_{\nu(r+1)}}{4(\mu_{z+1} - 1)}\left(1 + (1+\sqrt{2})^{2(z-z'-1)}\right)
\end{align*}
and
\begin{align*}
    \lk_{2^{k+1} - 2^z - 1, 2^{k+1} - 2 - \ell} &= \frac{\beta_{\nu(\ell+1)}}{2(\mu_{z+1} - 1)}\left((1+\sqrt{2})^{2(z-z'-1)} - 1\right).
\end{align*}
Summing up all terms gives
\begin{align*}
    &\frac{(1+\sqrt{2})^{2(z-z'-1)}(\mu_{z'+1}-1)}{2(\mu_{z+1}-1)}\left(\beta_{\nu(r+1)}\lk_{2^{k+1}-2^{z'}-1,2^{k+1}-2-\ell'}\right)+ \frac{\beta_{\nu(\ell+1)}\beta_{\nu(r+1)}}{4(\mu_{z+1} - 1)}\left(1 + (1+\sqrt{2})^{2(z-z'-1)}\right)\\
    &\qquad\qquad - \frac{\beta_{\nu(r+1)}}{2}\left(\lk_{2^{k+1} - 2^z - 1, 2^{k+1} - 2 - \ell} + \lk_{2^{k+1} - 2^z - 2^{z'} - 1, 2^{k+1} - 2 - \ell}\right)\\
    &\qquad=\frac{\beta_{\nu(r+1)}}{2(\mu_{z+1} - 1)}\left(1 + (1+\sqrt{2})^{2(z-z'-1)}\right) - \frac{\beta_{\nu(r+1)}}{2}\lk_{2^{k+1} - 2^z - 1, 2^{k+1} - 2 - \ell}\\
    &\qquad=\frac{\beta_{\nu(\ell+1)}\beta_{\nu(r+1)}}{4(\mu_{z+1} - 1)}\left(1 + (1+\sqrt{2})^{2(z-z'-1)}\right) - \frac{\beta_{\nu(\ell+1)}\beta_{\nu(r+1)}}{4(\mu_{z+1} - 1)}\left((1+\sqrt{2})^{2(z-z'-1)} - 1\right)\\
    &\qquad =\frac{\beta_{\nu(\ell+1)}\beta_{\nu(r+1)}}{2(\mu_{z+1} - 1)}.\qedhere
\end{align*}

\end{proof} \end{document}